\documentclass[final,leqno]{siamltex}

\usepackage{graphicx}
\usepackage{amssymb}
\usepackage{amsmath}
\usepackage{xspace}
\usepackage{epstopdf}
\usepackage{color}
\usepackage{dsfont}
\usepackage{bbm}
 \usepackage[colorlinks=true,linkcolor=blue]{hyperref}

 \usepackage[colorlinks=true,linkcolor=blue]{hyperref}
\hypersetup{
    bookmarks=true,         
    unicode=false,          
    pdftoolbar=true,        
    pdfmenubar=true,        
    pdffitwindow=false,     
    pdfstartview={FitH},    
    pdftitle={My title},    
    pdfauthor={Author},     
    pdfsubject={Subject},   
    pdfcreator={Creator},   
    pdfproducer={Producer}, 
    pdfkeywords={keyword1} {key2} {key3}, 
    pdfnewwindow=true,      
    colorlinks=true,       
    linkcolor=magenta,          
    citecolor=green,        
    filecolor=magenta,      
    urlcolor=cyan           
}
\usepackage{hyperref}

\usepackage{color,calc,xcolor}

\newcommand{\nnnew}[1]{{\color{green} #1}}
\renewcommand{\nnnew}[1]{#1}

\usepackage{ifthen}
\newcommand{\isdraft}{\boolean{true}} 
\renewcommand{\isdraft}{\boolean{false}} 
\ifthenelse{\isdraft}{
    \usepackage[color]{showkeys}
    \usepackage{xcolor}
}{}
\newcommand{\markupdraft}[2]{
    \ifthenelse{\equal{#1}{display}}{#2}{}
    \ifthenelse{\equal{#1}{color}}{\color{#2}}{}
}
\newcommand{\nnotecolored}[3][]{\markupdraft{display}{{\color{#2}\noindent[{\bf Note #1}: #3]}}}
\newcommand{\mathnote}[1]{\markupdraft{display}{{\color{brown}\noindent[{\bf Math note}: #1]}}}
\newcommand{\generalnote}[1]{\markupdraft{display}{{\color{brown}\noindent[{\bf General note}: #1]}}}
\newcommand{\newcolored}[3][]{{\markupdraft{color}{#2}#3}
    \ifthenelse{\equal{#1}{}}{}{\markupdraft{display}{{\color{yellow!70!black}[#1]}}}} 

\providecommand{\del}[2][]{{\markupdraft{display}{{\color{red!20!yellow}[rmed: "#2"[#1]]}}}} 
\providecommand{\new}[2][]{\newcolored[#1]{blue}{#2}}
\providecommand{\nnew}[2][]{\newcolored[#1]{red}{#2}}
\providecommand{\rem}[2][]{\nnotecolored[#1]{green}{#2}} 
   
\providecommand{\nnote}[2][]{\nnotecolored[#1]{magenta}{#2}}   
\providecommand{\hide}[2][]{\nnotecolored[#1]{magenta}{#2}}

\providecommand{\smalltodo}[2][]{\markupdraft{display}{{\color{cyan}~\noindent== small todo: #2 {\color{yellow}(#1)} ==}}}
\providecommand{\todo}[2][]{\markupdraft{display}{{\color{red}\noindent++TODO: #2 {\color{yellow}(#1)}++}}}

\ifthenelse{\isdraft}{}{\renewcommand{\markupdraft}[2]{}}
\newcommand{\niko}[1]{\rem[Niko]{\color{brown}#1}}
\newcommand{\anne}[1]{\rem[Anne]{\color{orange}#1}}

{\bfseries}{\itshape}
\newtheorem{remark}{Remark}{\bfseries}{\itshape}

\renewcommand{\t}{t}
\renewcommand{\k}{t}

\newcommand{\F}{\mathcal{F}}
\newcommand{\FF}{\mathcal{F}}
\newcommand{\GG}{\mathcal{G}}
\newcommand{\G}{G}

\newcommand{\Y}{\mathbf{Y}}

\newcommand{\ZZ}{\mathcal{Z}}

\newcommand{\Monotone}{\ensuremath{\mathcal{M}}}
\newcommand{\mueff}{\mu_{\rm w}}
\newcommand{\Nplus}{\NNN_{>}}

\newcommand{\p}{\mathbf{p}}

\def\Tr{{\rm Tr}}

\def\CR{{\rm CR}}

\def\Var{{\rm Var}}

\newcommand{\SSel}{\mathcal{O}rd}
\newcommand{\OOrd}{\mathcal{O}rd}

\newcommand{\xzero}{\mathbf{x}_{0}}

\newcommand{\eq}{Eq.}
\newcommand{\onefifth}{\ensuremath{1\!/\!5}}

\newcommand{\xNES}{\ensuremath{\mathrm{xNES}}}
\newcommand{\dsa}{\ensuremath{\mathrm{CSAw/o}}}
\newcommand{\dsas}{\ensuremath{{\mathrm{CSAw/o}^{2}}}}
\newcommand{\sa}{\ensuremath{{\mathrm{SA}}}}

\newcommand{\commaES}{\ensuremath{(\mu \slash \mu_w,\lambda)}\xspace}
\newcommand{\plusES}{\ensuremath{(1+1)}\xspace}

\renewcommand{\dim}{n}
\newcommand{\Id}{\mathbf{I}_n}

\newcommand{\B}{\mathcal{B}}
\newcommand{\R}{\mathbb{R}}
\newcommand{\N}{\ensuremath{\mathcal{N}}}

\newcommand{\NNN}{\ensuremath{{\mathbb{{N}}}}}

\newcommand{\X}{\ensuremath{\mathbf{X}}}
\newcommand{\Z}{\ensuremath{\mathbf{Z}}}
\newcommand{\s}{\ensuremath{\mathbf{s}}}
\newcommand{\x}{\ensuremath{\mathbf{x}}}
\newcommand{\uu}{\ensuremath{\mathbf{u}}}

\newcommand{\y}{\ensuremath{\mathbf{y}}}

\newcommand{\z}{\ensuremath{\mathbf{z}}}

\newcommand{\UUU}{\ensuremath{\mathbb{U}}}
\newcommand{\Uspace}{\ensuremath{\mathbb{U}^{p}}}
\newcommand{\U}{\ensuremath{\mathbf{U}}}
\newcommand{\A}{\ensuremath{\mathbf{A}}}
\newcommand{\Ut}{\ensuremath{\mathbf{U}_{\k}}}
\newcommand{\Utt}{\ensuremath{\mathbf{U}_{\k+1}}}
\newcommand{\Zt}{\ensuremath{\mathbf{Z}_{\k}}}
\newcommand{\Ztt}{\ensuremath{\mathbf{Z}_{\k+1}}}
\newcommand{\Xt}{\ensuremath{\mathbf{X}_\k}}
\newcommand{\Xtt}{\ensuremath{\mathbf{X}_{\k+1}}}

\newcommand{\Ytt}{\ensuremath{\mathbf{Y}_{\k+1}}}
\newcommand{\st}{\ensuremath{\sigma_\k}}
\newcommand{\stt}{\ensuremath{\sigma_{\k+1}}}

\newcommand{\etastar}{\ensuremath{\eta^{\star}}}

\newcommand{\factonefifth}{\gamma}

\newcommand{\xstar}{\ensuremath{\mathbf{x}^{\star}}}

\newcommand{\Rplus}{\R^{+}}
\newcommand{\Rplusstar}{\R^{+}_{>}}
\newcommand{\LRm}{\kappa_{m}}
\newcommand{\LRsigma}{\kappa_{\sigma}}
\newcommand{\ptarget}{p_{\rm target}}
\newcommand{\Normal}{\mathcal{N}(0,\Id)}

\newcommand{\logn}{{\rm Logn}}

\newcommand{\cp}{comparison-based}
\newcommand{\rs}{randomized search}
\newcommand{\cprs}{\cp\ step-size adaptive \rs\ } 
\newcommand{\cprss}{\cp\ step-size adaptive \rs}

\newcommand{\acprs}{CB-SARS}
\newcommand{\asars}{SARS}

\newcommand{\scaleSI}{\rho}
\newcommand{\Sol}{\mathcal{S}ol}
\newcommand{\Perm}{\varsigma}

\newcommand{\StateSpace}{\Omega}
\newcommand{\Isom}{\rm Homo}

\newcommand{\SymGroup}{\mathcal{S}}

\title{{\large Linear Convergence of Comparison-based Step-size Adaptive Randomized Search via Stability of Markov Chains\\}
\todo{
}}
\author{Anne Auger\thanks{Inria, LRI, B\^at 660, University Paris Sud, 91405, Orsay, France ({\tt first.lastname\_at\_inria.fr}).} \and Nikolaus Hansen$^*$}

\begin{document}
\maketitle

%


\generalnote{From a discussion of the 16th of July 2013: about the fact that we cannot prove the sign of the convergence rate in the comma case: (Niko) in practice for reasonable algorithms convergence occurs, it does occur at a rate related to the highest curvature of the function, that will determine the target step-size that allow progress. Hence the higher curvature, the smaller convergence rate. Still we expect that on some functions the observed convergence rate depends on which ``curvature-valey'' we follow (see general plot of scaling-invariant function). In terms of stationary distribution of the normalized Markov chain, it means that it has several modes associated to the different regions where the curvature changes. Geometric ergodicity does not tell us anything about the mixing time of the chain, or how long it takes to go from one mode to the next one.\\
The claim of Niko concerning how the convergence rate depends on the curvature could be backup by some experiments.
On the fact that the self-adaptive ES with recombination can diverge: artefact of the recombination or bias because of individual step-size but does not mean that reasonable comma-ES will diverge on some scaling-invariant function.

Asymptotic Behavior of a Markovian Stochastic Algorithm with Constant Step
Jean-Claude Fort and Gilles Pagès
SIAM Journal on Control and Optimization 1999, Vol. 37, No. 5, pp. 1456-1482

}

\begin{abstract}
In this paper, we consider  \emph{comparison-based} adaptive stochastic algorithms for solving numerical optimisation problems. We consider a specific subclass of algorithms \new{that} we call \cprs (CB-SARS), where the state variables at a given iteration are a vector of the search space and a positive parameter, the step-size, typically controlling the overall standard deviation of the underlying search distribution.

We investigate the \emph{linear} convergence of CB-SARS on
\emph{scaling-invariant} objective functions. Scaling-invariant
functions preserve the ordering of points with respect to their function
value when the points are scaled with the same positive parameter (the
scaling is done w.r.t.\ a fixed reference point). This class of
functions includes norms composed with strictly increasing functions as
well as many \emph{non quasi-convex} and \emph{non-continuous}
functions. On scaling-invariant functions, we show the existence of a
homogeneous Markov chain, as a consequence of natural invariance
properties of CB-SARS (essentially scale-invariance and invariance to
strictly increasing transformation of the objective function). We then
derive sufficient conditions for \emph{global linear convergence} of
CB-SARS, expressed in terms of different stability conditions of the
normalised homogeneous Markov chain (irreducibility, positivity, Harris
recurrence, geometric ergodicity) and thus define a general methodology
for proving global linear convergence of CB-SARS algorithms on
scaling-invariant functions. As a by-product we provide a
connexion between  \emph{comparison-based} adaptive stochastic
algorithms and Markov chain Monte Carlo algorithms.

\generalnote{When function evaluations are noiseless, DFO methods can ac
hieve the same rates of convergence
as noiseless gradient methods up to a small factor depending
on a low-order polynomial of the di-
mension [9, 5, 10]. This leads one to wonder if the same equiva
lence can be extended to the case
when function evaluations and gradients are noisy - ref: http://arxiv.org/pdf/1209.2434v1.pdf}
\end{abstract}

\begin{keywords} 
stochastic algorithms, numerical optimisation, Markov chains, Markov chain Monte Carlo, comparison-based, linear convergence, invariance, adaptive randomized search, adaptive algorithms, derivative-free optimization
\end{keywords}


\pagestyle{myheadings}
\thispagestyle{plain}
\markboth{A. AUGER AND N. HANSEN}{LINEAR CONVERGENCE OF CB-SARS VIA STABILITY OF MARKOV CHAINS}


\nnote{Start of big revisions for SIAM final version: 3588, August 2015. Correction of definition of scaling invariant functions still inside, as well as new proof for construction of normalized Markov chain.}

\section{Introduction}

We consider the problem of minimizing an objective function $f: \R^{n} \to \R$ where the search cost is defined as the number of calls to the function $f$. We investigate \emph{comparison-based} search algorithms that use the $f$-values only through \emph{comparisons}. Because the $f$-values are totally ordered, from pair-wise comparisons a ranking of $f$-values can be derived and we can equivalently refer to our scenario as \emph{comparison-} or \emph{ranking-based}. In allusion to the term \emph{derivative-free optimization}, we might speak of \emph{function-value-free optimization} in this case. 
Well-known derivative-free methods are comparison-based algorithms, for instance pattern searches methods \cite{Hooke:Jeeves29,torczon1997convergence,audet2002analysis} and the simplex method by Nelder and Mead \cite{NelderMead:65,mckinnon1998convergence} and we believe that their success is to some extent due to their comparison-based property.

From the fact that the methods only use the comparison information
follows invariance of the algorithms to composing the objective function
(to the left) by a strictly increasing function $g: \R \to \R$. This
invariance property provides robustness because an error on the
objective function value--that can stem from various sources of
noise--has an impact only if it changes the result of a comparison,
i.e., if it changes the $f$-\emph{ordering} of the candidate solutions
under consideration. This invariance provides robustness also in that
very small or very large $f$-values can only have a limited impact. The
invariance also facilitates predictability, because the sequence of
solutions generated on $f$ and on $g\circ f$ are indistinguishable. Naturally, comparison-based algorithms have a wider range of applicability than derivative-free algorithms as they can be used in the absence of a numerical objective function value, for instance in the case where a user would provide relative preferences to the algorithm \cite{lewis1996rank}. At
the same time, invariance to strictly increasing transformations arguably makes convergence proofs harder to tackle, as
one has a weaker control on the objective function decrease.

In this context, this paper investigates the \emph{linear convergence} of a class of \emph{adaptive stochastic} comparison-based algorithms, namely comparison-based (CB) \emph{step-size adaptive} randomised search (SARS), abbreviated as \acprs. 
Formally, a SARS is a stochastically recursive sequence on the state space $\StateSpace =  \R^{n} \times \Rplusstar$. Given $(\X_{0},\sigma_{0}) \in \R^{n} \times \Rplusstar$,  the sequence is iteratively defined as
\begin{equation}\label{eq:FF}
(\Xtt,\stt) = \FF((\Xt,\st),\Utt)
\end{equation}
where $\Xt \in \R^{n}$ represents the favorite or incumbent solution at
iteration $\t$, $\st \in \Rplus_{>}$ is the so-called step-size, $\FF$
is a measurable function and $(\Ut)_{t \in \Nplus}$ is an independent
identically distributed (i.i.d.) sequence of random vectors. Often, the
step-size $\st$ represents the overall standard-deviation of an
underlying sampling distribution. Its proper control is crucial to
obtain linear convergence (a constant step-size gives a sub linear
convergence rate). The objective function $f$ must be available to the
\emph{transition function} $\FF$. While for \asars, the transition
function can use the $f$-\emph{values} of candidate solutions, the
transition function of \acprs\ uses only $f$-comparisons. A formal
definition will be given in Definition~\ref{def:SSAES}. In practice, in
addition to the adaptation of the scaling via the step-size, the
geometric shape of the underlying sampling distribution should be
adapted so as to properly solve ill-conditioned problems. If the
 sampling distribution is a multivariate normal
distribution, this can be done by adapting the covariance matrix as in
CMA-ES \cite{hansen2001}, the state-of-the-art randomized method for
continuous optimization. The methods investigated in this paper cover
thus some simplified version of CMA-ES.

Invariance to strictly increasing transformations of $f$ implies affine
\emph{covariance} (i.e.\ to applying an affine transformation to the left
of $f$) \cite{deuflhard2011newton}. We investigate here methods that are
in addition \emph{scale-invariant}, a particular case of
affine-invariance in the search space or affine \emph{contravariance}
(i.e.\ to applying an affine transformation to the right of $f$)
\cite{deuflhard2011newton}. Scale-invariance corresponds to affine
invariance where the general linear transformation is restricted to an
homothety. It translates that the algorithm has no intrinsic notion of
scale.

Affine invariance is a key aspect of several famous optimization algorithms like Newton or Nelder Mead methods which is also exploited in some of their theoretical analysis \cite{deuflhard2011newton,lagarias1998convergence,lagarias2012convergence}. Similarly, scale-invariance is an essential feature of the algorithms investigated here that we exploit heavily in our analysis.

The definition via \eqref{eq:FF} is general and abstract, however, often, \asars\ and \acprs\ take a specific form where the connexion with gradient methods becomes clear while the methods are \emph{derivative} and even \emph{function-value} free. Indeed, the update of the incumbent solution generally writes
\begin{equation}\label{eq:update-1}
\Xtt= \Xt + \kappa \st \Ytt
\end{equation}
where $\Ytt$ is a combination of selected random directions that can be seen as an  approximation of a gradient direction and $\kappa$ is a learning rate. This connexion can be pushed further for some specific algorithms where $\theta_t=(\Xt,\st)$ encodes the mean vector and standard deviation of a Gaussian distribution and a joint optimization criterion  formulated on the manifold defined by the family of Gaussian distributions $P_\theta$ can be defined. Applying a gradient update step with respect to $\theta$ to this joint criterion and taking a Monte Carlo approximation of the gradient\footnote{The gradient is taken wrt the Fisher Information metric, it is also called natural gradient.} defines a \cprs whose update equations are given in \eqref{eq:commamean} and \eqref{eq:nnes2} \cite{akimoto2010bidirectional,ollivier2013information}. Note that the learning rate $\kappa$ (and $\LRm, \LRsigma$ in \eqref{eq:commamean} and \eqref{eq:nnes2}) corresponds to the step-size of the underlying \emph{stochastic approximation algorithm} (here we however reserve the step-size name for $\st$ unless explicitly specified).


Several random optimization methods akin to the update in
\eqref{eq:update-1} were recently studied. First Nesterov proved
complexity bounds for some gradient-free algorithms using oracles for
directional derivatives (that use Gaussian random directions)
\cite{Nesterov:2011}. Later on, Stich et al.\ analyzed the simple Random
Pursuit (RP) where $\Ytt$ is a random direction and $\st$ is the result
of a line-search in the $\Ytt$ direction \cite{Stich2015,stich:2013}.
Assuming exact or approximate line search, they prove the linear
convergence of RP for strongly convex functions. They experimentally
compared RP to an accelerated version of RP, to Nesterov's schemes and
to a classical CB-SARS \cite{Schumer:68, Rechenberg,Devroye:72}. The
accelerated RP and Nesterov's schemes need as input some constants
related to the function (i.e., they are not tested in a black-box
setting). Concluding their observations on the performance of the
CB-SARS, the authors emphasize ``\textit{that the performance of the
adaptive step-size ES scheme [the classical CB-SARS] is remarkable given
the fact that it does not need any function-specific parametrization. A
comparison to the RP shows that it needs four times fewer function
evaluations on functions $f_{2} - f_{4}$.}'' \cite{stich:2013}. The main
reason are the saved expenses due to the omitted line searchs.
The theoretical analysis in \cite{stich:2013} heavily relies on the
control of the $f$-decrease at each iteration with the assumption of
exact line search (or with a controlled error in the case of approximate
line search). We believe that the author's analysis however is difficult
to generalize to our context. We resort thus to a different approach
that can prove in particular the linear convergence of the \acprs\
algorithm experimented in the aforementioned paper (for which the
authors stress the remarkable performance) \cite{companion-oneplusone}.

The optimization of noisy functions with derivative-free optimization algorithms has been recently investigated in \cite{jamieson2012,Shamir13}.  Comparisons to methods having access to the gradient are discussed in particular in those latter references.

While the previously mentioned papers analyze the algorithms on strongly convex and convex functions, we consider here the class of \emph{scaling-invariant functions}, natural in the context of \emph{comparison-based} algorithms.
We call a function $f$ scaling-invariant with respect to $\xstar$ if for all $\scaleSI > 0$, $\x,\y \in \R^{n}$
 $$
\nnnew{   f(\xstar +\x ) \leq f(\xstar + \y) 
   \Leftrightarrow
   f( \xstar+ \scaleSI \x ) \leq f( \xstar+ \scaleSI \y ) \enspace .}
     $$
This class includes all norms and all functions that are the composition of norm functions by increasing transformations--having hence convex sublevel sets--but also non quasi-convex functions, i.e., functions with non-convex sublevel sets. Non-constant scaling-invariant functions admit neither strict local optima besides $\xstar$ nor plateaus.

We  prove that  if a \acprs\ is scale-invariant, then, on a
scaling-invariant function, the normalised process $(\Xt - \xstar)/\st$
is a homogeneous Markov chain. In addition, stability properties of
this Markov chain imply asymptotic linear convergence of the original
algorithm independently of the starting point. We then formulate
different \emph{sufficient conditions}--expressed as stability conditions
on $(\Xt - \xstar)/\st$--that induce \emph{global} linear convergence
almost surely and of the expected log-progress. We also formulate
conditions for proving that the empirical estimate of the convergence
rate converges geometrically to the theoretical one from which we can
deduce non-asymptotic bounds. We hence define a general methodology to
prove linear convergence of \acprs\ algorithms. Our methodology
generalizes  previous works, restricted to a specific \acprs\ on the
sphere function \cite{Bienvenue:2003,TCSAnne04}, to a broader class of
algorithms and a much broader class of functions. In a companion
manuscript, the methodology has been applied to another \cprs algorithm
\cite{companion-oneplusone}.

The rest of this paper is organized as follows. We define in Section~\ref{sec:algo} \acprs\ algorithms. In Section~\ref{sec:inv}, we formalize different invariance properties commonly associated to \acprs.
In Section~\ref{sec:ex} we present several examples of existing methods that follow our general definition of \acprs\ and study their invariance properties.  In Section~\ref{sec:scaling-inv} we define the class of scaling-invariant functions. In Section~\ref{sec:jointMC}, we prove that for certain translation and scale-invariant \acprs\ algorithms optimizing scaling-invariant functions, $(\Xt- \xstar)/\sigma_{t}$ is a homogeneous Markov chain. In Section~\ref{sec:whystability}, we give sufficient conditions to linear convergence expressed in terms of stability of the Markov chain exhibited in Section~\ref{sec:jointMC}. A discussion is \del{finally }provided in Section~\ref{sec:discussion}. \new{In an appendix we describe in more details several examples of \acprs\ and present numerical experiments on those \acprs\ compared with experiments on Nelder Mead and Random Pursuit.}

\paragraph{Notations and definitions}
The set of nonnegative real numbers is denoted $\Rplus$ and $\Rplus_{>}$ denotes elements of $\Rplus$ excluding $0$, $\NNN$ is the set of natural numbers including zero while $\Nplus$ excludes zero. The Euclidian norm of a vector $\x$ of $\R^{n}$ is denoted $\| \x \|$.
A Gaussian vector or multivariate normal distribution with mean vector $\mathbf{m}$ and covariance matrix $\mathbf{C}$ is denoted $\N(\mathbf{m},\mathbf{C})$. The identity matrix in $\R^{n \times n}$ is denoted $\Id$ such that a standard multivariate normal distribution, i.e.\ with mean vector zero and identity covariance matrix is denoted $\Normal$. The density of a standard multivariate normal distribution (in any dimension) is denoted $p_{\N}$. The set of strictly increasing functions $g$ from $\R$ to $\R$ or from a subset  $I \subset \R$ to $\R$ is denoted $\Monotone$.
The notation $\xstar$ will be used in particular to denote the global minimum of the functions whose convergence is investigated. Sometimes we assume without loss of generality that $\xstar=0$.

\todo{DO I NEED THAT: We shall assume that if solutions are sampled from a continuous distribution for instance multivariate standard distribution, the ranking of the solutions with respect to $f$ does not produce ties almost surely. Hence fitness with some flat regions are excluded.}

\section{Comparison Based Step-size Adaptive Randomized Search \new{(\acprs)}}

In this section, we present a formal definition of \acprs\ algorithms. We then define invariance properties generally associated to those algorithms and finish by giving several concrete examples of \acprs\ algorithms as well as analyzing their invariance properties.

\subsection{Algorithm Definitions}\label{sec:algo}
We consider a \asars\ as defined in \eqref{eq:FF} and consider that each vector $\Ut$ belongs to a space $\Uspace=\UUU \times \ldots \times \UUU$ and has $p$ coordinates $\Ut^{i}$ belonging to $\UUU$. 
The probability distribution of the vector $\Ut$ is denoted $p_{\U}$.  From the definition \eqref{eq:FF} follows that $\left( (\Xt,\st)\right)_{t \in \NNN}$ is a time homogeneous Markov chain. 
We call $\FF$ the \emph{transition function} of the algorithm. The objective function $f$ is also an input argument to the transition function $\FF$ as the update of $(\Xt,\st)$ depends on $f$, however we omit this dependence in general for the sake of simplicity in the notations. If there is an ambiguity we add the function $f$ as upper-script, i.e.\ $\FF^{f(\x)}$ or $\FF^{f}$.

A \acprs\ is a \asars\ where the transition function $\FF$ depends on
$f$ only through comparison of candidate solutions and is the
composition of several functions  that we specify in the sequel. The $p$
coordinates of $\Utt$ are in a first time used to create new candidate
solutions according to a $\Sol$ function:
\begin{equation*}
\Xtt^{i} = \Sol((\Xt,\st),\Utt^{i}) \,, i=1,\ldots,p \enspace.
\end{equation*}
(For instance in the case where $\UUU=\R^{n}$ the $\Sol$ function can equal $\Sol((\x,\sigma),\uu^{i}) = \x + \sigma \uu^{i}$.) The $p$ candidate solutions are then evaluated on $f$ and ordered according to their objective function value (We break possible ties by considering the first solution sampled when two solutions are equal). The permutation corresponding to the ordered objective function values $f(\Xtt^{i})$ is denoted $\Perm \in \mathfrak{S}(p)$ where we denote $\mathfrak{S}(p)$ the set of permutations of $p$ elements. Formally $\Perm$ is the output of the $\OOrd$ function defined below. It is then used to order the coordinates of the vector $\Utt$ accordingly. More formally the permutation acts on the coordinates of $\Utt$ via the following function:
\begin{equation}
\begin{aligned}
\mathfrak{S}(p) \times \Uspace  \to & \Uspace \\
(\Perm, \Utt) \mapsto & \Perm*\Utt = \left(\Utt^{\Perm(1)},\ldots,\Utt^{\Perm(p)} \right) 
\end{aligned}
\end{equation}
where the previous equation implicitly defines the operator $*$.

The update of $(\Xt,\st)$ is achieved using the current state $(\Xt,\st)$ and the ranked coordinates of $\Utt$. More precisely let us consider a measurable function $\GG$ called update function that maps $\Omega \times \Uspace$ onto $  \Omega $, the update of $(\Xt,\st)$ reads
\begin{equation}\label{eq:ss-G}
(\Xtt,\stt) = \GG((\Xt,\st), \Perm * \Utt) =  \GG((\Xt,\st),\Ytt) \enspace,
\end{equation}
where $\Ytt$ denotes the ordered coordinates of $\Utt$, i.e.\
\begin{equation}\label{eq:defY}
\Ytt=(\Utt^{\Perm(1)},\ldots,\Utt^{\Perm(p)}) \enspace.
\end{equation}
We formalize the definition of a \acprs\ below after introducing a definition for the function $\Sol$ for generating solutions as well as for the ordering function.
\begin{definition}[$\Sol$ function]\label{def:Sol} Given $\UUU$, the state space for the sampling coordinates of $\Ut$, a $\Sol$ function used to create candidate solutions is a measurable function mapping $\Omega \times \UUU$ into $\R^{n}$, i.e.\
$$
\Sol: \Omega \times \UUU \mapsto \R^{n} \enspace .
$$
\end{definition}
\noindent  We now define the ordering function that returns a permutation based on the objective function values.
\begin{definition}[$\OOrd$ function]\label{def:Ord} The ordering function $\OOrd$ maps $\R^{p}$ to $\mathfrak{S}(p)$, the set of permutations with $p$ elements and returns for any set of \new{indexed} real values $(f_{1},\ldots,f_{p})$ a permutation of ordered indexes. That is $\Perm = \OOrd(f_{1},\ldots,f_{p}) \in \mathfrak{S}(p)$ where
$$
f_{\Perm(1)} \leq \ldots \leq f_{\Perm(p)} \enspace.
$$
When more convenient we denote $\Perm$ as $\OOrd((f_{i})_{i=1,\ldots,p})$ instead of $\OOrd(f_{1},\ldots,f_{p}) $. When needed for the sake of clarity we use the notations $\OOrd^{f}$ or $\Perm^{f}$ to emphasize the dependency in $f$.
\end{definition}

We are now ready to give  a formal definition of a \cprss.
\begin{definition}[\acprs\ minimizing $f:\R^{\dim} \to \R$]\label{def:SSAES} Let $p \in \Nplus$ and $\Uspace=\UUU \times \ldots \times \UUU$ where $\UUU$ is a subset of $\R^{m}$. Let $p_{\U}$ be a probability distribution defined on $\Uspace$ where each $\U$ distributed according to $p_{\U}$ has a representation $(\U^{1},\ldots,\U^{p})$ (each $\U^{i} \in \UUU$). Let $\Sol$ be a solution function as in Definition~\ref{def:Sol}. Let $\GG_{1}:  \Omega \times \Uspace \mapsto  \R^{n} $ and $\GG_{2}: \Rplus \times \Uspace \mapsto  \Rplus$ be two mesurable mappings and let denote $\GG=(\GG_{1},\GG_{2})$.

A \acprs\ is determined by  the quadruplet $(\Sol,\GG,\Uspace,p_{\U})$ from which the recursive sequence $(\Xt,\st) \in \Omega$ is defined via $(\X_{0},\sigma_{0}) \in \Omega$ and for all $t$:
\begin{align}\label{eq:sol}
\Xtt^{i} & = \Sol((\Xt,\st),\Utt^{i}) \,, i=1,\ldots,p \\\label{eq:perm}
\Perm & = \SSel(f(\Xtt^{1}),\ldots,f(\Xtt^{p})) \in \mathfrak{S}(p)\\\label{eq:G1}
\Xtt & = \GG_{1}\left( (\Xt,\st), \Perm*\Utt \right)\\\label{eq:G2}
\stt & = \GG_{2}\left(\st,  \Perm*\Utt \right)  
\end{align}
where $(\Ut)_{t \in \Nplus}$ is an i.i.d.\ sequence of random vectors on \Uspace\ distributed according to $p_{\U}$, $\OOrd$ is the ordering function as in Definition~\ref{def:Ord}.
\end{definition}

The previous definition illustrates the \emph{function-value-free} property as we see that the update of the state $(\Xt,\st)$ is performed using solely the information given by the permutation that contains the order of the candidate solutions according to $f$.
For a \cprss, the function $\FF$ introduced in \eqref{eq:FF} is the composition of the update function $\GG$, the solution operator $\Sol$ and the ordering function, more precisely
\begin{equation}\label{eq:FvsG}
\boxed{
\FF^{f}((\x,\sigma),\uu) = \GG((\x,\sigma),\SSel(f(\Sol((\x,\sigma),\uu^{i}))_{i=1,\ldots,p}) * \uu ) \enspace.
}
\end{equation}
Note that for the update of the step-size (see \eqref{eq:G2}), we assume that $\Xt$ does not come into play.\niko{I wonder whether this is good (more specific) or bad (less general). I could imagine that this is a necessary consequence of scale-invariance.} Examples of \acprs\ are given in Section~\ref{sec:ex}. 



\subsection{Invariance Properties}\label{sec:inv}
Invariance is an important principle in science in general and in optimization. When an invariance property holds, convergence results that are true on a \emph{single} function generalize to a \emph{whole class} of functions.  Some invariances are taken for granted in optimization, like translation invariance, while others are less common or less recognized. In the sequel we start by formalizing that \acprs\ are invariant to strictly monotonic transformations of $f$. We focus then in Section~\ref{sec:ISS} on invariance in search space and  formalize translation invariance and scale invariance. We also derive sufficient conditions for a \acprs\ to be translation and scale-invariant. 

\subsubsection{Invariance to Strictly Monotonic Transformations of $f$}\label{sec:invmonot}
Invariance to strictly monotonic transformations of $f$ of a \acprs\ algorithm is a direct consequence of the algorithm definition.
It stems from the fact that the objective function only comes into play through the ranking of the solutions via the ordering function (see \eqref{eq:perm}, \eqref{eq:G1} and~\eqref{eq:G2}). This ordering function does output the same result on $f$ or any strictly monotonic transformation of $f$. More formally let us define $\Monotone_{I}$ the set of strictly increasing functions $g: I \to \R$, where $I$ is a subset of $\R$ i.e.\ if for all $x$ and $y$ in $I$ such that $x < y$ we have $g(x) < g(y)$ and define $\Monotone = \cup_{I} \Monotone_{I}$.
%
%
The elements of $\Monotone$ preserve the ordering. The invariance to composite of $f$ by a function in $\Monotone$ is stated in the following proposition.

\begin{proposition}\label{prop:inv-strict-monotone}[Invariance to strictly monotonic transformations]
Consider $(\Sol,\GG,\Uspace,p_{\U})$ a \acprs\ as defined in Definition~\ref{def:SSAES} optimizing $f: \R^{\dim} \to \R$ and let $(\Xt,\st)$ be the Markov chain sequence defined as $(\X_{0},\sigma_{0}) \in \R^{\dim} \times \Rplusstar$ and
$$
(\Xtt,\stt) = \GG((\Xt,\st),\Perm^{f}*\Utt)
$$
where $(\Ut)_{t \in \Nplus}$ is an i.i.d.\ sequence of random vectors on \Uspace\ distributed according to $p_{\U}$ and $\Perm^{f} = \OOrd(f(\Sol((\Xt,\st),\Utt^{i}))_{1 \leq i \leq p})$. Let $g: f(\R^{n}) \to \R$ (where $f(\R^{n})$ is the image of $f$) in $\Monotone$ be a strictly increasing function  and $(\Xt',\st')$ be the Markov chain obtained when optimizing $g \circ f$ using the same sequence $(\Ut)_{t \in \Nplus}$ and same initial state $({\X_{0}}',{\sigma_{0}}')=(\X_{0},\sigma_{0})$. Then \del{almost surely}for all $t$
$$
(\Xt,\st) = (\Xt',\st') \enspace.
$$
\end{proposition}

{\em Proof.}
The proof is immediate, by induction. Assume $(\Xt,\st) = (\Xt',\st')$ and let denote $\Xtt^{i} = \Sol((\Xt,\st),\Utt^{i})$. Because $\SSel(f(\Xtt^{1}),\ldots,f(\Xtt^{p})) = \SSel(g\circ f(\Xtt^{1}),\ldots,g\circ f(\Xtt^{p})) = \Perm$, then 
$$
(\Xtt,\stt)=\GG((\Xt,\st),\Perm* \Utt) = \GG((\Xt',\st'),\Perm* \Utt) = (\Xtt',\stt') \enspace. \eqno \endproof
$$
Consequently, on the three functions depicted in Figure~\ref{fig:monotone}, a \cprs will produce the same sequence $(\Xt,\st)_{t \in \NNN}$. Hence if convergence takes place on one of those functions, it will take place on the two others and at the same  convergence rate. This invariance property is shared by pattern search and the Nelder-Mead methods. A particular case of strictly increasing functions are affine functions: $ x \in \R \to \alpha x + \beta $ with $\alpha > 0$. Thus a consequence of the previous proposition is that \acprs\ are \emph{affine covariant}.

\begin{figure}
\centering
\includegraphics[width=0.32\textwidth]{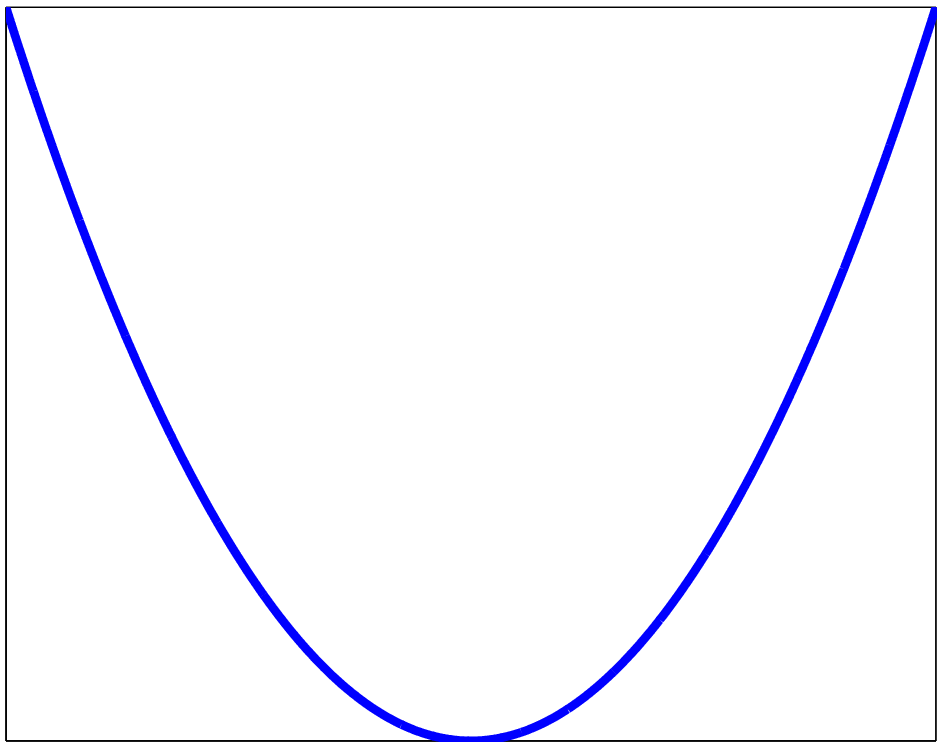}
\includegraphics[width=0.32\textwidth]{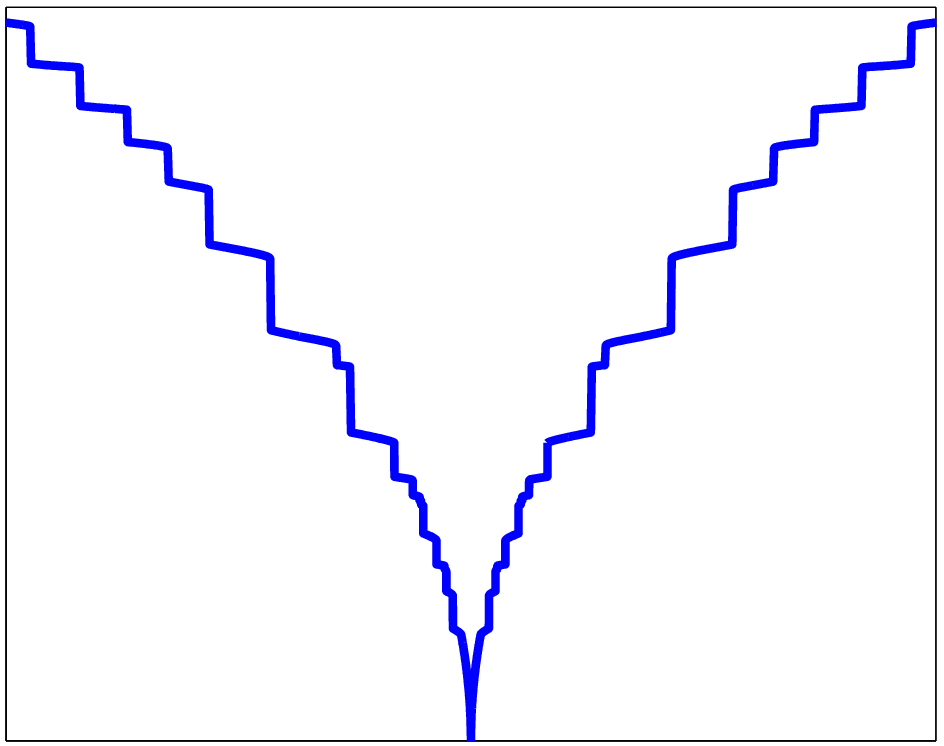}
\includegraphics[width=0.32\textwidth]{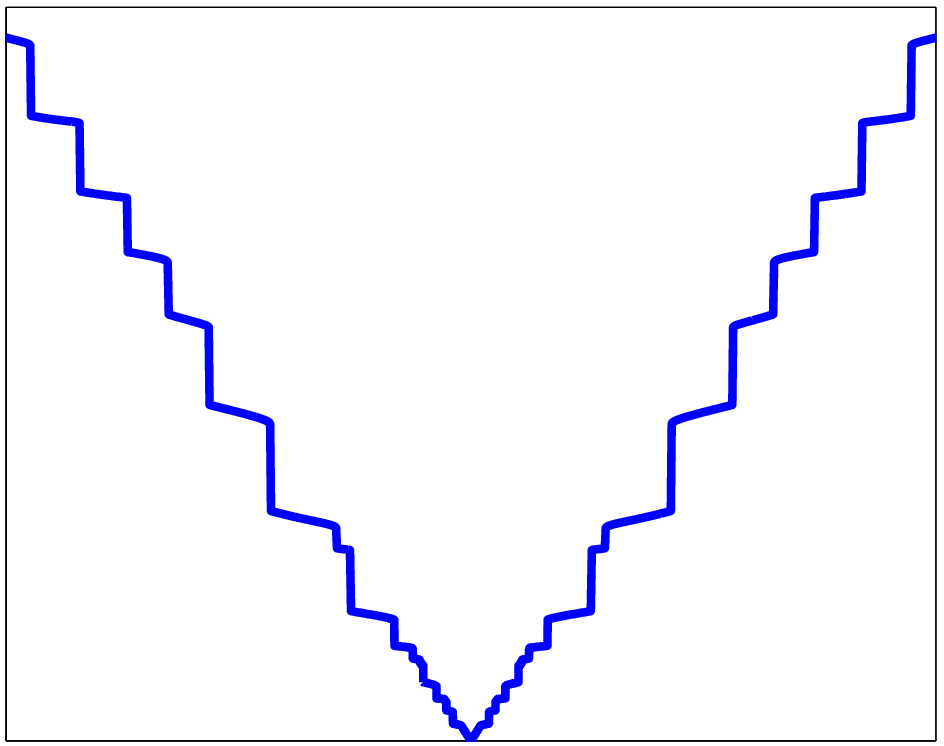}
\caption{\label{fig:monotone} Illustration of invariance to strictly increasing transformations. Representation of three instances of functions belonging to the invariance (w.r.t.\ strictly increasing transformations) class of $f(\x) = \| \x \|^{2}$ in dimension $1$. On the left the sphere function and middle and right functions $g \circ f$ for two different $g \in \Monotone$. On these three functions, a \cprs will generate the same sequence $(\Xt,\st)$ (see Proposition~\ref{prop:inv-strict-monotone}) and consequently convergence will take place at the same rate.}
\end{figure}

\subsubsection{Invariances in the Search Space: Translation and Scale-Invariance}\label{sec:ISS}
We consider now invariance of \cprs related to transformations in the search space. We use a classical approach to formalize invariance using homomorphisms transforming state variables via  a group action and visualize invariances with a commutative diagram \cite{grizzle1985structure,martin2004invariant}. We start by translation invariance, usually taken for granted in optimization. 

\paragraph*{Translation invariance}\label{sec:TI}
Most optimization algorithms are  translation invariant, which implies the same performance when optimizing $\x\mapsto f(\x-\xzero)$ for all $\xzero$ provided that a respective initialization of the algorithm is taking place. More precisely, let us consider $\R^{n}$ endowed with the addition operation $+$ as a group and consider $\SymGroup(\Omega)$ the symmetric group on $\Omega$, i.e the set of all bijective functions from $\Omega$ to itself (endowed with the composition $\circ$, it yields a group structure).
We remind the definition of a group homomorphism.
\begin{definition}[Group homomorphism]Let $(G_1, .)$ and $(G_2, *)$ be two groups. A mapping $\Phi: G_1 \to G_2$ is called group homomorphism if for all $x, y\in G_1$ we have $\Phi(x . y) = \Phi(x) * \Phi(y)$.
\generalnote{Terminology Niko (not validated by Cyril), i.e.\, if $\Phi$ commutes with the group operations}  \nnote{Where did you see that the commutation means that?}\niko{I didn't "see" this anywhere, it follows from my understand what "commute" means, namely that two "things" (e.g.\ two operators) commute if we can change their order without changing the result, i.e.\ their composition is commutative, see e.g.\ \url{http://mathworld.wolfram.com/Commute.html} considering a matrix to be an operator, or \href{http://chemwiki.ucdavis.edu/Physical_Chemistry/Quantum_Mechanics/02._Fundamental_Concepts_of_Quantum_Mechanics/Heisenberg's_Uncertainty_Principle/Commuting_Operators}{here}.}
\end{definition}

From the definition follows that for any $x \in G_{1}$, $\Phi(x^{-1})=[\Phi(x)]^{-1}$ where $x^{-1}$ (resp.$[\Phi(x)]^{-1}$)  denotes the inverse of $x$ (resp.\ of $[\Phi(x)]$). Note that in case $x$ belongs to an additive group, the inverse is denoted $-x$.
Let $\Isom( (\R^{n},+), (\SymGroup(\Omega), \circ) )$ be the set\niko{family seems a better word here, even though I don't know the exact meaning of it (I did check on wikipedia)}\anne{From this discussion between the difference between set and family I would keep set - as I do not see an obvious natural indexing of the elements of the group morphism\url{http://math.stackexchange.com/questions/35462/what-is-the-difference-between-family-and-set}}\niko{Indexing means parametrization, right? (Like, e.g., for the family of normal distributions). I would be a little surprised if our morphisms wouldn't be highly structured, which seem to be the point to call it a family, but it's not important anyway. }\anne{I had in mind for indexing, more a countable set. But you are right that we talk about the family of multivariate normal distributions ...} of group homomorphisms from $ (\R^{n},+)$ to $(\SymGroup(\Omega), \circ)$. For instance, consider $\Phi \in \Isom( (\R^{n},+), (\SymGroup(\Omega), \circ) )$, i.e.\ $\Phi: \y \in (\R^{n},+) \mapsto \Phi(\y)$ where $\Phi(\y)$ is a state space transformation such that for all  $(\x,\sigma) \in \R^{n} \times \Rplusstar$, $\Phi(\y)(\x,\sigma) = (\x + \y,\sigma)$.\del{ We also define $\mathcal{T}$ the translation in $\R^{n}$ which can also be seen as a group homomorphism: $\mathcal{T}: \y \in (\R^{n},+) \mapsto \mathcal{T}(\y) \in (\R^{n},+)$ with $\mathcal{T}(\y): \x \in \R^{n} \mapsto \x + \y \in \R^{n} $.} We are now ready to state a definition of translation invariance.
\mathnote{We had previously isomorphism instead of homomorphism - A group isomorphism is a bijective group homomorphism. Here we are not in the situation of a bijective group homomorphism as this would for instance imply that $\SymGroup(\Omega)$ and $(\R^{n},+)$ are isomorphms. i.e.\ it's not true for instance that the morphism $\Phi$ given above is surjective: consider the state space transformation $(\x,\sigma) \mapsto (\x, 2 \sigma)$ it has no antecedent by $\Phi$. Remind also that from the group homomorphism property (let us call the group homomorphism $h$) it holds that, the neutral element of the first group is send to the neutral element of the second group and that  $h(u^{-1}) = h(u)^{-1}$ }
\begin{definition}[Translation Invariance] A \asars\ with transition function $\FF$ is translation invariant if there exists a group homomorphism $\Phi \in \Isom((\R^{n},+), (\SymGroup(\Omega), \circ))$ such that for any objective function $f$, for any $\xzero \in \R^{n}$,  for any $(\x,\sigma) \in \Omega$ and for any $\uu \in \Uspace$
\begin{equation}\label{eq:TI}
\FF^{f(\x)}((\x,\sigma),\uu) = \underbrace{\Phi(\xzero)^{-1}}_{\Phi(-\xzero)} \left( \FF^{f\left( \x - \xzero \right)}(\Phi(\xzero)(\x,\sigma),\uu) \right)
\enspace,
\end{equation}
\new{or equivalently
\begin{equation}
\Phi(\xzero)\left(\FF^{f(\x)}((\x,\sigma),\uu)\right) =  \FF^{f\left( \x - \xzero \right)}(\Phi(\xzero)(\x,\sigma),\uu) 
\enspace,
\end{equation}}%
where the function to be optimized is shown as upper-script of the transition function $\FF$. 
\end{definition}
The previous definition means that a \asars\ algorithm is translation invariant, if we can find an homomorphism $\Phi$ (that depends on the algorithm) that defines for any translation $\xzero$, a search space transformation  $\Phi(\xzero)$, such that we can obtain $(\Xtt,\stt)$ from $(\Xt,\st)$ in two ways: (i) we apply one iteration of the algorithm to optimize $f(\x)$ or (ii) we apply the state space transformation $\Phi(\xzero)$ to the state of the algorithm, apply one iteration of the algorithm on $\x \mapsto f(\x - \xzero )$ and transform back the state of the algorithm via $\Phi(-\xzero)$. 
This property is pictured via a double-commutative diagram (see Figure~\ref{fig:commute}).

\begin{figure}
\unitlength0.56mm
\vspace{6mm}\centering
\begin{picture}(100,50)
\put(-1,50){$(\X_{t},\st)$}
\put(22,52){\vector(1,0){45}}
\put(33,56){$\FF^{f(\x)}$}
\put(69,50){$(\Xtt,\stt)$}
\put(-1,0){$(\X'_{t}, \st')$}
\put(22,2){\vector(1,0){45}}
\put(33,6){$\FF^{f(\x - \xzero)} $}
\put(69,0){$(\Xtt',\stt')$}
\put(8,45){\vector(0,-1){35}}
\put(10,10){\vector(0,1){35}}
\put(-9,25){$\Phi({\xzero})$}
\put(12,25){$\Phi({-\xzero})$}
\put(87,10){\vector(0,1){35}}
\put(85,45){\vector(0,-1){35}}
\put(68,25){$\Phi({\xzero})$}
\put(88,25){$\Phi(-\xzero)$}
\end{picture}
\begin{picture}(100,50)
\put(12,50){$(\X_{t},\st)$}
\put(35,52){\vector(1,0){45}}
\put(46,56){$\FF^{f(\x)}$}
\put(82,50){$(\Xtt,\stt)$}
\put(12,0){$(\X'_{t}, \st')$}
\put(35,2){\vector(1,0){45}}
\put(48,6){$\FF^{f(\alpha \x)} $}
\put(82,0){$(\Xtt',\stt')$}
\put(22,45){\vector(0,-1){35}}
\put(24,10){\vector(0,1){35}}
\put(8,25){$\Phi({\alpha})$}
\put(26,25){$\Phi({\frac{1}{\alpha}})$}
\put(99,10){\vector(0,1){35}}
\put(97,45){\vector(0,-1){35}}
\put(101,25){$\Phi(\frac{1}{\alpha})$}
\put(83,25){$\Phi({\alpha})$}
\end{picture}
\caption{\label{fig:commute} Left: Commutative diagram for the translation invariance property applied to one iteration of a step-size adaptive algorithm ($\Phi(-\x_{0}) = [\Phi(\x_{0})]^{-1}$). Right: Commutative diagram for the scale-invariance property applied to one iteration of a step-size adaptive algorithm ($\Phi(1/\alpha) = [\Phi(\alpha)]^{-1}$). The homomorphisms $\Phi$ (different on the left and right) define for any $\x_{0}$ (resp. $\alpha$) a search space transformation $\Phi(\x_{0})$ (resp. $\Phi(\alpha)$).}
\end{figure}

We consider in the next proposition some specific properties of $\Sol$ and $\GG$ that render a \cprs translation invariant. These properties are satisfied for algorithms presented in Section~\ref{sec:ex}.
\begin{proposition}\label{prop:TI}
Let $(\Sol,\GG,\Uspace,p_{\U})$ be a \acprs\ according to Definition~\ref{def:SSAES}. If the following conditions are satisfied:\\ (i) 
for all $\x, \xzero \in \R^{n}$ for all $\sigma > 0$, for all $\uu^{i} \in \UUU$
\begin{equation}\label{eq:condTISol}
\Sol(( \x+\xzero,\sigma),\uu^{i}) = \Sol((\x,\sigma),\uu^{i}) + \xzero
\end{equation}
(ii)
for all $\x, \xzero \in \R^{n}$ for all $\sigma > 0$, for all $\y \in \Uspace$
\begin{equation}\label{eq:condTIGG}
\GG_{1}((\x+\xzero,\sigma),\y)
= 
\GG_{1}((\x,\sigma),\y) + \xzero
\end{equation}
then $(\Sol,\GG,\Uspace,p_{\U})$ is translation invariant and the associated group homomorphism $\Phi$ is defined by
\begin{equation}\label{eq:momoTI}
\Phi(\xzero)(\x,\sigma) = (\x+\xzero,\sigma) \mbox{ for all } \xzero, \x, \sigma \enspace.
\end{equation}
In addition, assuming that the $\Sol$ function satisfies property \eqref{eq:condTISol}, then if $(\Sol,\GG,\Uspace,p_{\U})$ is translation invariant with \eqref{eq:momoTI} as homomorphism, then \eqref{eq:condTIGG} is satisfied.
\end{proposition}
\begin{proof}
Consider the homomorphism defined in \eqref{eq:momoTI}, then \eqref{eq:condTIGG} writes
\begin{equation}\label{eq:tototo}
\GG_1(\Phi(\xzero)(\x,\sigma),\y) = \Phi(\xzero) \left( \GG_1((\x,\sigma),\y) \right) \enspace,
\end{equation}
and \eqref{eq:condTISol} writes
$
\Sol(\Phi(\xzero)(\x,\sigma),\uu^{i}) - \xzero = \Sol((\x,\sigma),\uu^{i}) 
$.
This latter equation implies that the same permutation $\Perm$ will result from ordering solutions generated by the $\Sol$ function on $f$ from $(\x,\sigma) $ or on $f(\x - \xzero)$ starting from $\Phi(\xzero)(\x,\sigma)$. Using \eqref{eq:tototo} we hence have
$
\GG(\Phi(\xzero)(\x,\sigma), \Perm^{f(\x-\xzero)}_{\Phi(\xzero)(\x,\sigma)} * \uu)
= \Phi(\xzero) \left( \GG((\x,\sigma), \Perm^{f}_{(\x,\sigma)}*\uu) \right)
$
which turns out to coincide with \eqref{eq:TI}. The inverse is immediate. 
\end{proof}

\paragraph*{Scale-invariance property}\label{sec:SI}

Scale-invariance is a particular case of affine invariance in the search space where we consider transformation of a function $\x \mapsto f(\x)$ into $\x \mapsto f(\alpha \x) $ for $\alpha > 0$.
The scale invariance property translates \del{the fact }that the algorithm has no intrinsic notion of scale. It can be defined similarly to translation invariance by considering the set of group homomorphisms from the group $(\Rplusstar,.)$ (where $.$ denotes the multiplication between two real numbers) to the group $(\SymGroup(\Omega),\circ)$. We denote this set $\Isom((\Rplusstar,.),(\SymGroup(\Omega),\circ))$.

\del{Formally, consider $\Rplus_{>}$ endowed with the multiplication as a group and define the group homomorphism $\Phi: \alpha \in \Rplus_{>} \mapsto \Phi(\alpha) $ as: for all $(\x,\sigma) \in \R^{n} \times \Rplus_{>} $, $\Phi(\alpha) (\x,\sigma) = (\x/\alpha,  \sigma / \alpha)$. }
\begin{definition}[Scale-invariance]
A \asars\ with transition function $\FF$ is \emph{scale-invariant} if there exists an homomorphism $\Phi \in \Isom((\Rplusstar,.),(\SymGroup(\Omega),\circ))$ such that for any $f$, for any $\alpha > 0$, for any $(\x,\sigma) \in \R^{n} \times \Rplusstar$ and for any $\uu \in \Uspace$
\begin{equation}\label{eq:SI}
\FF^{f(\x)}((\x,\sigma),\uu) = \Phi(1/\alpha) \left( \FF^{f\left(\alpha \x \right)}(\Phi(\alpha)(\x,\sigma),\uu) \right)
\enspace,
\end{equation}
where the function optimized is shown as upper-script of the transition function $\FF$. 
\end{definition}
In the previous definition we have used the fact that for any element $\alpha$ of the multiplicative group $(\Rplusstar,.)$ its inverse is $1/\alpha$.
The scale-invariance property can be pictured via a double-commutative diagram (see Figure~\ref{fig:commute}).

We derive in the next proposition some conditions for a \acprs\ to be scale-invariant that will be useful in the sequel to prove that the algorithms presented in Section~\ref{sec:ex} are scale-invariant.
\begin{proposition}\label{lem:scaleinvariant}
Let $(\Sol,\GG,\Uspace,p_{\U})$ be a \acprs\ according to Definition~\ref{def:SSAES}. If for all $\alpha > 0$ the following three conditions are satisfied: (i) for all $\uu^{i} \in \UUU, (\x,\sigma) \in \R^{n} \times \Rplusstar$, 
\begin{equation}\label{eq:SIsol}
\Sol((\x,\sigma),\uu^{i}) = \alpha \Sol \left( \left(\frac{\x}{\alpha}, \frac{\sigma}{\alpha}\right), \uu^{i}  \right)
\end{equation}
(ii) for all $\y \in \Uspace, (\x,\sigma) \in \R^{n} \times \Rplusstar$
\begin{equation}\label{eq:SIG1}
\GG_{1}( (\x,\sigma),\y) = \alpha \GG_{1}\left(\left(\frac{\x}{\alpha},\frac{\sigma}{\alpha}\right),\y\right)
\end{equation}
and (iii) for all $\y \in \Uspace,\sigma \in \Rplusstar$
\begin{equation}\label{eq:SIG2}
\GG_{2}(\sigma,\y)= \alpha \GG_{2}\left(\frac{\sigma}{\alpha},\y\right)\enspace,
\end{equation}
then it is scale invariant and the associated homomorphism is $\Phi: \alpha \in \Rplus_{>} \mapsto \Phi(\alpha) $ where for all $(\x,\sigma) \in \R^{n} \times \Rplus_{>} $, 
\begin{equation}\label{eq:momoSI}
\Phi(\alpha) (\x,\sigma) = (\x/\alpha,  \sigma / \alpha) \enspace.
\end{equation}
Inversely, assuming that the $\Sol$ function satisfies \eqref{eq:SIsol}, if $(\Sol,\GG,\Uspace,p_{\U})$ is scale-invariant with the homomorphism defined in \eqref{eq:momoSI}, then \eqref{eq:SIG1} and \eqref{eq:SIG2} are satisfied.
\end{proposition}
\begin{proof}
From (i) we deduce that for any $(\x, \sigma)$ in $\R^{n} \times \Rplusstar$ and any $\uu^{i} \in \UUU$,
$$
f(\Sol((\x,\sigma),\uu^{i})) = f \left( 
\alpha \Sol \left( \left(\frac{\x}{\alpha}, \frac{\sigma}{\alpha}\right), \uu^{i}  \right)
\right)
$$
which implies that the same permutation $\Perm$ will result from ordering solutions (with $\OOrd$) on $f$ starting from $(\x,\sigma)$ or on $f(\alpha \x)$ starting from $(\x/\alpha,\sigma/\alpha)$, i.e.\ with some obvious notations
$\Perm_{(\x,\sigma)}^{f(\x)} = \Perm_{(\frac{\x}{\alpha}, \frac{\sigma}{\alpha})}^{f(\alpha \x)}$.
On the other hand using \eqref{eq:FvsG} the following holds
\begin{align}
\F^{f(\x)}((\x,\sigma),\uu) & = \GG((\x,\sigma),\Perm_{(\x,\sigma)}^{f(\x)}*\uu) \\
\F^{f(\alpha \x)} \left( \left( \frac{\x}{\alpha},\frac{\sigma}{\alpha} \right), \uu \right) & = \GG \left( \left( \frac{\x}{\alpha},\frac{\sigma}{\alpha} \right) , \Perm_{(\frac{\x}{\alpha}, \frac{\sigma}{\alpha})}^{f(\alpha \x)}*\uu \right) \enspace.
\end{align}
Assuming (ii) and (iii) we find that 
$
\F^{f(\x)}((\x,\sigma),\uu) = \alpha \F^{f(\alpha \x)} \left( \left( \frac{\x}{\alpha},\frac{\sigma}{\alpha} \right), \uu \right).
$
Using the homomorphism defined in \eqref{eq:momoSI} the previous equation reads
$$
\F^{f(\x)}((\x,\sigma),\uu) = \Phi(1/\alpha) \F^{f(\alpha \x)} \left( \left( \Phi(\alpha) (\x,\sigma) \right), \uu \right)
$$
which is \eqref{eq:SI}.
Hence we have found an homomorphism such that \eqref{eq:SI} holds, which is the definition of scale-invariance.
The inverse is immediate.
\end{proof}

Remark that given a \acprs\ that satisfies the conditions (i), (ii) and (iii) from the previous proposition, we can reparametrize the state of the algorithm by $\tilde{\st} = \st^{2}$ (if the sampling distribution is Gaussian, this means parametrize by variance instead of standard deviation) leaving unchanged the parametrization for the mean vector. Then the conditions of the previous proposition are not anymore valid for the new parametrization. Yet the algorithm is still scale-invariant but a different morphism needs to be considered, namely
\begin{equation}\label{eq:morphism2}
\Phi(\alpha)(\x,\tilde \sigma) = (\x/\alpha,\tilde \sigma/\alpha^{2}) \enspace.
\end{equation}
Hence the sufficient conditions derived are not general, however they cover typical settings for \acprs. Adapting however Proposition~\ref{lem:scaleinvariant} for other parametrizations is usually easy.

\subsection{Examples of \acprs}\label{sec:ex}

In order to illustrate the \acprs\ framework introduced, we shortly present in this section several examples of \acprs\ algorithms and analyze their invariance properties. For the interested reader, a more thorough description of the working principle and the rationale behind the algorithms is presented in Appendix~\ref{app:ex}.

The algorithms presented pertain to the class of Evolution Strategies (ES) where multivariate normal distribution are used to sample new solutions. We consider first the step-size adaptive ES using cumulative step-size adaptation (CSA) \cite{hansen1995adaptation} (however here with a specific parameter setting that disables the cumulation of information over past iterations) that corresponds to the step-size update rule of the state-of-the art CMA-ES algorithm \cite{hansen2001}. Given the current state $(\Xt,\st)$, $p$ candidate solutions are sampled according to
\begin{equation}\label{eq:mut}
\Sol((\Xt,\st),\Utt^{i}) (= \Xtt^i ) = \Xt + \st \Utt^i \,, \, i=1,\ldots,p \enspace,
\end{equation}
where $(\Utt^i)_{1 \leq i \leq p}$ are i.i.d.\ and follow standard multivariate normal distributions. Hence $\Uspace=\R^{n \times p}$ and $p_{\U}(\uu^{1},\ldots,\uu^{p})$ is the product $p_{\N}(\uu^{1})  \ldots p_{\N}(\uu^{p})$ where 
$ p_{\N}(\x) = \frac{1}{( 2 \pi)^{n/2}} \exp \left( - \frac12 \x^{T} \x \right)$. Let $\Ytt$ be the the vector  $\Ytt=\Perm * \Utt = (\Utt^{\Perm(1)}, \ldots, \Utt^{\Perm(p)})$ where $\Perm$ is the permutation resulting from the ranking of objective function values of the solutions (see \eqref{eq:perm}). The update of $\Xt$ reads
\begin{equation}\label{eq:commamean}
\Xtt =  \GG_{1}((\Xt,\st),\Ytt) : = \Xt + \LRm \st \sum_{i=1}^p w_i \Ytt^{i}
\end{equation}
where  $\LRm \in \Rplus $ is the learning rate (often set to $1$) and $w_{i} \in \R$ are weights that satisfy $w_{1} \geq \ldots \geq w_{p}$ and $\sum_{i=1}^{p} |w_{i}| = 1$. When $\LRm = 1$ the update corresponds to the weighted average of the \new{ranked} candidate solutions $\Xtt^{\Perm(i)}$. The step-size is then updated according to 
\begin{equation}\label{eq:ssDSA3}
\stt = \GG_{2}(\st,\Ytt) = \st \exp \left( \LRsigma \left( \frac{ \sqrt{\mueff } \| \sum_{i=1}^{p} w_{i} \Ytt^{i} \|}{E[\| \N(0,\Id) \|]} - 1 \right) \right)
\end{equation}
where $\LRsigma>0$ is a learning rate (usually set close to one) and $\mueff = 1/\sum w_{i}^{2} $.
Overall, the update function associated to the CSA without cumulation reads
$$
\GG_{\dsa}((\x,\sigma),\y) = \left( \begin{smallmatrix}   \x + \sigma \LRm \sum_{i=1}^{p} w_{i} \y^{i} \\ \sigma  \exp\left( \LRsigma \left( \frac{ \sqrt{\mueff} \| \sum_{i=1}^{p} w_{i} \y^{i} \|}{E[\| \N(0,\Id) \|]} - 1 \right) \right) \end{smallmatrix} \right)  \enspace.
$$
The second example corresponds to the natural gradient update for the step-size with exponential parametrization (\xNES, that stands for exponential natural evolution strategy) \cite{Glasmachers2010}. It uses the same equations to sample solutions, only the step-size update differs and writes
\begin{align}\label{eq:nnes1}
\stt  & = \st \exp\left( \frac{\LRsigma}{2 \dim} \Tr \left( \sum_{i=1}^{p} w_i \Ytt^{i}(\Ytt^{i})^T - \Id \right) \right) \\\label{eq:nnes2} & = \st \exp \left( \frac{\LRsigma}{2\dim} \sum_{i=1}^p w_i ( \| \Ytt^{i}  \|^2 - \dim)  \right).
\end{align}
\generalnote{xNES does not work for n to infinity because no variation in the length of the vectors. Actually not so clear because of the square - the variance might "mess up" this argument.}
The update function for the \xNES\ step-size adaptive algorithm thus reads
\begin{equation}\label{eq:xNES}
\GG_{\xNES}((\x,\sigma),\y) = \left( \begin{smallmatrix}   \x + \sigma \LRm \sum_{i=1}^{p} w_{i} \y^{i} \\ \sigma  \exp\left( \frac{\LRsigma}{2 \dim} \sum_{i=1}^{p} w_{i} ( \| \y^{i}  \|^{2} - \dim) 
\right) \end{smallmatrix} \right)  \enspace.
 \end{equation}
Here, when $\LRm$ and $\LRsigma$ are equal, they coincide with the step-size of the (natural) gradient step of a joint criterion defined on the manifold of Gaussian distributions with covariance matrices equal to a scalar times identity \cite{Glasmachers2010,akimoto2010bidirectional,ollivier2013information}.

\paragraph{Invariance properties} The two different \cprs algorithms presented above are translation invariant and scale-invariant. They indeed satisfy the sufficient conditions derived in Proposition~\ref{prop:TI} and Proposition~\ref{lem:scaleinvariant}.

We present another example where the space $\Uspace$ does not equal $\R^n$. The algorithm belongs to the class of self-adaptive evolution strategies \cite{Rechenberg,Schwefel:77} where parameters are added to the ``genome"  of a solution (vector that encodes a solution) to undergo some variations. We consider a simple example where $\Uspace$ equals $\R^{(n+1) \times p}$. The $n$ first coordinates of an element $\Utt^{i} \in \UUU=\R^{n+1}$ denoted $[\Utt^{i}]_{1\ldots n}$ 
 ($\in \R^{n}$) correspond to the coordinates of a standard multivariate normal distribution vector and the last coordinate denoted $[\Utt^{i}]_{n+1}$ to a standard normal distribution. The candidate solutions sampled via the solution function satisfy
\begin{equation}\label{sample:SA}
\Sol((\Xt,\st),\Utt^{i})= \Xtt^{i} = \Xt + \st \exp \left( \tau [\Utt^{i}]_{n+1} \right) [\Utt^{i}]_{1\ldots n}
\end{equation}
with $\tau > 0$. The update of the mean vector and step-size corresponds to selecting the best solution and its associated step-size, more precisely
\begin{equation}\label{mean:SA}
\Xtt = \Xt + \st \exp( \tau [\Ytt^{1}]_{n+1}) [\Ytt^{1}]_{1 \ldots n}
\end{equation}
and the update for the step-size is
\begin{equation}\label{ss:SA}
\stt = \st \exp ( \tau  [\Ytt^{1}]_{n+1} ) \enspace.
\end{equation}
A step-size adaptive Evolution Strategy satisfying \eqref{sample:SA},\eqref{mean:SA} and \eqref{ss:SA} is called $(1,p)$ self-adaptive step-size ES ($(1,p)$-SA). The $(1,p)$ refers to the fact that a single solution is selected out of the $p$.
The update function $\GG$ for the $(1,p)$-SA reads
$$
\GG_{(1,p){\rm-SA}}((\x,\sigma),\y) = \left( \begin{smallmatrix} \x + \sigma \exp(\tau [\y^{1}]_{n+1}) [\y^{1}]_{1\ldots n} \\ 
\sigma \exp( \tau [\y^{1}]_{n+1} )
 \end{smallmatrix} \right)  \enspace.
$$
\paragraph{Invariances} In virtue of Proposition~\ref{prop:TI} and Proposition~\ref{lem:scaleinvariant} the $(1,p)$-SA is translation and scale-invariant.

We present in a last example, the (1+1)-ES with $1/5$ success rule \cite{Rechenberg} (also introduced as step-size random search  \cite{Schumer:68} or compound random search \cite{Devroye:72}). In this algorithm, $f(\Xt)$ is non-increasing (the algorithm is termed elitist as the best solution cannot be lost) and $\Xt$ is thus the best solution ever seen till iteration $t$. A single solution is sampled from $\Xt$ as
\begin{equation*}
\Xtt^{1} = \Xt + \st \Utt^{1}
\end{equation*}
where $\Utt^{1} \in \R^{\dim}$ follows a standard multivariate normal distribution. The solution is accepted if the candidate solution $\Xtt^1$ is better than $\Xt$ and rejected otherwise. Hence we denote $\Utt^{2}=0 \in \R^{n}$ the zero vector and take $\Utt=(\Utt^{1},\Utt^{2})$ such that $\Uspace = \R^{n \times 2 }$ and the probability distribution of $\U$ equals $p_{\U}(\uu^{1},\uu^{2})=p_{\N}(\uu^{1})  \delta_{0} (\uu^{2})$ where $\delta_{0}$ is the Dirac delta function. The $\Sol$ function corresponds thus to the function in \eqref{eq:mut} and the update for $\Xt$ is similar to \eqref{eq:commamean} with weights $(w_{1},w_{2})=(1,0)$.

The step-size is updated so as to maintain a certain probability of success $\ptarget \in )0,1( $---a probability of success around $1/5$ turns out to be near to optimal in some scenario, see Appendix~\ref{app:ex} for more explanations and references.
One proposed implementation reads
\begin{align*}
\stt & = \st \exp \left( \LRsigma \frac{1_{\{ \Ytt^{1} \neq 0 \}} - \ptarget }{1 - \ptarget}  \right) 
\end{align*}
where $\LRsigma >0$ is a learning rate coefficient. Denoting $\gamma = \exp(\LRsigma)$ and $q=\frac{\ptarget}{1-\ptarget}$ (for a target success probability set to $1/5$, the odds ratio $q=1/4$) yields
\begin{equation}\label{eq:ssOPO}
\stt = \st \left( \factonefifth 1_{\{ \Ytt \neq 0 \}} + \factonefifth ^{-q}  1_{\{ \Ytt^{1} = 0 \}}\right) = \st \left( (\factonefifth - \factonefifth ^{-q}) 1_{\{ \Ytt^{1} \neq 0 \}} + \factonefifth ^{-q}  \right) \enspace.
\end{equation}
Overall, the update transformation for the $(1+1)$-ES with generalized one-fifth success rule is
$$
\GG_{\plusES_{\onefifth}}((\x,\sigma),\y) = \left( \begin{smallmatrix}   \x + \sigma  \y^{1} \\ 
\sigma  
\left( (\factonefifth - \factonefifth ^{-q}) 1_{\{ \y^{1} \neq 0 \}} + \factonefifth ^{-q}  \right)
\end{smallmatrix} \right)  \enspace.
 $$
\paragraph{Invariance} Using again Proposition~\ref{prop:TI} and Proposition~\ref{lem:scaleinvariant}, the $(1+1)$-ES with generalized one-fifth success rule is translation and scale-invariant.

\begin{remark}
In all the examples presented, the $p$ components $(\Utt^{i})_{1 \leq i \leq p}$ of the vectors $\Utt$ are independent. It is however not a requirement of our theoretical setting. \hide{Some algorithms using within an iteration non independent samples were recently introduced \cite{abh2011b,abh2011a} and could be analyzed with the Markov chain approach presented in this paper.}
\end{remark}

%
\section{Scaling-Invariant Functions}\label{sec:scaling-inv}

In this section we define the class of scaling-invariant functions that preserve the $f$-ordering of two points centered with respect to a reference point $\xstar$ when they are scaled by any given factor. 

\begin{definition}[Scaling-invariant function]\label{def:sif1}
A function $f: \R^{\dim} \mapsto \R$ is scaling-invariant with respect to $\xstar \in \R^{n}$, if for all $\scaleSI > 0$, $\x, \y \in \R^{\dim}$
\nnnew{
\begin{equation}\label{eq:scale-inv}
 f(\xstar +\x ) \leq f(\xstar + \y) 
   \Leftrightarrow
   f( \xstar+ \scaleSI \x ) \leq f( \xstar+ \scaleSI \y ) \enspace .
\end{equation}
}
\end{definition}
\nnnew{
This definition implies that two points $\xstar + \x $ and $\xstar + \y$ belong to the same level set if and only if for all $\scaleSI>0$ also $\xstar +  \scaleSI \x$ and $\xstar +  \scaleSI \y $ belong to the same level set, i.e.\
$$
f(\xstar + \x) = f(\xstar + \y) \Leftrightarrow f(\xstar +  \scaleSI \x) = f(\xstar +  \scaleSI \y) \enspace.
$$}
Hence, scaling-invariance can be equivalently defined with strict inequalities in \eqref{eq:scale-inv}.
Remark that if $f$ is scaling-invariant, then for any $g$ strictly increasing the composite $g \circ f$ is also scaling-invariant.

\begin{proposition}\label{prop:unimodalityOFScaleInv}
Let $f$ be a scaling-invariant function, then, $f$ cannot admit any \new{strict} local optima except $\xstar$. In addition, on a line crossing $\xstar$ a scaling invariant function is either constant equal to $f(\xstar)$ or cannot admit a local plateau, i.e.\ a ball where the function is \new{locally} constant.
\del{Let $f$ be a scaling-invariant function, then, $f$ cannot admit any strict local optima except $\xstar$. In addition, $f$ cannot admit \emph{local} plateaus\del{on a line crossing $\xstar$ a scaling invariant function is either constant equal to $f(\xstar)$ or cannot admit a local plateau}, i.e.\ a ball where the function is {locally} constant.}
\end{proposition}
\begin{proof}
We can assume w.l.o.g.\ scaling-invariance with respect to $\xstar = 0$. Assume to get a contradiction that $f$ admits a strict local maximum different from $\xstar$, i.e.\ there exist $\x_{0}$ and $\epsilon > 0$ such that for all $\x \in B(\x_{0},\epsilon)$ (open ball of center $\x_{0}$ and radius $\epsilon$), $f(\x) < f(\x_{0})$. 
We now consider a point $\x_{1}$ belonging to  $B(\x_{0},\epsilon)$ and to the line $(0,\x_{0})$ such that $\|\x_{1}  \| > \| \x_{0}  \|$. Then $f(\x_{1}) < f( \x_{0})$ as $\x_{0}$ is a strict local maximum and $\x_{1}$ can be written $\x_{1} = \theta \x_{0}$ with $\theta >1$ as $\x_{1} \in (0,\x_{0})$ and has a larger norm than $\x_{0}$. Hence
$
f(\x_{0}) > f(\x_{1}) = f(\theta \x_{0}) 
$
which is by the scaling-invariance property equivalent to
$
f(\x_{0} / \theta) >  f( \x_{0})
$. However, $\x_{0} / \theta \in B( \x_{0}, \epsilon)$ as $\|  \x_{0}/\theta - \x_{0}\| = |1 - 1/\theta| \| \x_{0}\| = (\theta -1 ) \| \x_{0} \| / \theta = \| \x_{1} - \x_{0}\| / \theta < \epsilon / \theta < \epsilon$. Then we have found a point $\x_{0} / \theta \in B( \x_{0}, \epsilon)$ that has a function value strictly larger than $f(\x_{0}) $ which contradicts the fact that $\x_{0}$ is a strict local maximum.
The same reasoning holds to prove that the function has no strict local minimum.

The fact that the function is constant on a line crossing $\xstar$ or cannot admit a local plateau, comes from the fact that if the function is non-constant on a line and admits a local plateau, then we can find two points from the plateau $\x$ and $\y$ with equal function value such that the point $\x$ is at the extremity of the local plateau, then we just scale $\x$ and $\y$ such that $\x$ is outside the plateau and $\y$ stays on the plateau. By the scaling invariant property, the scaled points should still have an equal function value which is impossible as we have scaled $\x$ to be outside the plateau. 
\end{proof}

Examples of scaling-invariant functions include linear functions or composite of norm functions by functions in $\Monotone$, i.e.\  $f(\x)=g( \| \x \|)$ where $\| . \|$ is a norm on $\R^\dim$ and $g \in \Monotone$. Thus the famous sphere function $f(\x)=\sum_{i=1}^\dim \x_{i}^{2}$ which is the square of the Euclidian norm or more generally any convex quadratic function $f(\x) = (\x-\xstar)^{T} {\mathit \bf H} (\x - \xstar)$ with ${\mathit \bf H} \in \R^{\dim \times \dim}$  positive definite symmetric are scaling-invariant functions with respect to $\xstar$. 
The sublevel sets defined as the sets $\{ \x \in \R^\dim, f(\x) \leq c \}$ for $c \in \R$ (and $c \geq \inf f$) for those previous examples are convex sets, i.e.\ the functions are \emph{quasi-convex}. However, functions with non-convex sublevel sets can also be scaling-invariant (see Figure~\ref{fig:scalinginvariant}).

\mathnote{Not all unimodal functions are scaling-invariant, consider a unimodal function with level sets that become flatter and flatter. While for scaling invariant functions, the level sets are ``self-similar''.}

\begin{figure}
\centering
\includegraphics[height=0.215\textwidth]{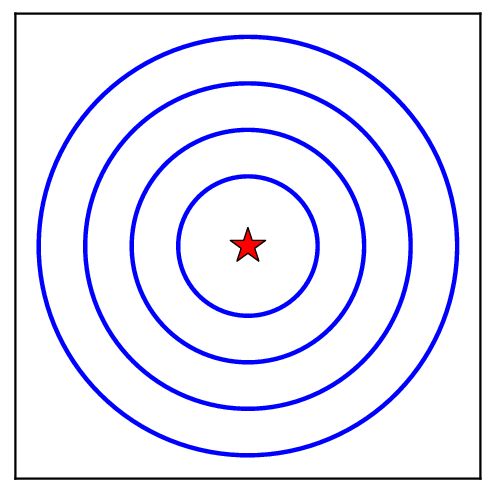}
\includegraphics[height=0.215\textwidth]{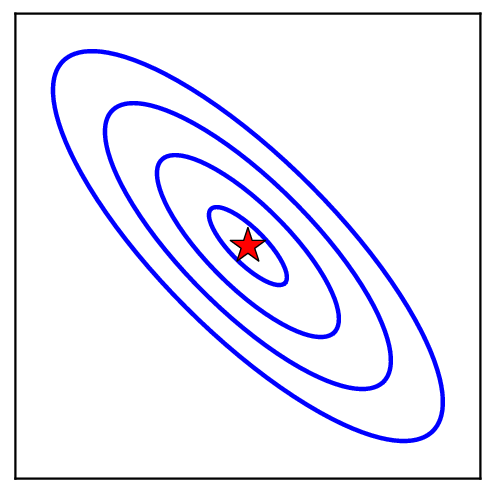}
\includegraphics[height=0.215\textwidth]{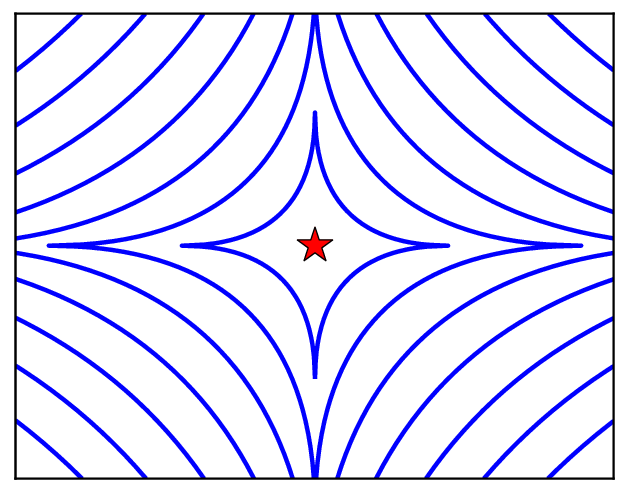}
\includegraphics[height=0.215\textwidth]{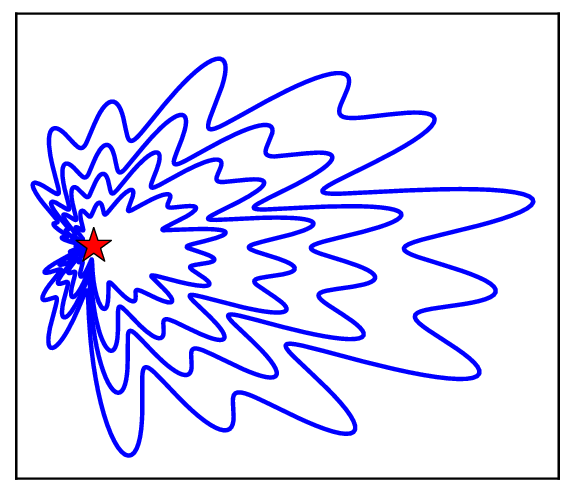}
\caption{\label{fig:scalinginvariant} Illustration of scaling-invariant functions w.r.t.\ the point $\xstar$, depicted with a star. The four functions are composite of $g \in \Monotone$ by $f(\x - \xstar)$ where $f$ is a positively homogeneous function (see Definition~\ref{def:poshom}). From left to right: $f(\x) = \| \x - \xstar \|$; $f(\x) = (\x - \xstar)^{T} \A (\x- \xstar) $ for $\A$ symmetric positive definite; $f(\x)=\left(\sum_i \x_i^{1/2}\right)^2$ the $1/2$-norm; randomly generated scaling-invariant function from a ``smoothly" randomly perturbed sphere function. The two functions on the left have convex sub-level sets contrary to those on the right.}
\end{figure}
A particular class of scaling-invariant functions are positively homogeneous functions whose definition is reminded below.
\begin{definition}[Positively homogeneous functions]\label{def:poshom}
A function $f:\R^\dim \mapsto \R$ is said positively  homogeneous with degree $\alpha$ if
for all $\scaleSI >0$ and for all $\x \in \R^\dim$, $f(\scaleSI \x) = \scaleSI^{\alpha} f(\x)$. 
\end{definition}
From this definition it follows that if a function $\hat f$ is positively homogeneous with degree $\alpha$ then $\hat f(\x - \xstar)$ is scaling-invariant with respect to $\xstar$ for any $\xstar \in \R^{n}$.
Remark that positive homogeneity is not always preserved if $f$ is composed by a strictly increasing transformation.

Examples of positively homogeneous functions are linear functions that are positively homogeneous functions with degree $1$. Also, every function deriving from a norm is positively homogeneous with degree $1$. Examples of scaling-invariant functions deriving from positively homogenous functions are depicted in Figure~\ref{fig:scalinginvariant}. 

In the paper \cite{companion-oneplusone}, stability of the normalized Markov chain is studied on functions $h = g \circ f$ where $f$ is positive homogeneous and $g \in \Monotone$.

\newcommand{\BB}{\mathbf{B}}
\newcommand{\Surf}{\mathbf{\mathcal{S}}}

\section{Joint Markov Chains on Scaling-Invariant Functions}\label{sec:jointMC}

We consider \acprs\ algorithms that are translation invariant and scale-invariant satisfying the properties \eqref{eq:SIsol}, \eqref{eq:SIG1} and \eqref{eq:SIG2} in Proposition~\ref{lem:scaleinvariant}. The functions considered are scaling-invariant. We prove under those conditions that $(\Xt-\xstar)/\st$ is a homogeneous Markov chain. 
\begin{proposition}\label{prop:MConscalinginvariant}
\todo{comparison based: $g \circ f$.}
    Consider a scaling-invariant (w.r.t. $\xstar$) objective function $f$ optimized by $(\Sol,(\GG_{1},\GG_{2}),\Uspace,p_{\U})$, a \acprs\ algorithm assumed to be translation-invariant and scale-invariant satisfying \eqref{eq:SIsol}, \eqref{eq:SIG1} and \eqref{eq:SIG2}. Let $(\Xt,\st)_{t \in \NNN}$ be the Markov chain associated to this \acprs\ and let $(\Ut)_{t \in \Nplus}$ be the i.i.d.\ sequence of random vectors on \Uspace, distributed according to $p_\U$ used for the construction of the Markov chain. Let $\Zt = \frac{\Xt-\xstar}{\st}$ for all $t \in \NNN$. Then $(\Zt)_{t \in \NNN}$ is a homogeneous Markov chain that can be defined independently of $(\Xt,\st)$, provided $\Z_{0} = (\X_{0}-\xstar)/\sigma_{0}$ via
\begin{align}
\Ztt^{i} & = \Sol( (\Zt,1),\Utt^{i}), i=1,\ldots,p \\
\Perm & = \nnnew{\SSel(f(\Ztt^{1}+\xstar),\ldots,f(\Ztt^{p}+\xstar))}\\
\Ztt & = \G(\Zt,\Perm*\Utt) 
\end{align}
where the function $\G$ equals for all $\z \in \R^{n}$ and $\y \in \Uspace$
\begin{equation}
G(\z,\y) = \frac{\GG_{1}((\z,1),\y)}{\GG_{2}(1,\y)}
 \enspace .
\end{equation}
\end{proposition}
This proposition states that the normalized homogeneous Markov chain $(\Zt)_{t \in \NNN}$ is generated independently of $(\Xt,\st)_{t \in \NNN}$ provided the initial condition $\Z_{0} = (\X_{0}-\xstar)/\sigma_{0}$ holds, by (i) sampling candidate solutions with the $\Sol$ function starting from $(\Zt,1)$ (i.e.\ with step-size $1$) (ii) ordering the candidate solutions on $f(.+ \xstar)$ (iii) using the ranking of the candidate solutions to compute $\Ztt$ as the ratio of $\GG_{1}((\Zt,1),\Perm*\Utt)$ (i.e.\ the mean update equation but with step-size $1$ and starting from $\Zt$) divided by the multiplicative update for the step-size taken for a step-size equal to $1$.
\begin{remark}
The previous proposition assumes that scale-invariance is satisfied via the conditions specified in Propositions~\ref{lem:scaleinvariant}. We believe however that when a \acprs\ is scale-invariant under different conditions, a normalized homogeneous Markov chain can be found. For instance when the parametrization $(\Xt,v_t)=(\Xt,\st^{2})$ is used (see discussion around \eqref{eq:morphism2}) the normalized Markov chain is $(\Xt- \xstar)/\sqrt{v_t}$.
\end{remark}
\begin{proof}(of Proposition~\ref{prop:MConscalinginvariant})
We start to prove that the same ordering permutation $\Perm$ is extracted when ranking the candidate solutions $\Xtt^i= \Sol((\Xt,\st),\Utt^i)$ for $i=1,\ldots,p$ on $f$ than ranking the candidate solutions $\Ztt^i = \Sol((\Zt,1),\Utt^i)$ on $ f( . + \xstar)$ assuming $\Zt = (\Xt - \xstar)/\sigma_t$.
We remark first that 
\begin{multline}
f( \Sol((\Xt,\st),\Utt^i)) =  f(\Sol((\Xt - \xstar,\st),\Utt^i) + \xstar) =  \\ f \left(\st \Sol\left(\left(\frac{\Xt - \xstar}{\st},1\right),\Utt^i\right) + \xstar \right) 
=  f \left(\st \Sol\left(\left(\Zt,1\right),\Utt^i\right) + \xstar \right) 
\end{multline}
where we have used successively the translation and scale invariance property of $\Sol$  (see \eqref{eq:condTISol} and \eqref{eq:SIsol}). Given that $f$ is scaling-invariant, the ranking of $$ \{ f \left(\st \Sol\left(\left(\Zt,1\right),\Utt^i\right) + \xstar \right) \}_{1 \leq i \leq p}$$ is the same as the ranking of $\{f(\Sol\left(\left(\Zt,1\right),\Utt^i\right) + \xstar ) \}_{1 \leq i \leq \lambda}$. Hence we have proven that the same permutation ordering is extracted when ranking $\Xtt^i$ on $f(.)$ than ranking $\Ztt^i$ on $f(.+\xstar)$. The following then holds
$$
\Ztt = \frac{\Xtt - \xstar}{\st} = \frac{\GG_1((\Xt,\st),\Perm*\Utt) - \xstar}{\GG_2(\st,\Perm*\Utt)} = \frac{\GG_1((\Zt,1),\Perm*\Utt)}{\GG_2(1,\Perm*\Utt)}
$$
where we have successively used the scale-invariance of $\GG_1$ and $\GG_2$, properties \eqref{eq:SIG1} and \eqref{eq:SIG2}  and the translation invariance of $\GG_1$ \eqref{eq:condTIGG}.
\end{proof}

\generalnote{\begin{proof}(of Proposition~\ref{prop:MConscalinginvariant})
Consider a scaling-invariant function in zero, $f$. Candidate solutions sampled according to the $\Sol$ operator satisfy according to property \eqref{eq:SIsol} $\Sol((\x,\sigma),\uu^{i}) = \sigma \Sol((\x/\sigma,1),\uu^{i}) $. However in a \cprss, the permutation $\Perm$ results from ordering the  objective function of the candidate solutions, i.e.\ ordering $f(\Sol((\x,\sigma),\uu^{i}))$ which is the same as ordering $f( \sigma \Sol((\x/\sigma,1),\uu^{i}) )$ according to property~\eqref{eq:SIsol}. By the scaling-invariant property of the function $f$, we see that it is the same as ordering $f(\Sol((\x/\sigma,1),\uu^{i}) ) $. In other words, on a scaling-invariant function, $\Perm = \Perm^{f}_{(\Xt,\st)} = \Perm^{f}_{(\Xt/\st,1)}$ (putting the initial state as lower subscript).\\
Let $\Xt,\st,\Ut$ be given and let $\Zt=\Xt/\st$, then
$
\Ztt = \frac{\Xtt}{\stt} = \frac{\GG_{1}((\Xt,\st),\Perm*\Ut)}{\GG_{2}(\st,\Perm*\Ut)} \enspace.
$
Because of properties~\eqref{eq:SIG1} and \eqref{eq:SIG2},
$
\Ztt = \frac{\st  \GG_{1}((\Xt/\st,1),\Perm_{(\Xt,\st)}^{f}*\Ut)}{\st \GG_{2}(1,\Perm_{(\Xt,\st)}^{f}*\Ut)}$ and thus $\Ztt = \frac{  \GG_{1}((\Zt,1),\Perm_{(\Zt,1)}^{f}*\Ut)}{ \GG_{2}(1,\Perm_{(\Zt,1)}^{f}*\Ut)}
$. Since we have assume translation invariance of the algorithm, the same construction holds if $\xstar \neq 0$ with $\Zt = \frac{\Xt - \xstar}{\st}$.
\end{proof}
}

Because we assume scale-invariance via the properties of Proposition~\ref{lem:scaleinvariant}, the step-size update has a specific shape. Indeed \eqref{eq:SIG2} implies that
\begin{equation}
\stt = \st \GG_{2}(1,\Ytt)
\end{equation}
where $\Ytt = \Perm * \Utt$. Let us denote the multiplicative step-size update as $\etastar$, i.e.\ 
\begin{equation}\label{eq:defetastar}
\etastar(\Ytt) = \GG_{2}(1,\Ytt) \enspace.
\end{equation}
As explained in the proof of the previous proposition, for a scaling-invariant function $f$, the ranking permutation  is the same \new{on $ f(.)$} starting from $(\Xt,\st)$ or \new{on $f(.+\xstar)$ starting} from $(\Zt,1)$ such that we find that on scaling-invariant functions
\begin{equation}\label{eq:toutt}
\etastar\left(\Perm_{(\Xt,\st)}^{f(.)} * \Utt \right) = \etastar\left(\Perm_{(\Zt,1)}^{f(.+\xstar)} * \Utt \right)
\end{equation}
where $\Perm_{(\Xt,\st)}^{f(.)}$ is the permutation giving the ranking on $f$ starting from the state $(\Xt,\st)$ and $\Perm_{(\Zt,1)}^{f(.+\xstar)}$ the permutation giving the ranking on $f(.+\xstar)$ starting from $(\Zt,1)$.

\begin{remark}
Remark that the construction of the homogeneous Markov chain in the previous proposition only requires that the function is scaling-invariant. We do not assume here that the function has a unique global optimum. Hence the function could be the linear function $f(\x) = \x_{1}$.
\end{remark}\newline
It is immediate now to obtain the transition functions $\G$ associated to the different \cprs examples described in Section~\ref{sec:ex}:
\begin{align}
\G_{\dsa}(\z,\y)& =\frac{\z + \LRm \sum_{i=1}^{p} w_{i} \y^{i}}{\exp\left(\LRsigma \left( \frac{\sqrt{\mueff} \| \sum_{i=1}^{p} w_{i} \y^{i} \|}{E[\| \Normal \|]} -1 \right)  \right)} \\
\G_{\xNES}(\z,\y)& = \frac{\z + \LRm \sum_{i=1}^{p} w_{i} \y^{i}}{\exp\left(\frac{\LRsigma}{2 \dim} \left(   \sum_{i=1}^{p} w_{i}( \| \y^{i} \|^{2} - \dim) \right)  \right)}  \enspace,
\end{align}
where $\y \in \Uspace = \R^{n \times p} $. For the $(1,p)$-SA, $\y \in \Uspace=\R^{(n+1)\times p}$ and


\begin{equation}
\G_{\sa}(\z,\y) = \frac{\z + \exp( \tau [\y^{1}]_{n+1}) [\y^{1}]_{1 \ldots n}}{\exp( \tau [\y^{1}]_{n+1})}
\end{equation}
\new{and finally for the $(1+1)$-ES with generalized $1/5$ success rule,}
$\y$ in $\R^{n \times 2}$ \new{and}
\begin{equation}
\G_{\plusES_{\onefifth}}(\z,\y)= \frac{\z + \y^{1}}{\left( (\factonefifth - \factonefifth ^{-q}) 1_{\{ \y^{1} \neq 0 \}} + \factonefifth ^{-q}  \right)} \enspace .
\end{equation}

\nnote{
Property of the transition functions, no need to formally state as far as continuity is concerned refer that later on we will need the C1 almost everywhere.}

\nnote{Here we need to talk about stochastically recursive sequence. Check what we could get easily from the theory of stochastically monototonic sequence. See chapter 3 of the book ``Understanding MCMC'' by Roberts and Tweedie.}

\section{Sufficient Conditions for Linear Convergence of \acprs\ on Scaling-Invariant Functions}\label{sec:whystability}

We consider throughout this section that $(\Xt,\st)_{t \in \NNN}$ is a Markov chain resulting from a \acprs\ (as defined in Definition~\ref{def:SSAES}) that is translation invariant and scale-invariant satisfying the conditions of Proposition~\ref{lem:scaleinvariant}. The function optimized in this section is a scaling-invariant function $f$ in $\xstar=0$ (this can be assumed without loss of generality in order to simplify the notations). In this context, let $(\Zt = \frac{\Xt}{\st})_{t \in \NNN}$ be the homogeneous Markov chain defined in Proposition~\ref{prop:MConscalinginvariant}.

 For proving linear convergence, we investigate the log-progress $\ln \| \Xtt \|/\| \Xt\|$. The chains $(\Xt,\st)_{t \in \NNN}$ and $(\Zt)_{t \in \NNN}$ being connected by the relation $\Zt =\Xt / \st $, the log-progress  can be expressed as 
\begin{equation}\label{eq:josi}
\ln \frac{\| \Xtt \|}{\| \Xt\|} = \ln \frac{\| \Ztt \| \etastar(\Y({\Zt,\Utt}))}{ \| \Zt \|}
\end{equation}
where the ordered vector $\Perm_{(\Zt,1)} *\Utt$ is denoted $\Y({\Zt,\Utt})$ to signify its dependency in $\Zt$ and $\Utt$, i.e.\
\begin{equation}\label{eq:Yz}
\Y({\z,\uu)} = \Perm_{(\z,1)} * \uu = \OOrd(f(\Sol((\z,1),\uu^{i})_{i=1,\ldots,p})) * \uu  \enspace.
\end{equation}

For \eqref{eq:josi} we have used the fact that the step-size change starting from $(\Xt,\st)$ equals the step-size change starting from $(\Zt,1)=(\Xt/\st,1)$ (see \eqref{eq:toutt}).
Using the property of the logarithm, we express $\frac{1}{t} \ln  \frac{\| \Xt \|}{\| \X_{0} \|}$ as
\begin{align}\label{eq:LCsum}
\frac{1}{t} \ln \frac{\| \Xt \|}{\| \X_{0} \|}
& = \frac1t \sum_{k=0}^{t-1} \ln \frac{\| \X_{k+1} \|}{\| \X_{k}  \|} = \frac1t \sum_{k=0}^{t-1} \ln \frac{\| \Z_{k+1}\| }{\| \Z_{k} \|} \etastar(\Y({\Zt,\Utt})) \enspace.
\end{align}
Let us define for $\z \in \ZZ$, $\mathcal{R}(\z)$ the expectation of the logarithm of $\etastar(\Y({\z,\U}))$ for $\U \sim p_{\U}$, i.e.\
\begin{align}\label{eq:defR}
\mathcal{R}(\z) & = E[ \ln( \etastar( \Y(\z,\U)) ] \\ & =  \int \ln \left( \etastar \left( \SSel(f(\Sol((\z,1),\uu^{i}))_{i=1,\ldots,p}) \right) * \uu \right) p_{\U}(\uu) d \uu \enspace.
\end{align}

\paragraph{Linear convergence}
Almost sure linear convergence can be proven by exploiting \eqref{eq:LCsum} that suggests the application of a Law of Large Numbers (LLN) for Markov chains. Sufficient conditions for proving a LLN for Markov chains are $\varphi$-irreducibility, Harris recurrence and positivity whose definitions are briefly reviewed, see however Meyn and Tweedie for more background \cite{Tweedie:book1993}. 

Let $\Z=(\Zt)_{t \in \NNN}$ be a Markov chain  defined on a state space $\ZZ$ equipped with the Borel sigma-algebra $\B(\ZZ)$.  We denote $P^{t}(\z,A)$, $t \in \NNN$, $\z \in \ZZ$ and $A \in \B(\ZZ)$ the transition probabilities of the chain
\begin{equation}\label{eq:transition}
P^{t}(\z,A)=P_{\z}(\Zt \in A)
\end{equation}
where $P_{\z}$ and $E_{\z}$ denote the probability law and expectation of the chain under the initial condition $\Z_{0} = \z$. If a probability $ \mu$ on $(\ZZ,\B(\ZZ))$ is the initial distribution of the chain, the corresponding quantities are denoted $P_{\mu}$ and $E_{\mu}$. For $t=1$, the transition probability in \eq~\eqref{eq:transition} is denoted $P(\z,A)$.
The chain $\Z$ is $\varphi$-irreducible if there exists a non-zero measure $\varphi$ such that for all $A \in \B(\ZZ)$ with $\varphi(A)>0$, for all $\z_{0} \in \ZZ$, the chain started at $\z_{0}$ has a positive probability to hit $A$, that is there exists $t \in \Nplus$ such that 
$P^{t}(\z_{0},A) > 0 $.
A $\sigma$-finite measure $\pi$ on $\B(\ZZ)$ is said invariant if it satisfies
$$
\pi(A) = \int \pi(d\z) P(\z,A) , \, \, A \in \B(\ZZ) \enspace.
$$
If the chain $\Z$ is $\varphi$-irreducible and admits an invariant probability measure then it is called \emph{positive}. A small set is a set $C$ such that for some $\delta >0$ and $t >0$ and some non trivial probability measure $\nu_{t}$,
$$
P^{t}(\z, .) \geq \delta \nu_{t}(.) , \z \in C \enspace.
$$
The set $C$ is then called a $\nu_{t}$-small set.
Consider a small set $C$ satisfying the previous equation with $\nu_{t}(C) >0$ and denote $\nu_{t} = \nu$. The chain is called aperiodic if the g.c.d. of the set
$$
E_{C} = \{ k \geq 1: C \text{ is a } \nu_{k} \mbox{-small set with } \nu_{k} = \alpha_{k} \nu \mbox{ for some } \alpha_{k} >0 \}
$$
is one for some (and then for every) small set $C$.

A $\varphi$-irreducible Markov chain is \emph{Harris-recurrent} if for all $A \subset \ZZ$ with $\varphi(A) > 0$, and for all $\z \in \ZZ$, the chain will eventually reach $A$ with probability $1$ starting from $\z$, formally if $P_{\z}(\eta_{A} = \infty) = 1$ where $\eta_{A}$ be the \emph{occupation time} of $A$, i.e.\ $\eta_{A} = \sum_{t=1}^{\infty} 1_{\Zt \in A}$. An (Harris-)recurrent chain admits an unique (up to a constant multiple) invariant measure \cite[Theorem~10.0.4]{Tweedie:book1993}.

Typical sufficient conditions for a Law of Large Numbers to hold are $\varphi$-irreducibility, positivity and Harris-recurrence:

\begin{theorem}\label{theo:LLN}[Theorem 17.0.1 in \cite{Tweedie:book1993}]
Assume that $\Z$ is a positive Harris-recurrent chain with invariant probability $\pi$. Then the LLN holds for any $g$ with $\pi(|g|)= \int |g(\x)| \pi(d\x) < \infty$, that is for any initial state $\Z_{0}$,
$
\lim_{t \to \infty} \frac{1}{t} \sum_{k=0}^{t-1} g(\Z_{k}) = \pi(g) \,\, a.s.
$
\end{theorem}

This theorem allows to state sufficient conditions for the almost sure linear convergence of scale-invariant \acprs\ satisfying the assumptions of Proposition~\ref{prop:MConscalinginvariant}. However, before stating those sufficient conditions, let us remark that as a consequence of \eqref{eq:josi}, assuming positivity of $\Z$ and denoting $\pi$ its invariant probability measure, and assuming that (i) $\Z_{0} \sim \pi$, (ii)  $\int \ln \| \z \| \pi(d\z) < \infty$ and (iii) $\int \mathcal{R}(\z) \pi ( d \z)  < \infty$, then for all $t \geq 0$
\begin{equation}
E_{\pi} \left[ \ln \frac{\| \Xtt \| }{\| \Xt \|} \right] = \int E_{\U \sim p_{\U}}[\ln (\etastar(\Y(\z,\U)))] \pi(d\z) = \int \mathcal{R}(\z) \pi( d \z)  \enspace.
\end{equation}
We define the convergence rate $\CR$ as the opposite of the RHS of the previous equation, i.e.\
\begin{equation}\label{eq:ConvRate}
\CR = - \int E_{\U \sim p_{\U}}[\ln (\etastar(\Y(\z,\U)))] \pi(d\z) = -  \int \mathcal{R}(\z) \pi( d \z) \enspace . 
\end{equation}
We now state sufficient conditions such that  linear convergence at the rate $\CR$ holds almost surely independently of the initial state.
\begin{theorem}[Almost sure linear convergence]\label{theo:asCV}
Let $(\Xt,\st)_{t \in \NNN}$ be the recursive sequence generated by a translation and scale-invariant \acprs\ satisfying the assumptions of Proposition~\ref{prop:MConscalinginvariant} and optimizing a scaling-invariant function where w.l.o.g.\ $\xstar=0$. Let $(\Zt)_{t \in \NNN}$ be the homogeneous Markov chain defined in Proposition~\ref{prop:MConscalinginvariant}.
Assume that $(\Zt)_{t \in \NNN}$ is Harris-recurrent and positive with invariant probability measure $\pi$, that $E_{\pi} \ln \| \z \| <  \infty$ and $E_{\pi} \mathcal{R}(\z) d \z < \infty$. Then for all $\X_{0}$, for all $\sigma_{0}$, linear convergence holds asymptotically almost surely, i.e.\
$$
\lim_{t \to \infty} \frac{1}{t} \ln \frac{\| \Xt \|}{\| \X_{0} \|} = - \CR \mbox{ and }  \lim_{t \to \infty} \frac{1}{t} \ln \frac{\st }{ \sigma_{0} }   = - \CR  \enspace a.s.
$$
\end{theorem}

{\em Proof.}
Using \eqref{eq:LCsum} we obtain
$$
\frac{1}{t} \ln \frac{\| \Xt \|}{\| \X_{0} \|} = \frac{1}{t} \sum_{k=0}^{t-1} \ln \| \Z_{k+1} \| - \frac1t \sum_{k=0}^{t-1} \ln \| \Z_{k} \| + \frac1t \sum_{k=0}^{t-1} \ln \etastar(\Y({\Z_k},\U_{k+1}))  \enspace .
$$
We then apply Theorem~\ref{theo:LLN} to each term of the RHS and find
\begin{align*}
\lim_{t \to \infty} \frac{1}{t} \ln \frac{\| \Xt \|}{\| \X_{0} \|} & = 
\int \ln \| \z \| \pi( d \z)
- \int \ln \| \z \| \pi( d \z) + 
\int E[\ln \etastar(\Y({\z},\U)] \pi( d \z) \\
& = 
\int E[\ln \etastar(\Y({\z},\U)] \pi( d \z) = - \CR
 \enspace .
\end{align*}
Similarly since $\frac1t \ln \frac{\sigma_{t}}{\sigma_{0}} = \frac1t \sum_{k=0}^{t-1} \ln \etastar(\Y({\Z_k},\U_{k+1}))$, by applying Theorem~\ref{theo:LLN}, then
$
\lim_{t \to \infty} \frac{1}{t} \ln \frac{\st }{ \sigma_{0} }   = - \CR  \enspace. 
$ \endproof
\mathnote{From MT page 236: ``Positive chains are often called ``positive recurrent'' to reinforce the fact that they are recurrent. Indeed positive chains are recurrent according to Prop 10.1.1.}

Positivity also guarantees convergence of $E_{\z}[ h(\Zt) ]$ from ``almost all'' initial state $\z$ provided $\pi(| h | ) < \infty$. More precisely from \cite[Theorem~14.0.1]{Tweedie:book1993} given a $\varphi$-irreducible and aperiodic chain $\Z$, for $h \geq 1$ a function on $\ZZ$, the following are equivalent:
(i) The chain $\Z$ is positive (recurrent)\footnote{Positive chains are recurrent according to Proposition 10.1.1 of \cite{Tweedie:book1993} but the term positive recurrent is used to reinforce in the terminology the fact that they are recurrent (see \cite{Tweedie:book1993} page 236).} with invariant probability measure $\pi$ and
$
\pi(h) : = \int \pi( d \z) h( \z ) < \infty \enspace.
$
(ii) There exists some petite set $C$ (\cite[Section~5.5.2]{Tweedie:book1993}) and some extended-valued non-negative function $V$ satisfying $V(\z_{0}) < \infty$ for some $\z_{0}$, and
\begin{equation}\label{eq:drift-fnorm}
\Delta V (\z) \leq - h( \z) + b 1_{C}(\z), \enspace \z \in \ZZ,
\end{equation}
where $\Delta$ is the drift operator defined as
\begin{equation}\label{eq:drift}
\Delta V(\z) = \int P(\z,d\y) V(\y) - V(\z) = E_{\z} \left[ V(\Z_{1}) - V(\Z_{0}) \right] \enspace.
\end{equation}
\mathnote{One consequence of this equivalence is a technique to prove the integrability w.r.t. stationary measure. I suspect that this is exactly what we need for the (1+1)-ES and that we somehow prove again by hand. However for the comma case it might be that I need more to prove the integrability of $|\ln \| \z \||$ close to zero as it is not dominated by the drift function - and therefore the need of the extra theorem found in one additional paper of Meyn or Tweedie. Note that I didn't check carefully what I claim in this note.}
Any of those two conditions imply that for any $\z$ in $S_{V}=\{ \z: V(\z) < \infty \}$ 
\begin{equation}\label{eq:cv-fnorm}
\| P^{t}(\z, .) - \pi \|_{h} \xrightarrow[t \to \infty]{} 0 \enspace,
\end{equation}
where $\| \nu \|_{h}:= \sup_{g: |g| \leq h} | \nu (g)|$. Typically the function $V$ will be finite everywhere such that the convergence in \eqref{eq:cv-fnorm} will hold without any restrictions on the initial condition. The conditions (i) or (ii) for the chain $\Z$ with $h(\z) = | \ln \| \z \|  | \nnew{+1} $ imply the convergence of the expected log-progress independently of the starting point $\z$ taken into $S_{V}=\{ \z: V(\z) < \infty \}$ where $V$ is the function such that \eqref{eq:drift-fnorm} is satisfied. More formally
\begin{theorem}[Linear convergence of the expected log-progress]\label{theo:LCexp}
Let $(\Xt,\st)_{t \in \NNN}$ be the recursive sequence generated by a translation and scaling-invariant \acprs\ algorithm satisfying the assumptions of Proposition~\ref{prop:MConscalinginvariant} optimizing a scaling-invariant function where w.l.o.g.\ $\xstar$ is zero. Let $(\Zt)_{t \in \NNN}$ be the homogeneous Markov chain defined in Proposition~\ref{prop:MConscalinginvariant}.
Assume that $(\Zt)_{t \in \NNN}$ is $\varphi$-irreducible and aperiodic and assume that either condition (i) or (ii) above are satisfied with $h(\z) = | \ln \| \z \| | \nnew{+1} $. Assume also that there exists $\beta \geq 1$ such that 
\begin{equation}
\y \mapsto \mathcal{R}(\y)= \int \ln \etastar(\Y({\y},\uu)) p_{\U}(\uu) d \uu \leq \beta (| \ln \| \y \| | \nnew{+1}) \enspace.
\end{equation}
Then for all initial condition $(\X_{0},\sigma_{0})=(\x,\sigma) $ such that $V(\x/\sigma) < \infty $ where $V$ satisfies \eqref{eq:drift-fnorm}
\begin{equation}\label{eq:cvExpe}
\lim_{t \to \infty} E_{\frac{\x}{\sigma}} \left[ \ln \frac{\|\Xtt\|}{\| \Xt \|} \right] = - \CR  \mbox{  and  }  \lim_{t \to \infty} E_{\frac{\x}{\sigma}} \left[ \ln \frac{\sigma_{t+1}}{\sigma_{t}} \right] = - \CR \enspace.
\end{equation}
\end{theorem} 

{\em Proof.}
{\bf Remark $\star$:} Note first that if \eqref{eq:drift-fnorm} is satisfied for a function $V$ for a given $h \geq 1$ then, for $\beta \geq 1$ the function $\beta V$ will satisfy \eqref{eq:drift-fnorm} for the function $\beta h$ such that \eqref{eq:cv-fnorm} will hold with $\beta h$.\\
Let us start by proving the RHS of \eqref{eq:cvExpe} (we set $\z = \x/\sigma$)
\begin{align*}
E_{\frac{\x}{\sigma}} \left[ \ln \frac{\stt}{\st} \right] & = E_{\z} \left[ \ln \etastar(\Y(\Zt,\Utt))  \right] \\
& = \int  P^{t}(\z, d \y) \int  \ln \etastar(\Y(\y,\uu)) p_{\U}(\uu) d \uu = \int P^{t}(\z, d \y) \mathcal{R}(\y) \enspace .
\end{align*}
Since $\mathcal{R} (\y) \leq \beta ( | \ln \| \y \| | + 1) $ and $|\ln \| \y \|  | + 1 $ satisfies either (i) or (ii) we know from the remark $\star$ that 
$
\lim_{t \to \infty} \|  P^{t}(\z, . ) - \pi \|_{\beta (\y \mapsto | \ln \| \y \| |+1)} = 0 
$.
Hence 
$$
| \int P^{t}(\z, d \y) \mathcal{R}(\y) - \underbrace{\int \mathcal{R}(\y) \pi ( d \y)}_{- \CR} | \leq \|  P^{t}(\z, . ) - \pi \|_{\beta (\y \mapsto | \ln \| \y \| |+1)}
$$
converges to $0$ when $t$ goes to $\infty$ that proves the right limit in \eqref{eq:cvExpe}. To prove the left limit in \eqref{eq:cvExpe}, let us write
\begin{align*}
E_{\frac{\x}{\sigma}} \left[ \ln \frac{\|\Xtt\|}{\| \Xt \|} \right] & = E_{\z} \left[ \ln \frac{ \etastar(\Y(\Zt,\Utt)) \| \Ztt \|}{\| \Zt  \|}  \right] \\
& = E_{\z} \left[ \ln \etastar(\Y(\Zt,\Utt)) \right] + E_{\z} [ \ln \| \Ztt \| ] - E_{\z} [ \ln \| \Zt \| ] \enspace.
\end{align*}
However $E_{\z} [ \ln \| \Zt \| ] = \int P^{t}( \z, d \y) \ln \| \y \| $ that converges to $\int \ln \| \y \| \pi(d\y)$ according to \eqref{eq:cv-fnorm}. This in turn implies that $E_{\z} [ \ln \| \Ztt \| ] $ converges to $\int \ln \| \y \| \pi(d\y)$ and hence using the proven result for the right limit in \eqref{eq:cvExpe}, we obtain the left limit in \eqref{eq:cvExpe}.\endproof

Stability like positivity and Harris-recurrence can be studied using drift conditions or Foster-Lyapunov criteria. A drift condition typically states that outside a set $C$, $\Delta V(\z)$ is ``negative''. However ``negativity'' is  declined in different forms. A drift condition for Harris recurrence of a $\varphi$-irreducible chain reads: if there exist a petite set $C$ and a function $V$ unbounded off petite sets such that
$$
\Delta V( \z ) \leq 0 \,, \z \in C^{c}
$$
holds, then the chain $\Z$ is Harris-recurrent \cite[Theorem~9.1.8]{Tweedie:book1993}. To ensure in addition positivity, a drift condition reads: if there exists a petite set $C$ and $V$ everywhere finite and bounded on $C$, a constant $b < \infty$ such that 
$$
\Delta V( \z) \leq -1 + b 1_{C}( \z) , \z \in \ZZ
$$
holds, then $\Z$ is positive Harris-recurrent \cite[Theorem 11.3.4]{Tweedie:book1993}. 

Positivity and Harris-recurrence are typically proven using a stronger stability notion called \emph{geometric ergodicity}   \cite{companion-oneplusone,TCSAnne04}. Geometric ergodicity characterizes that $P^{t}(\z,.)$ approaches the invariant probability measure $\pi$ geometrically fast, at a rate $\rho < 1$ that is independent of the initial point $\z$. A drift condition for proving geometric ergodicity for a $\varphi$-irreducible and aperiodic chain reads: there exist a petite set $C$ and constants $b < \infty$, $\beta > 0$ and a function $V \geq 1$ finite at some $\z_{0} \in \ZZ$ satisfying
\begin{equation}
\Delta V( \z) \leq - \beta V( \z) + b 1_{C}( \z),  \z \in \ZZ \enspace.
\end{equation}
This geometric drift condition  implies that there exist constants $r>1$ and $R < \infty$ such that for any starting point in the set $S_{V}= \{ \z : V(\z) < \infty \} $
\begin{equation}\label{eq:fund}
\sum_{t} r^{t} \| P^{t}(\z_{0},.) - \pi \|_{V} \leq R V(\z_{0})
\end{equation}
where $\| \nu \|_{V} = \sup_{g: |g| \leq V} | \nu(g)|$ (see  \cite[Theorem 15.0.1]{Tweedie:book1993}).
This latter equation allows to have a stronger formulation for the linear convergence of the expected log-progress expressed in Theorem~\ref{theo:LCexp} as formalized in the next theorem.

\begin{theorem}\label{theo:fromGEO-ERGO}
\new{Let $(\Xt,\st)_{t \in \NNN}$ be the recursive sequence generated by a translation and scaling-invariant \acprs\ algorithm satisfying the assumptions of Proposition~\ref{prop:MConscalinginvariant} optimizing a scaling-invariant function where w.l.o.g.\ $\xstar$ is zero. Let $(\Zt)_{t \in \NNN}$ be the homogeneous Markov chain defined in Proposition~\ref{prop:MConscalinginvariant}.}
Assume that $\Z$ is geo\-me\-tri\-cally ergodic satisfying a drift condition with $V$ as drift function. Let $g(\z) = E\left[ \ln[{\|\GG_{1}((\z,1),\Y(\z,\U))\|}/{\| \z \|}] \right]$  and assume that $|g| \leq \beta V$ with $\beta \geq 1$. Then, there exist $r > 1$ and $R< \infty$ such that for any starting point $(\x_{0}, \sigma_{0})$
\begin{equation}\label{eq:non-asymptGE}
\sum_{t} r^{t} |E_{\frac{\x_{0}}{\sigma_{0}}} \ln \frac{\| \Xtt\|}{\|\Xt\|} - (-\CR) | \leq R V\left(\frac{\x_{0}}{\sigma_{0}}\right) \enspace .
\end{equation}
In particular, for any initial condition $(\x_{0},\sigma_{0})$,
$\lim_{t \to \infty} |E_{\frac{\x_{0}}{\sigma_{0}}} \ln \frac{\| \Xtt\|}{\|\Xt\|} - (- \CR)  | r^{t} = 0$
where $r$ is independent of the starting point. Or also for any initial condition, for any $t$,
$
\left|E_{\frac{\x_{0}}{\sigma_{0}}} \ln \frac{\| \Xtt\|}{\|\Xt\|} - (- \CR) \right|  \leq \frac{R V(\x_{0}/\sigma_{0})}{r^{t}}
$.
Let $\tilde g(\z) = E [ \ln \etastar(\z,\Y(\z,\U) ] $. If $\tilde g \leq \beta V$ for $\beta \geq 1$, then there exist $r > 1$ and $R < \infty$ such that 
for any starting point $(\x_{0}, \sigma_{0})$
\begin{equation}\label{eq:GEstep-size}
\sum_{t} r^{t} |E_{\frac{\x_{0}}{\sigma_{0}}} \ln \frac{\sigma_{t+1}}{\sigma_{t}} - (-\CR) | \leq R V\left(\frac{\x_{0}}{\sigma_{0}}\right) \enspace.
\end{equation}
In particular, for any initial condition $(\x_{0},\sigma_{0})$,
$\lim_{t \to \infty} |E_{\frac{\x_{0}}{\sigma_{0}}} \ln \frac{\sigma_{t+1}}{\sigma_{t}} - (- \CR)  | r^{t} = 0$
where $r$ is independent of the starting point. Or also for any initial condition, for any $t$,
$
\left|E_{\frac{\x_{0}}{\sigma_{0}}} \ln \frac{\sigma_{t+1}}{\sigma_{t}} - (- \CR) \right|  \leq \frac{R V(\x_{0}/\sigma_{0})}{r^{t}}
$.
\end{theorem}
\begin{proof}
We assume that $\Z$ is geometrically ergodic satisfying a drift condition with $V$ as drift function. It also implies that $\beta V$ satisfies a drift condition for $\beta \geq 1$. Hence according to \eqref{eq:fund}, there exists $R > 0$ and $r > 1$ such that for any starting point $\z_0$ in the set $S_V = \{ \z : V(\z) < \infty \}$
\begin{equation}\label{eq:local1}
\sum_t r^t \| P^t(\z_0, .) - \pi \|_{\beta V} \leq R V(\z_0) \enspace,
\end{equation}
where $\| \nu \|_{V} = \sup_{g: |g| \leq V} | \nu(g)|$.
Remark now that $E_{\z_0} \ln \| \Xtt \| / \| \Xt \| = E_{\z_0} g(\Z_t) = P^t(\z_0,.) (g)$ where $\z_0 = \x_0 / \sigma_0 $ and thus
\begin{multline*}
| E_{\z_{0}}\left[\ln \frac{\| \Xtt \|}{\| \Xt \|}\right] - \CR | =  | \int g(\z) P^{t}(\z_{0},d\z) - \int \pi(d\z) g(\z) | = \\ | (P^{t}(\z_{0},.) - \pi)(g)|  \leq  \| P^{t}(\z_{0},.) - \pi \|_{\beta V} \enspace
\end{multline*}
where for the last inequality we have used the assumption that $|g| \leq \beta V$. Hence according to \eqref{eq:local1}, there exists $R>0$ and $r > 1$ such that
$$
\sum_t r^t | E_{\z_{0}}\left[\ln \frac{\| \Xtt \|}{\| \Xt \|}\right] - \CR | \leq R V ( \x_0 / \sigma_0) \enspace.
$$
The same holds {\it mutatis mutandis} to prove \eqref{eq:GEstep-size}.
\end{proof}

\del{Geometric ergodicity is also a sufficient condition for the existence of a Central Limit Theorem (see \cite[Theorem~7.0.1]{Tweedie:book1993}) that can characterize how fast $\frac1t \ln \st/\sigma_{0}$ or $\frac1t \ln \| \Xt \| / \| \X_{0} \|$ approach the limit $- \CR$. We refer to \cite[Theorem~4.10]{Tweedie:book1993} for the details.}

\subsection{On Non-asymptotic Results}

We have presented asymptotic convergence results that hold for time to infinity. However from Theorem~\ref{theo:fromGEO-ERGO}, we can derive non-asymptotic results as stated in the following proposition.
\begin{proposition}\label{prop:non-asympt}
Assume that \eqref{eq:non-asymptGE} holds, then for all initial condition $(\x_0, \sigma_0)$ and for all time step $t$, the following non-asymptotic bound holds
\begin{equation}\label{eq:amoup}
E_{\frac{\x_{0}}{\sigma_{0}}} \ln \frac{\| \Xt\|}{\|\X_0\|} \leq - t \CR + R V\left(\frac{\x_{0}}{\sigma_{0}}\right) \frac{r}{r - 1} \enspace,
\end{equation}
where $\CR$ is the convergence rate defined in \eqref{eq:ConvRate}, $V$ is the drift function (see Theorem~\ref{theo:fromGEO-ERGO}) and $r$ and $R$ are the constants that appear in \eqref{eq:non-asymptGE}.
\end{proposition}
\begin{proof}
  Indeed, from \eqref{eq:non-asymptGE}, we obtain that for all $t$,
$$
r^t \left( E_{\frac{\x_{0}}{\sigma_{0}}} \ln \frac{\| \Xtt\|}{\|\Xt\|} - (-\CR) \right) \leq  R V\left(\frac{\x_{0}}{\sigma_{0}}\right)
$$
and thus 
$
E_{\frac{\x_{0}}{\sigma_{0}}} \ln \frac{\| \Xtt\|}{\|\Xt\|} - (-\CR) \leq \frac{1}{r^t} R V\left(\frac{\x_{0}}{\sigma_{0}}\right) $.
By summing up the previous inequation, we find
$$
E_{\frac{\x_{0}}{\sigma_{0}}} \ln \frac{\| \Xt\|}{\|\X_0\|} + t \CR \leq R V\left(\frac{\x_{0}}{\sigma_{0}}\right) \sum_{k=0}^{t-1}  (1/r)^k \leq R V\left(\frac{\x_{0}}{\sigma_{0}}\right) \frac{1}{1 - 1/r} = R V\left(\frac{\x_{0}}{\sigma_{0}}\right) \frac{r}{ r - 1}
$$
and thus \eqref{eq:amoup} holds.
\end{proof}

In \eqref{eq:amoup}, the constants are not explicitly known. The convergence rate $\CR$ is expressed as an expectation with respect to the stationary measure of the Markov chain $\Z$ \new{while} $r$ \new{and} $R$ are \new{finite} constants for which no further estimates are know {\it a priori}. \generalnote{The term $R V\left({\x_{0}}/{\sigma_{0}}\right) \frac{r}{r - 1}$ quantifies the adaptation time. Interestingly, we deduce that if we start from $(\x_0',\sigma_0') = (\alpha \x_0,\sigma_0)$, then the adaptation time will be equal to $R V(\alpha \x_0/\sigma_0) \frac{r}{r - 1}$. Assume for instance that $V(\z) = \| \z \| + 1 $, then the adaptation time will typically roughly be $\alpha$ times longer than if started from $(\x_0,\sigma_0) $. }\niko{The problem with this reasoning is that the term only gives the upper bound for the adaptation time. To say that the adaptation time "will be equal" is not true in general, not even approximately. A simple counter example: if $\x_0=0.01, \sigma_0=1$, then $\alpha=1/1000$ practically won't change the adaptation time in a non-elitist ES, and certainly not by a factor of 1000.  }\anne{That's a very good point. Thanks. I would suggest to keep all this in a note four ourselves and just delete what I wrote.}\todo{put at the place where we comment the experiment on too small step-size.}
\mathnote{
\begin{proposition}
Assume that \eqref{eq:non-asymptGE} holds and that $ E_{\frac{\x_{0}}{\sigma_{0}}} \ln \frac{\| \Xt\|}{\|\X_0\|} $ and $s_t := \sqrt{\Var_{\frac{\x_{0}}{\sigma_{0}}} \ln \frac{\| \Xt\|}{\|\X_0\|}}$ are finite. Then for all $t$, with probability at least $1 - \delta$
\begin{equation}
\ln \frac{\| \Xt\|}{\|\X_0\|} < - t \CR + R V\left(\frac{\x_{0}}{\sigma_{0}}\right) \frac{r}{r - 1} + \frac{ s_t }{  \sqrt{\delta}}.
\end{equation}
Hence let $T_\epsilon : = \min \{ t > 0 |  \| \Xt \| / \| \X_0 \| \leq \epsilon   \}$ be the first hitting time such that the fraction error $\| \Xt \| / \| \X_0 \| $ is less than $\epsilon$, then 
$$
T_\epsilon \leq \bar{T}_\epsilon
$$
where 
\begin{equation}\label{eq:hittingTimeBound}
\bar{T}_\epsilon : = \min \{ t > 0 \, | \, t \, \CR - R V\left(\frac{\x_{0}}{\sigma_{0}}\right) \frac{r}{r - 1} - \frac{ s_t }{  \sqrt{\delta}} \geq  \ln 1/ \epsilon \} \enspace.
\end{equation}
\end{proposition}
\begin{proof}
From \eqref{eq:non-asymptGE}, we obtain that for all $t$
$$
r^t \left( E_{\frac{\x_{0}}{\sigma_{0}}} \ln \frac{\| \Xtt\|}{\|\Xt\|} - (-\CR) \right) \leq  R V\left(\frac{\x_{0}}{\sigma_{0}}\right)
$$
and thus 
$$
E_{\frac{\x_{0}}{\sigma_{0}}} \ln \frac{\| \Xtt\|}{\|\Xt\|} - (-\CR) \leq \frac{1}{r^t} R V\left(\frac{\x_{0}}{\sigma_{0}}\right)
$$
By summing up the previous inequation
$$
E_{\frac{\x_{0}}{\sigma_{0}}} \ln \frac{\| \Xt\|}{\|\X_0\|} + t \CR \leq R V\left(\frac{\x_{0}}{\sigma_{0}}\right) \sum_{k=0}^{t-1}  (1/r)^k \leq R V\left(\frac{\x_{0}}{\sigma_{0}}\right) \frac{1}{1 - 1/r} = R V\left(\frac{\x_{0}}{\sigma_{0}}\right) \frac{r}{ r - 1}
$$
and thus
\begin{equation}\label{eq:moup}
E_{\frac{\x_{0}}{\sigma_{0}}} \ln \frac{\| \Xt\|}{\|\X_0\|} \leq - t \CR + R V\left(\frac{\x_{0}}{\sigma_{0}}\right) \frac{r}{r - 1}
\end{equation}
Let us denote $A_t$ the sequence of random variable $A_t =  \ln \frac{\| \Xt\|}{\|\X_0\|} $ and $s_t = \sqrt{\Var(A_t)}$ (assume to be finite). Then from Chebyshev's inequality, for any $\alpha > 0$,
$$
\Pr \left( | A_t - E_{\frac{\x_0}{\sigma_0}} \left[ A_t \right] | \geq s_t \sqrt{\alpha} \right)  \leq \frac{1}{\alpha}
$$
For $\alpha > 1$ and defining $\delta = 1 - 1/\alpha$, we obtain that with probability at least $1 - \delta$
$$
\left| A_t - E_{\frac{\x_0}{\sigma_0}} \left[ A_t \right] \right| < \frac{ s_t }{  \sqrt{\delta}} \enspace ,
$$
which implies that with probability at least $1 - \delta$
$$
A_t < E_{\frac{\x_0}{\sigma_0}} \left[ A_t \right] + \frac{ s_t }{  \sqrt{\delta}}
$$
Using \eqref{eq:moup} we find that with probability at least $1-\delta$
$$
\ln \frac{\| \Xt\|}{\|\X_0\|} < - t \CR + R V\left(\frac{\x_{0}}{\sigma_{0}}\right) \frac{r}{r - 1} + \frac{ s_t }{  \sqrt{\delta}}
$$
\end{proof}
To get a more explicit expression for a bound of the hitting time using \eqref{eq:hittingTimeBound}, we would need to have an estimate of the variance $\s_t$. If we assume that $\s_t^2 = c \, t  $ where $c > 0$ is a real constant which is in particular the case if $\ln \| \Xtt \| / \| \Xt \|$ are i.i.d., then we find that
$$
\bar{T}_\epsilon = \left\lceil \frac{1}{4 \CR^2} \left( \frac{c}{\sqrt{\delta}} + \sqrt{ \frac{c^2}{\delta} + 4 \CR \left( R V\left(\frac{\x_{0}}{\sigma_{0}}\right) \frac{r}{r - 1} + \ln(1/\epsilon) \right) }  \right)^2 \right\rceil
$$
where $  \lceil x \rceil $ denotes the ceiling of $x$. In general we can expect to obtain lower and upper bounds estimates for the variance $\s_t^2$ that can be used to obtain estimates of the hitting time $T_\epsilon$.}

\subsection*{Interpretation and Illustration}
\begin{figure}
\centering
\includegraphics[width=0.32\textwidth]{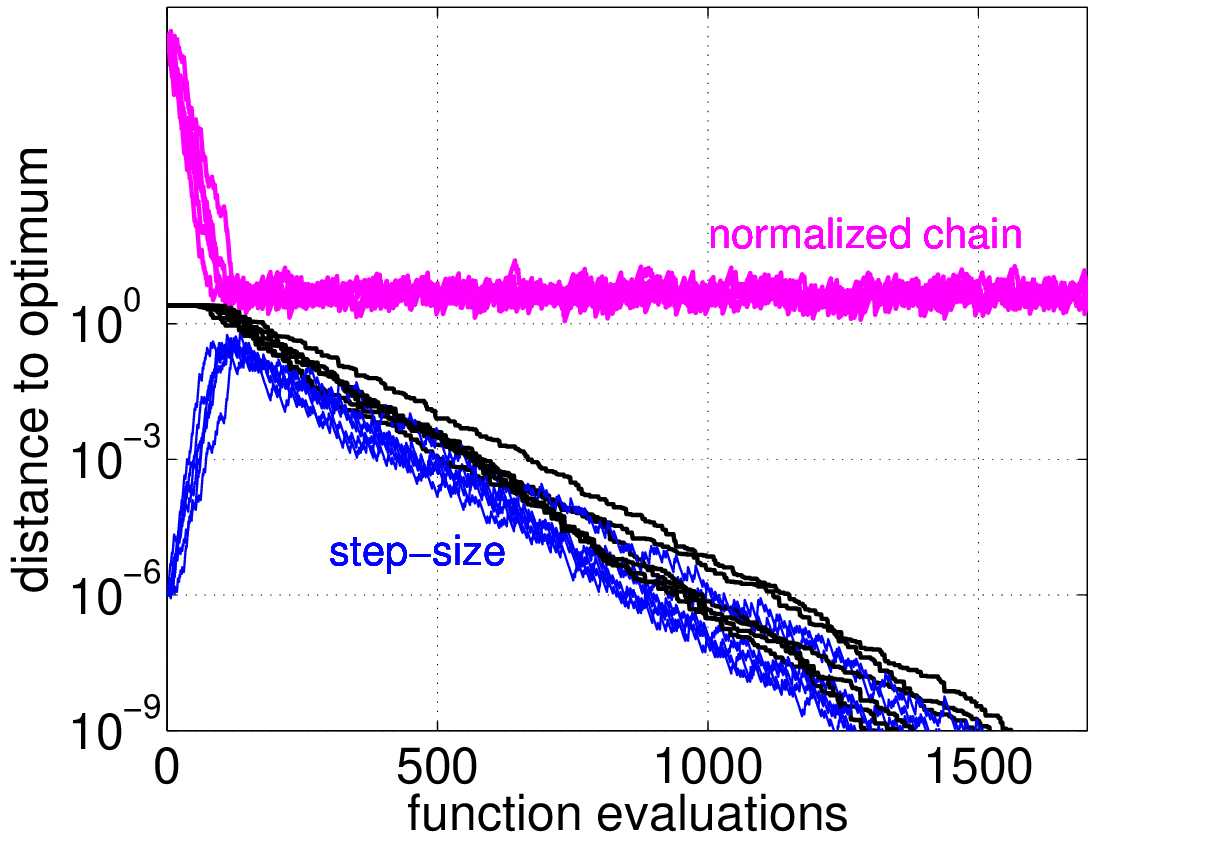}
\includegraphics[width=0.29\textwidth]{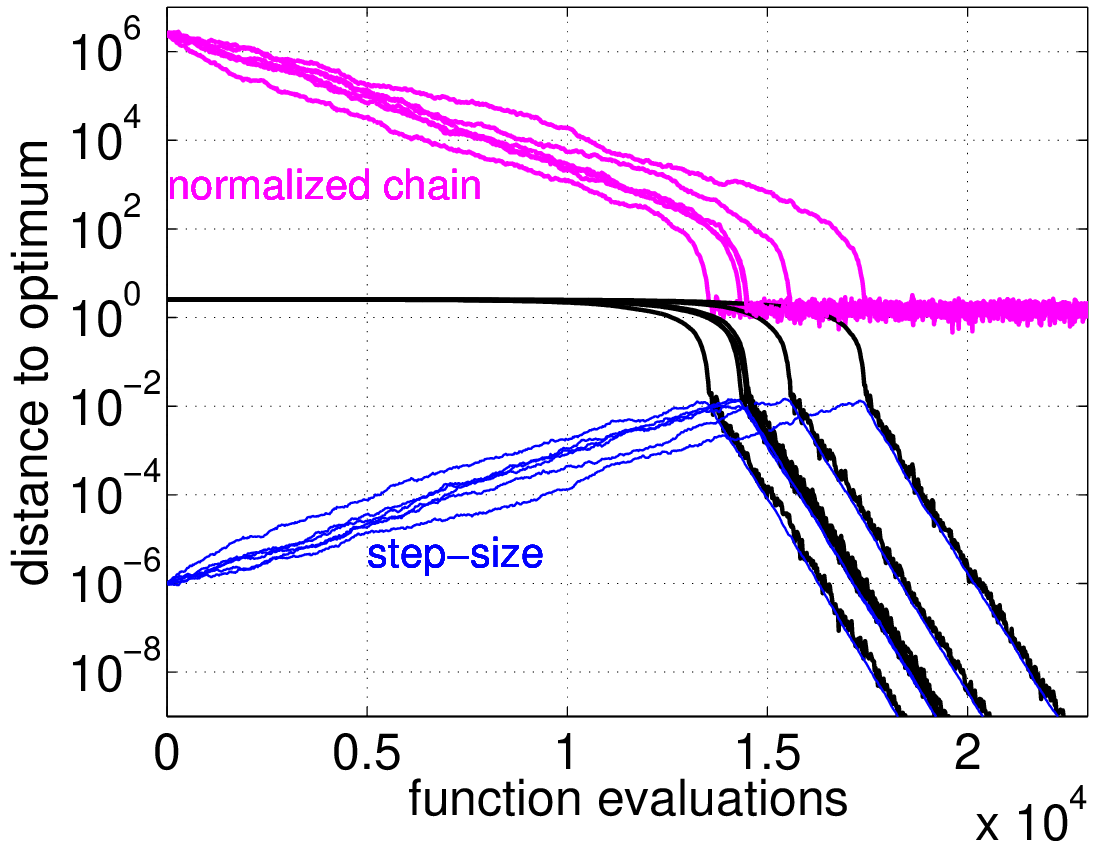}
\includegraphics[width=0.295\textwidth]{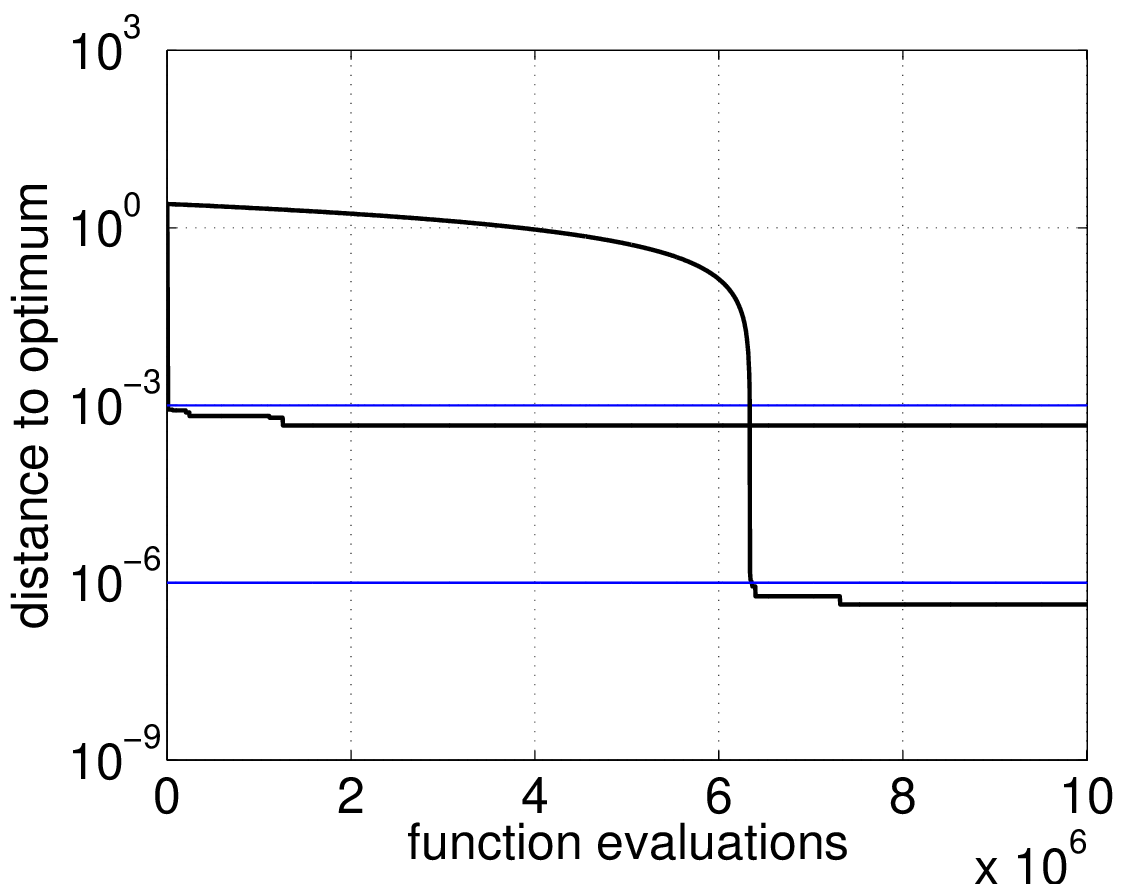}
\caption{\label{fig:simul} Convergence simulations  on spherical functions $f(\x)= g(\| \x \|)$ for $g \in \Monotone$ in dimension $\dim = 10$.
Left: Simulation of the $(1+1)$-ES with one-fifth success rule (see Section~\ref{sec:ex}, step-size update of \eqref{eq:ssOPO} implemented with parameters $\ptarget =1/5$, $\LRsigma=1/3$ were used)). Middle:
 \xNES  (see Section~\ref{sec:ex}) using the parameters of Table~\ref{param-expe}. 
Each plot is in log scale and depicts in black the distance to optimum, i.e.\ $\| \Xt\|$, in blue the respective step-size $\st$ and in magenta the norm of the normalized chain $\| \Zt\|$. The $x$-axis is the number of function evaluations corresponding thus to the iteration index $t$ for the $(1+1)$-ES and to $p \times t$ for \xNES.
 For both simulations $6$ independent runs are conduced starting from $\X_{0}=(0.8 ,0.8 ,\ldots,0.8 )$ and $\sigma_{0}=10^{-6}$. Right: Simulation of a $(1+1)$-ES with constant step-size. Two runs conducted with a constant step-size equal to $10^{-3}$ and $10^{-6}$. The distance to the optimum is depicted in black and the step-size in blue.}
\end{figure}
Figure~\ref{fig:simul} illustrates the theoretical results formalized above. On the two leftmost plots, six single runs of the $(1+1)$-ES with one-fifth success rule and of the \xNES\ algorithm optimizing spherical functions $f(\x)= g(\| \x \|)$ for $g \in \Monotone$ in dimension $n=10$ are depicted (see caption for parameters used). The evolution of $\| \Xt \|$, $\sigma_{t}$ and $\| \Zt \|$ are displayed using a logarithmic scale. In order to be able to compare the convergence rate between both algorithms, the $x$-axis represents the number of function evaluations and not the iteration index (however for the $(1+1)$-ES both number of function evaluations and iteration index coincide). The runs are voluntarily started with a too small step-size (equal to $10^{6}$) compared to the distance to the optimum in order to illustrate the adaptivity property of both algorithms. For the $(1+1)$-ES, we observe a low variance in the results: after 100 function evaluations all the runs reach a well adapted step-size and the linear convergence is observed for both the step-size and the norm. The slope of the linear decrease observed coincides with $-\CR$, the convergence rate associated to the $(1+1)$-ES (up to a factor because a base 10 is used for the display). As theoretically stated $\ln \st$ and $\ln \| \Xt\|$ converge at the same rate (same slope for the curves). The norm of the normalized chain $\Zt$ is depicted in magenta, we observe that the stationary regime or steady-state of the chain correspond to the moment where linear convergence starts as predicted by the theory.

For the \xNES\ algorithm, we observe the same behavior for each single run, i.e.\ a first phase where the adaptation of the step-size is taking place, here it means that the step-size is increased and a second phase where linear convergence is observed. In terms of normalized chain it corresponds to a first phase where a ``transient behavior'' is observed and a second  phase where the distribution of the chain is close from the stationary distribution. We however \new{observe that more time is needed for \xNES\ than for the $(1+1)$-ES  to reach the stationary regime (second phase).} The slope after reaching a reasonable step-size corresponds to the convergence rate $\CR$ multiplied by $p$ (up to the difference with the base 10 logarithm). \generalnote{WITH OLD PARAMETERS (=CMA ones) Both convergence rates between the $(1+1)$-ES and \xNES\ are comparable while of course the number of function evaluations to reach $10^{-9}$ starting from a step-size of $10^{-6}$ is much longer for \xNES\ as the adaptation phase is much slower for \xNES\ than for the $(1+1)$-ES. It illustrates that only comparing the convergence rate (per function evaluation) can be misleading as it does not reflect the adaptation time.}

\mathnote{So this larger variance is most certainly explained by the variance of the log-step-size change on linear functions. Computation of this variance should be do-able while a bit ugly. We also believe that cumulation decreases this variance (must be related to some findings of Alexandre).}

Convergence of each single run reflects the almost-sure convergence property. Theoretically, the geometric ergodicity ensures that the adaptation phase is ``short'' as the Markov chain reaches its stationary state geometrically fast, i.e.\ we can start from a bad initial step-size, this bad choice will be fast corrected by the algorithm that will then converge linearly. In terms of the Markov chain $\Zt$, the bad choice is translated as starting far away from the stationary distribution and the correction means reaching the stationary measure. We see however that in those ``fast'' statements the constants are omitted as for the \xNES\ we observe that the step-size increase can take up to more than 3 times more function evaluations than decreasing the step-size.

The rightmost plot in Figure~\ref{fig:simul} depicts the convergence of a non step-size adaptive strategy, here a $(1+1)$-ES with constant step-size equal to $10^{-3}$ and $10^{-6}$. Theoretically the algorithm converges with probability one, at the same rate than the pure random search algorithm though. The plots illustrate the necessity of a step-size adaptive method: a wrong choice of the initial parameter has a huge effect in terms of time needed to reach a given target value. Indeed starting from a step-size of $10^{-3}$, $10^4$ function evaluations are needed to reach a target of $10^{-3}$ while with a step-size of $10^{-6}$, roughly $6.2 \times 10^{6}$ function evaluations are needed to reach the same target (i.e.\ more than 3 orders of magnitude more). Also we see that starting from a step-size of $10^{-3}$, the number of function evaluations to reach a target of $10^{-6}$ will likely be above $10^{19}$ function evaluations.\niko{As of 2015, 1e7 evaluations take about 5 minutes on Anne's laptop, $10^{19}$ means then 10 million years. }

This rightmost plot also illustrates the importance to study theoretically convergence rates, as convergence with probability one can be associated to an algorithm having very poor performance for practical purposes.

In Appendix~\ref{sec:AppNE}, we present some more numerical tests of the $(1+1)$-ES with one-fifth success rule and of \xNES\ together with the other step-size adaptive algorithms sketched in the paper, the Nelder-Mead algorithm \nnew{\cite{NelderMead:65}} and the Random Pursuit algorithm \cite{stich:2013}.

\section{Discussion}\label{sec:discussion}
This paper provides a general methodology to prove global linear convergence of some \cprs algorithms on scaling-invariant functions, a class of functions that includes in particular many \emph{non quasi-convex} and \emph{non continuous} functions. The methodology exploits the invariance properties of the algorithms and turns the question of global linear convergence into the study of the stability of an underlying homogeneous normalized Markov chain. It generalizes previous works \cite{Bienvenue:2003,TCSAnne04} to a broader class of functions and a broader class of algorithms.

Different notions of stability for a Markov chain exist. They imply different (non equivalent) formulations of linear convergence that give many insights on the dynamic of the algorithm:
positivity and Harris recurrence essentially imply the existence of a convergence rate $\CR$ such that for any initial state almost surely
$$
\lim_{t \to \infty}\frac{1}{t} \ln \frac{\| \Xt - \xstar \|}{\| \X_{0} - \xstar \|} = - \CR = \lim_{t \to \infty} \frac{1}{t} \ln \frac{\st}{\sigma_{0}}   \enspace
$$
holds.
Positivity essentially implies that the limit of the expected log-progress or log step-size is $- \CR$. More precisely for any initial state $\X_{0}=\x, \sigma_{0}= \sigma$
$$
\lim_{t \to \infty} E_{\frac{\x}{\sigma}} \left[ \ln \frac{\|\Xtt - \xstar \|}{\| \Xt - \xstar \|} \right] = - \CR =
\lim_{t \to \infty} E_{\frac{\x}{\sigma}} \left[ \ln \frac{\sigma_{t+1}}{\sigma_{t}} \right]  \enspace .
$$
Geometric ergodicity then characterizes that the expected log-progress sequence converges geometrically fast to the convergence rate limit $-\CR$.

Linear convergence  holds under any initial condition. This reflects the practical adaptivity property: the step-size parameter is adjusted on the fly and hence a bad choice of an initial parameter is not problematic. We have illustrated that the transition phase, \new{in other words the time}\del{formally how long} it takes to be close to the invariant probability measure, relates to how long it takes to forget a bad initialization. 

The methodology provides an exact formula for the convergence rate $\CR$ expressed in terms of expectation w.r.t.\ the invariant probability measure of the normalized Markov chain. Exploiting the exact expression for deducing properties on the convergence rate like dependency w.r.t.\ the dimension or dependency on function properties (like condition number of the hessian matrix if the function is convex quadratic) seems however to be quite challenging with this approach while it is feasible with ad-hoc techniques for specific algorithms (see \cite{jens:2005}). Numerical simulations need then to be performed to investigate those properties. 
Nevertheless, the Markov chain methodology proposed here provides a rigorous framework for performing these simulations\del{: it proves that by essence Monte-Carlo simulation of the convergence rate is consistent}.\todo{}\niko{"to be by essence consistent" doesn't carry much meaning to me, given that everything we do should be consistent to begin with. }%
\del{ and even provides through the Central Limit Theorem asymptotic confidence intervals for the simulations. }

We have restricted for the sake of simplicity the \acprs\ framework to the update of a mean vector and a step-size. However some step-size adaptive algorithms like the cumulated step-size adaptation used in the CMA-ES algorithm include other state variables like an auxiliary vector (the path) used to update the step-size \cite{hansen2001}. Adaptation of the present methodology to cases with more state variables seems however relatively straightforward\new{ while we foresee that proving the stability of the underlying homogeneous Markov chains can be very complex}.

The current approach exploits heavily invariance properties of the algorithms investigated together with invariance properties of the objective function. Hence, we expect that the methodology does not generalize directly to all unimodal functions. However we believe that there is room for extensions of the framework, for instance in some noisy contexts (i.e.\ the objective function is stochastic).

The algorithms studied in this paper are adapting a global scaling of the underlying probability distribution through the adaptation of the step-size. We would like to stress however that in practice, algorithms should also adapt the geometric shape of the distribution, for instance through a covariance matrix. Adapting the shape is indeed crucial to efficiently solve \emph{ill-conditioned} problems. In effect, the state-of-the art CMA-ES algorithm adapts \emph{both} the step-size and the covariance matrix of the multivariate normal distribution used to sample new points.

\nnnew{Last, we want to emphasize that the current paper clarifies the relationship between comparison-based stochastic adaptive algorithms and Markov chain Monte Carlo (MCMC) algorithms. MCMC methods are algorithms used to sample probability distributions. 
They implement the construction of a stable Markov chain having as invariant distribution the distribution to be sampled. 
This latter distribution is typically non-singular. In contrast, given that the goal in optimization is to converge, the Markov chain generated by the optimization algorithm is not stable. However as seen in the paper, on scaling-invariant functions, a joint potentially stable homogeneous Markov chain associated to the original chain exists (here this chain is $\Zt=(\Xt - \xstar) / \st $). This Markov chain defines an MCMC algorithm associated to the optimization algorithm.}

Another possible approach to analyze the linear convergence of a \cprs consists in using stochastic approximation theory or the \emph{method of ordinary differential equations} \cite{Kushner2003book,Borkar2008book}. We believe that linear convergence can then be proven on different function classes for learning rates ($\LRm, \LRsigma$ in \eqref{eq:commamean} and \eqref{eq:nnes2} for instance) that are small enough. A step needed in this analysis is the investigation of an ordinary differential equation obtained by suitable averaging. We believe that this can be done by extending results presented in \cite{akimoto:2012}.\\

%
%
%
%
%
\paragraph{\bf Acknowledgements} We would like to thank Mihai Anitescu for his kind support during the pre-publication process of the paper and his engagement to find qualified reviewers for a paper which lies beyond standard scope of the journal. In addition, we would like to thank Youhei Akimoto for helpful discussions regarding the manuscript. Part of those discussions took place during the Dagstuhl seminar No 15211 on the Theory of Evolutionary Algorithms (\url{http://www.dagstuhl.de/15211}) that we would like to acknowledge. We would also like to thank Alexandre Chotard for helpful discussions and proofreading of the manuscript. This work was partially supported by the grant ANR-12-MONU-0009 (NumBBO) of the French National Research Agency.

\begin{appendix}
\section{Examples of \acprs}\label{app:ex}

We provide here a detailed description of the examples briefly presented in Section~\ref{sec:ex}. This appendix is self-contained and can be read independently of Section~\ref{sec:ex}.

\subsection{Non-elitist Step-size Adaptive Evolution Strategies (ES)}\label{app:sec:algo-comma}

We consider two examples of algorithms following Definition~\ref{def:SSAES} that were introduced under the name Evolution Strategies (ES). They all share the  same sampling space $\Uspace = \R^{n \times p}$. A vector $\Ut \in \Uspace=\R^{n \times p}$ is composed of $p$ i.i.d.\ standard multivariate normal distributions, i.e.\ $\Ut^{i} \sim \Normal \in \R^{n}$ 
and thus the joint density\footnote{\nnnew{With a small abuse of notations, we use the same notations for the density associate to the distribution $p_\U$ than for the distribution itself.}} $p_{\U}(\uu^{1},\ldots,\uu^{p})$ is the product $p_{\N}(\uu^{1})  \ldots p_{\N}(\uu^{p})$ where 
$ p_{\N}(\x) = \frac{1}{( 2 \pi)^{n/2}} \exp \left( - \frac12 \x^{T} \x \right)$.
The solution operator to sample new solutions is given by:
\begin{equation}\label{app:eq:mut}
\Sol((\Xt,\st),\Utt^{i}) (= \Xtt^i ) = \Xt + \st \Utt^i \,, \, i=1,\ldots,p \enspace,
\end{equation}
and hence each candidate solution $\Xtt^{i}$ follows the distribution $\N(\Xt,\st^{2} \Id)$. 
\hide{The vector $\Xt$ is thus the mean vector of the underlying distribution $\N(\Xt,\st^{2} \Id)$. Using the terminology sometimes employed for evolution strategies--for stressing the parallel with biology--, \eqref{eq:mut} implements a \emph{mutation}: the new solutions $\Xt^{i}$, also called \emph{offspring}, are obtained by mutation of the solution $\Xt$ also referred to as \emph{parent}. }

Given the vector of ordered samples $\Ytt=\Perm * \Utt = (\Utt^{\Perm(1)}, \ldots, \Utt^{\Perm(p)})$ where $\Perm$ is the permutation resulting from the ranking of objective function values of the solutions (see \eqref{eq:perm}), the update equation for the mean vector $\Xt$ that defines the function $\GG_{1}$ is given by 
\begin{equation}\label{app:eq:commamean}
\Xtt =  \GG_{1}((\Xt,\st),\Ytt) : = \Xt + \LRm \st \sum_{i=1}^p w_i \Ytt^{i}
\end{equation}
where $\LRm \in \Rplus $ is usually called the learning rate and is often set to $1$ and $w_{i} \in \R$ are weights that satisfy $w_{1} \geq \ldots \geq w_{p}$ and $\sum_{i=1}^{p} |w_{i}| = 1$. \hide{The update implements the idea to move the vector towards better solutions by doing a  \emph{weighted recombination} of the $p$ best solutions \cite{rudolph:97,hansen2001,arnold2006weighted}. Often, only positive weights are considered where optimally half of the weights should be non zero, i.e.\ $w_{1} \geq w_{2} \geq \ldots \geq w_{\lfloor p/2\rfloor} > 0$ and $w_{i} = 0$ for $i > \lfloor p/2\rfloor$. If equals weights are used the terminology intermediate recombination is employed \cite{HGBeyer01}.}

Recently, an interesting interpretation of the meaning of the vector $\st \sum_{i=1}^p w_i \Ytt^{i}$ was given: it is an approximation of the ($n$ first coordinates) of the natural gradient of a joint criterion defined on the manifold of the family of Gaussian probability distributions \cite{akimoto2010bidirectional,ollivier2013information}.

Several step-size updates have been used with the update of the mean vector $\Xt$ in \eqref{app:eq:commamean}. First of all, consider the update derived from the cumulative step-size adaptation or path-length control without cumulation \cite{hansen1995adaptation} that reads
\begin{equation}\label{app:eq:ssDSA}
\stt = \GG_{2}(\st,\Ytt) = \st \exp \left( \LRsigma \left( \frac{ \sqrt{\mueff } \| \sum_{i=1}^{p} w_{i} \Ytt^{i} \|}{E[\| \N(0,\Id) \|]} - 1 \right) \right)
\end{equation}
where $\LRsigma>0$ is the learning rate for the step-size update usually set close to one and $\mueff = 1/\sum w_{i}^{2} $. The value $1/\LRsigma$ is often considered as a damping parameter.
The ruling principle for the update  is to compare the length of the recombined step $\sum_{i=1}^{p} w_{i} \Ytt^{i}$ to its expected length if the objective function would return independent random values. Indeed if the signal given by the objective function is random, the step-size should stay constant or \new{increase moderately}. It is not difficult to see that in such conditions, a random ordering takes place and hence the distribution of the vector $\Ytt$ is the same as the distribution of the vector $\Utt$, finally it follows that $\sqrt{\mueff}  \sum_{i=1}^{p} w_{i} \Ytt^{i} $ is distributed according to a standard multivariate normal distribution. Hence \eqref{app:eq:ssDSA} implements to increase the step-size if the observed length of $\sqrt{\mueff}  \sum_{i=1}^{p} w_{i} \Ytt^{i} $ is larger than the expected length under random selection and decrease it otherwise. Overall, the update function associated to the CSA without cumulation reads
$$
\GG_{\dsa}((\x,\sigma),\y) = \left( \begin{smallmatrix}   \x + \sigma \LRm \sum_{i=1}^{p} w_{i} \y^{i} \\ \sigma  \exp\left( \LRsigma \left( \frac{ \sqrt{\mueff} \| \sum_{i=1}^{p} w_{i} \y^{i} \|}{E[\| \N(0,\Id) \|]} - 1 \right) \right) \end{smallmatrix} \right)  \enspace.
 $$
In practice, the step-size update CSA is used in combination with so-called cumulation and in the update \eqref{app:eq:ssDSA}, the term $\sqrt{\mueff } \| \sum_{i=1}^{p} w_{i} \Ytt^{i} \|$ is replaced by the norm of the cumulated path defined as
\begin{equation}\label{eq:path}
\p_{t+1} = (1-c) \p_t + \sqrt{c (2 - c)} \sqrt{\mueff} \sum_{i=1}^{p} w_{i} \Ytt^{i}
\end{equation}
where $c \in (0,1]$ is the cumulation parameter. The cumulated path cumulates information of previous iterations \cite{hansen2001}. The CSA with cumulation is the default step-size adaptation mechanism used in the CMA-ES algorithm.

The second example we present corresponds to the natural gradient update for the step-size with exponential parametrization \cite{Glasmachers2010} that writes 
\begin{align}\label{app:eq:nnes1}
\stt  & = \st \exp\left( \frac{\LRsigma}{2 \dim} \Tr \left( \sum_{i=1}^{p} w_i \Ytt^{i}(\Ytt^{i})^T - \Id \right) \right) \\\label{app:eq:nnes2} & = \st \exp \left( \frac{\LRsigma}{2\dim} \sum_{i=1}^p w_i ( \| \Ytt^{i}  \|^2 - \dim)  \right).
\end{align}
\generalnote{xNES does not work for n to infinity because no variation in the length of the vectors. Actually maybe because of the square the argument does not hold (because of variance) - might be worth being investigated.}We then define the update function for the step-size update in \xNES\ as \del{multiplicative update function for the step-size as
\begin{align}\label{app:eq:dsas}
\etastar_{\xNES}(\y^{1},\ldots,\y^{\mu})& = \exp \left(  \frac{\LRsigma}{2\dim} \sum_{i=1}^p w_i ( \| \y^{i}  \|^2 - \dim)\right)  \enspace ,
\end{align}
and hence}
\begin{equation}\label{app:eq:xNES}
\GG_{\xNES}((\x,\sigma),\y) = \left( \begin{smallmatrix}   \x + \sigma \LRm \sum_{i=1}^{p} w_{i} \y^{i} \\ \sigma  \exp\left( \frac{\LRsigma}{2 \dim} \sum_{i=1}^{p} w_{i} ( \| \y^{i}  \|^{2} - \dim) 
\right) \end{smallmatrix} \right)  \enspace.
 \end{equation}
 \new{Here, when $\LRm$ and $\LRsigma$ are equal, they coincide with the step-size of the (natural) gradient step of a joint criterion defined on the manifold of Gaussian distributions with covariance matrices equal to a scalar times identity \cite{Glasmachers2010}.}
\del{The three subsequent algorithms will be denoted $\commaES_{\times}$ES with $_{\times}$ being either \dsa, \dsas or \xNES\ and the generic notation for one of the three algorithms will be \commaES-ES.}
\hide{Often the parameter $p$ is denoted $\lambda$, a parameter $\mu$ corresponding to the number of positive weights is introduced and the terminology comma selection as opposed to plus (or elitist selection) that will be detailed later\niko{it must be mentioned that plus selection is not modeled here} is employed. This terminology stresses that the update of the vector $\Xt$ and step-size takes into account solutions sampled anew and do not consider ``old'' solutions. Consequently it is not guaranteed that the best solutions at iteration $t+1$ has a smaller objective function value than the best solution at iteration $t$.} 

With those algorithms, it is not guaranteed that the best solution at iteration $t+1$ has a smaller objective function value than the best solution at iteration $t$. In the case where only positive weights are used, a compact notation for the algorithms described above is $\commaES$-ES where $\lambda = p$ and $\mu$ equals the number of non-zero weights.

\paragraph{Invariance properties} The two different \cprs algorithms presented in this section are translation invariant and scale-invariant. They indeed satisfy the sufficient conditions derived in Proposition~\ref{prop:TI} and Proposition~\ref{lem:scaleinvariant}.

\subsection{Evolution Strategy with Self-adaptation}\label{app:sec:sa} 

Another type of algorithms included in the \cprs definition are the so-called self-adaptive step-size ES. The idea of self-adaptation dates back from the 70's and consists in adding the parameters to be adapted (step-size, covariance matrix, ...) to the vector that undergoes variations (mutations and recombinations) and let the selection (through the ordering function) adjusts the parameters \cite{Rechenberg,Schwefel:77}. In the case where one single step-size is adapted, the step-size  undergoes first a mutation: it is multiplied by a random variable following a log-normal distribution $\logn(0,\tau^{2})$ where $\tau \approx {1}/{\sqrt{n}}$. The mutated step-size is then used as overall standard deviation for the multivariate normal distribution $\Normal$. In this case, the space $\Uspace$ equals $\R^{(n+1) \times p}$. The $n$ first coordinates of an element $\Utt^{i} \in \UUU=\R^{n+1}$ denoted $[\Utt^{i}]_{1\ldots n}$ 
 ($\in \R^{n}$) correspond to the sampled standard multivariate normal distribution vector and the last coordinate denoted $[\Utt^{i}]_{n+1}$ to the sampled normal distribution for sampling the log-normal distribution used to mutate the step-size. The solution function is defined as
\begin{equation}\label{app:sample:SA}
\Sol((\Xt,\st),\Utt^{i})= \Xtt^{i} = \Xt + \st \exp \left( \tau [\Utt^{i}]_{n+1} \right) [\Utt^{i}]_{1\ldots n}
\end{equation}
where $[\Utt^{i}]_{1\ldots n} \sim \Normal$ and $[\Utt^{i}]_{n+1} \sim \N(0,1)$. The distribution $p_{\U}$ admits thus a density that equals 
$
p_{\U}(\uu^{1},\ldots,\uu^{p})=p_{\N}(\uu^{1})   \ldots p_{\N}(\uu^{p}) \,, \uu^{i} \in \R^{n+1}
$.
Remark that the ordering function selects the couple multivariate normal distribution and log-normal distribution used to mutate the step-size at the same time. Assuming that only the best solution plays a role in the update of $\Xt$ (i.e.\ it corresponds to a single non-zero weight in the recombination equation \eqref{app:eq:commamean}), the update for the mean vector reads
\begin{equation}\label{app:mean:SA}
\Xtt = \Xt + \st \exp( \tau [\Ytt^{1}]_{n+1}) [\Ytt^{1}]_{1 \ldots n}
\end{equation}
and the update for the step-size is
\begin{equation}\label{app:ss:SA}
\stt = \st \exp ( \tau  [\Ytt^{1}]_{n+1} ) \enspace.
\end{equation}
A step-size adaptive Evolution Strategy satisfying \eqref{app:sample:SA},\eqref{app:mean:SA} and \eqref{app:ss:SA} is called $(1,p)$ self-adaptive step-size ES ($(1,p)$-SA). The $(1,p)$ refers to the fact that a single solution is selected out of the $p$.
The update function $\GG$ for the $(1,p)$-SA reads
$$
\GG_{(1,p){\rm-SA}}((\x,\sigma),\y) = \left( \begin{smallmatrix} \x + \sigma \exp(\tau [\y^{1}]_{n+1}) [\y^{1}]_{1\ldots n} \\ 
\sigma \exp( \tau [\y^{1}]_{n+1} )
 \end{smallmatrix} \right)  \enspace.
$$
We see thus that the step-size is adapted by the selection that occurs through the ordering.
The rationale behind the method being that unadapted step-size cannot successfully give good solutions and that selection will adapt (for free) the step-size (explaining thus the terminology ``self-adaptation''). Self-adaptive algorithms have been popular in the 90's certainly due to the fact that their underlying idea is simple and attractive. However self-adaptation has shortcomings that were explained and discussed previously in \cite{hansen2006ecj,hansen-ppsn:2014}. Different variants of self-adaptation using multiple parents and recombinations exist, we refer to the review paper \cite{beyer2002} for further readings and references.

\paragraph{Invariances} In virtue of Proposition~\ref{prop:TI} and Proposition~\ref{lem:scaleinvariant} the $(1,p)$-SA is translation and scale-invariant.

The linear convergence of the self-adaptive ES algorithm described in this section in dimension $1$ was proven in \cite{TCSAnne04} on spherical functions using the Markov chain approach presented here. 

\subsection{Step-size Random Search or Compound Random Search or (1+1)-ES with $1/5$ Success Rule}\label{app:sec:algo-plus}

The last example presented is an algorithm where the sequence $f(\Xt)$ is decreasing, i.e.\ updates that only improve or leave $\Xt$ unchanged are performed. \hide{The algorithm is often called of \emph{elitist} as the best solution at a current iteration is kept for the next iteration. For the example presented,  at} At each iteration a single new solution is sampled from $\Xt$, i.e.\
\begin{equation*}
\Xtt^{1} = \Xt + \st \Utt^{1}
\end{equation*}
where $\Utt^{1} \in \R^{\dim}$ follows a standard multivariate normal distribution, and hence $\Xtt^{1}$ follows the distribution $\N(\Xt,\st^{2} \Id)$. The step $\Utt^{1}$ is accepted if the candidate solution is better than the current one and rejected otherwise. Let us denote $\Utt^{2}=0 \in \R^{n}$ the zero vector and take $\Utt=(\Utt^{1},\Utt^{2})$.\niko{$\Ut^{1}=0$ would generalize better to the $(1+\lambda)$-ES} Hence $\Uspace = \R^{n \times 2 }$ and the probability distribution of $\U$ equals $p_{\U}(\uu^{1},\uu^{2})=p_{\N}(\uu^{1})  \delta_{0} (\uu^{2})$ where $\delta_{0}$ is the Dirac delta function. The $\Sol$ function corresponds then to the function in \eqref{app:eq:mut}.

The update equation for $\Xt$ is similar to \eqref{app:eq:commamean} with weights $(w_{1},w_{2})=(1,0)$.  
Remark that contrary to the \hide{non-elitist} algorithms presented before, the sampled step $\Utt$, the selected step $\Ytt$ and the new mean $\Xtt$ have a singular part w.r.t. the Lebesgue measure. 
An algorithm following such an update is often referred to as $(1+1)$-ES but was also introduced under the name Markov monotonous search \cite{Zhigljavsky:2008}, step-size random search \cite{Schumer:68} or compound random search \cite{Devroye:72}. 

The adaptation of the step-size idea starts from the observation that if the step-size is very small, the probability of success (i.e.\ to sample a better solution) is approximately one-half but the improvements are small because the step is small. On the opposite if the step-size is too large, the probability of success will be small, typically the optimum will be overshoot and the improvement will also be very small. In between lies an optimal step-size associated to an optimal probability of success \cite{Schumer:68, Rechenberg,Devroye:72}. A proposed adaptive step-size algorithm consists in trying to maintain a probability of success (i.e.\ probability to sample a better solution) to a certain target value $p_{\rm target}$, increase the step-size in case the probability of success is larger than $\ptarget$ and decrease it otherwise \cite{Devroye:72,Rechenberg,Rechenberg:94}.  The optimal probability of success, i.e.\ allowing to obtain an optimal convergence rate has been computed  on the sphere function $f(\x)= \| \x \|^{2}$ for dimension of the search problem going to infinity and is roughly equal to $0.27$ \cite{Schumer:68, Rechenberg}.  Another function where the asymptotic optimal probability of success was computed is the corridor function\footnote{The corridor function is defined as $f(\x) = \x_{1}$ for $- b < \x_{2} < b, \ldots -b  < \x_{n} < b$, for $b >0$ otherwise $+ \infty$.} where it is equal to $1/(2e)$  \cite{Rechenberg:94}. As a trade-off between the probability of success on the sphere and on the corridor, the target probability  is often taken equal to $1/5=0.20$ and gave the name one-fifth success rule to the step-size adaptive algorithm. We call the algorithm with $p_{\rm target}$ as target success probability the \emph{generalized one-fifth success rule}.\footnote{Note that $p_{\rm target}$ does not correspond to the optimal probability of success as indeed if the probability of success equals the target probability, the step-size is kept constant. Hence if convergence occurs the achieved probability of success is smaller than the target probability. Therefore, on the sphere, if convergence occurs, $p_{\rm target}=0.20$ corresponds to an achieved probability of success smaller  than $0.20$, hence a probability of success smaller than optimal which will consequently favor larger step-sizes as the probability of success decreases with increasing step-sizes \cite{companion-oneplusone}.}

Several implementations of the generalized one-fifth success rule exist. In some implementations, the probability of success is estimated by fixing a step-size for a few iterations, counting the number of successful solutions and deducing an estimation of the probability of success. The step-size is then increased if the probability of success is larger than $p_{\rm target}$ and decreased otherwise \cite{Rechenberg,Rechenberg:94}. \hide{This version of the one-fifth success rule is in particular the one investigated theoretically by J\"agersk\"upper \cite{jens-hit-and-run,Jens:2007,jens:gecco:2006,jens:tcs:2006}.} A somehow simpler implementation consists in estimating at each iteration the probability of success as  $1_{\{ f(\Xtt^1) < f(\Xt) \}}=1_{\{\Ytt^{1} \neq 0\}}$\footnote{This equality is true only almost everywhere.}: this indicator function being equal to one in case of success and zero otherwise. Consequently the algorithm will increase the step-size after a successful step and decrease it otherwise as proposed in \cite{Devroye:72,CMA-EDA:ICALP2003}.\niko{FTR: This was indeed not proposed by Rechenberg, IIRC. }
The update rule for the step-size reads 
\begin{align}\label{one-fifth-update}
\stt & = \st \exp \left( \LRsigma \frac{1_{\{ \Ytt^{1} \neq 0 \}} - \ptarget }{1 - \ptarget}  \right) 
\end{align}
where $\LRsigma >0$ is a learning rate coefficient. Denoting $\gamma = \exp(\LRsigma)$ and the target odds ratio $q=\frac{\ptarget}{1-\ptarget}$ (for a target success probability set to $1/5$, the odds ratio $q=1/4$) yields
\begin{equation}\label{app:eq:ssOPO}
\stt = \st \left( \factonefifth 1_{\{ \Ytt \neq 0 \}} + \factonefifth ^{-q}  1_{\{ \Ytt^{1} = 0 \}}\right) = \st \left( (\factonefifth - \factonefifth ^{-q}) 1_{\{ \Ytt^{1} \neq 0 \}} + \factonefifth ^{-q}  \right) \enspace.
\end{equation}
Overall, the update transformation for the $(1+1)$-ES with generalized one-fifth success rule is
$$
\GG_{\plusES_{\onefifth}}((\x,\sigma),\y) = \left( \begin{smallmatrix}   \x + \sigma  \y^{1} \\ 
\sigma  
\left( (\factonefifth - \factonefifth ^{-q}) 1_{\{ \y^{1} \neq 0 \}} + \factonefifth ^{-q}  \right)
\end{smallmatrix} \right)  \enspace.
 $$
In such an algorithm, the best solution cannot be forgotten. Consequently in some noisy settings, the algorithm can get stuck with solutions that are suboptimal because of realizations \emph{of the noise} leading to particularly small (i.e.\ good) function values\niko{whether they are smaller than this one is not overly relevant, in particular as also the optimum can realize smaller values than its true value} (see \cite{jah:2009a} for instance). Consequently, the CMA-ES is implementing an update of $\Xt$ where the best solution is not preserved from one iteration to the next one.
\hide{Elitist selection is not robust to outliers and presence of noise on the objective function explaining why the state-of-the-art method CMA-ES is using a comma selection. However, elitist algorithms are theoretically interesting and the fact that the objective function value of the best solution at a given iteration decreases renders theoretical proofs often easier. This most certainly explicates why the (1+1) algorithm is popular among theoreticians.} 

\paragraph{Invariance} Using again Proposition~\ref{prop:TI} and Proposition~\ref{lem:scaleinvariant}, the $(1+1)$-ES with generalized one-fifth success rule is translation and scale-invariant.

\begin{remark}
In all the examples presented, the $p$ components $(\Utt^{i})_{1 \leq i \leq p}$ of the vectors $\Utt$ are independent. It is however not a requirement of our theoretical setting. \hide{Some algorithms using within an iteration non independent samples were recently introduced \cite{abh2011b,abh2011a} and could be analyzed with the Markov chain approach presented in this paper.}
\end{remark}

\section{Numerical Experiments}\label{sec:AppNE}
We present in this appendix some numerical experiments of the different step-size adaptive randomized search algorithms that were described within the paper, namely the $(1+1)$-ES with one-fifth success rule presented in Section~\ref{app:sec:algo-plus}, the $(1,p)$-ES with self-adaptive mutation presented in Section~\ref{app:sec:sa}, the exponential natural evolution strategy with covariance matrix adaptation switched off presented in Section~\ref{app:sec:algo-comma}, the cumulative step-size adaptation presented in Section~\ref{app:sec:algo-comma}. 
None of these randomized algorithms is state-of-the art because they solely adapt a step-size while the adaptation of the full covariance matrix of the sampling distribution is known to be crucial. We hence also tested the state-of-the art CMA-ES algorithm that combines step-size and covariance matrix adaptation \cite{hansen2001}. For the sake of comparison, we also tested the Nelder-Mead  and the Random Pursuit (RP) algorithms from \cite{stich:2013}. The different parameters or the implementation used are specified in Table~\ref{param-expe}. We refer to \cite{hansen:2014} and \cite{brockhoff:2010} for illustrations of the dependency of the convergence rates in the damping and cumulation parameters.

\begin{figure}
\begin{center}
\begin{tabular}{l|p{0.7\textwidth}}
  $(1+1)_{\onefifth}$-ES
 & $ \LRsigma= 1/ \sqrt{n+1}$ ; $\ptarget = 1/5$ (see \eqref{one-fifth-update}) \\
 $(1,p)$-SA  & $p=\lfloor4 + 3 \log(n) \rfloor$; $\tau = 1/\sqrt{n}$ (see \eqref{app:sample:SA}) \\
 xNES1  & parameters taken from \cite{Glasmachers2010} (except $\LRsigma$) with covariance matrix adaptation turned off ;   $\LRsigma=2  \frac{3}{5} \frac{3 + \log n}{n \sqrt{n}}$ \footnote{give better results than the default one.}\\
CSA-ES & default step-size mechanism of CMA-ES (Eq.~\eqref{eq:path}, \eqref{app:eq:commamean}, \eqref{app:eq:ssDSA} with the term $\sqrt{\mueff }  \| \sum_{i=1}^{p} w_{i} \Ytt^{i} \| $ replaced by $\| \p_{t+1} \|$). $p=\lfloor4 + 3 \log(n) \rfloor$; $w_i =  \left(\log((p+1)/2) - \log(i) \right) \vee 0$;  $c = \sqrt{\mueff + 2}/(\sqrt{n}+ \sqrt{\mueff+3}) $, $\LRm = 1$, $\LRsigma=1$, $\p_0=\mathbf{0}$\\
CSA-ES1 &  same as CSA-ES except $c=1$\\
  CMA-ES &  Python code version 1.1.06 available \\
  &  \url{https://pypi.python.org/pypi/cma/1.1.06} \\
  RP & Matlab implementation, use of  {\tt fminunc} for the line search with the option set to {\tt optimset('Display', 'off', 'LargeScale', 'off', 'TolX', 10\^\,\!(-10), 'TolFun', 10\^\,\!(-12))} (implementation of \cite{stich:2013} with different stopping criterion for the line search) \\
  Nelder-Mead & function \texttt{scipy.optimize.fmin} from the scipy Python library (version 0.16.0)\anne{import scipy ; 
scipy.\_\_version\_\_ }
\end{tabular}
\caption{\label{param-expe} List of algorithms and their parameters used for the experiments.}
\end{center}
\end{figure}

\newcommand{\fsphere}{\ensuremath{f_{\rm sphere}}}
\newcommand{\fpnorm}{\ensuremath{f_{\rm pnorm}}}
We experimented the algorithms on the four scaling-invariant functions presented in Table~\ref{def-function}. The convex quadratic functions $f_{\rm sphere}$, $f_{\rm elli}$ and the function $f_{\rm pnorm}$ for $p=2$ have convex sublevel sets in contrast to $f_{\rm pnorm} $ for $p=1/2$ (see Figure~\ref{fig:scalinginvariant}).
The function $f_{\rm elli}$ is ill-conditioned (the condition number of its Hessian matrix is $10^6$). The functions $f_{\rm sphere}$ and  $f_{\rm pnorm} $ for $p=2$ have the same level sets.
\begin{figure}
\begin{center}
\begin{tabular}{ll}
$f_{\rm sphere}(\x) = \sum_{i=1}^n \x_i^2$ &     $f_{\rm elli}(\x) = \sum_{i=1}^n (10^6)^{\frac{i-1}{n-1}} \x_i^2$ \\ $\fpnorm(\x)= ( \sum_{i=1}^n |\x_i|^p)^{1/p} $ & for $p=2$ and $p=0.5$ \\
\end{tabular}
\caption{\label{def-function} Definition of the functions used within the experiments.}
\end{center}
\end{figure}
Each algorithm has been tested on the four test functions using as starting point the vector $(1,\ldots,1)$ except Nelder-Mead---the only deterministic algorithm---using a random starting point sampled according to $(1,\ldots,1) + 0.1\ \Normal$. The initial step-size for the step-size adaptive algorithms and CMA-ES has been set to $1$.

\begin{figure}
\centering
\includegraphics[width=0.24\textwidth]{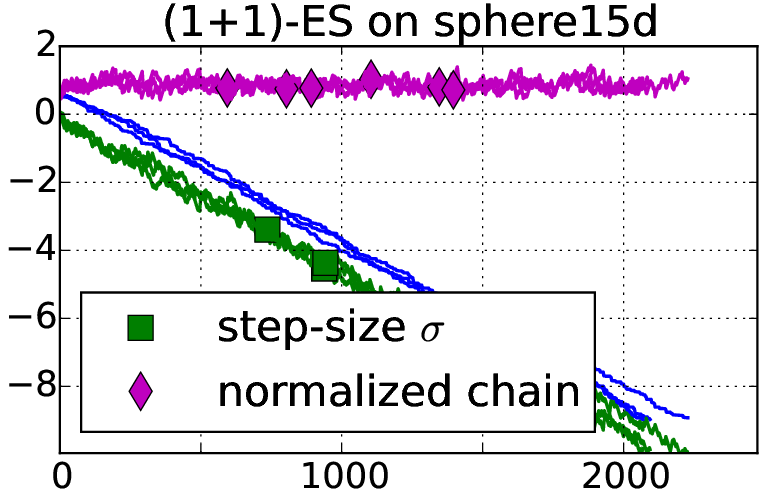}
\includegraphics[width=0.24\textwidth]{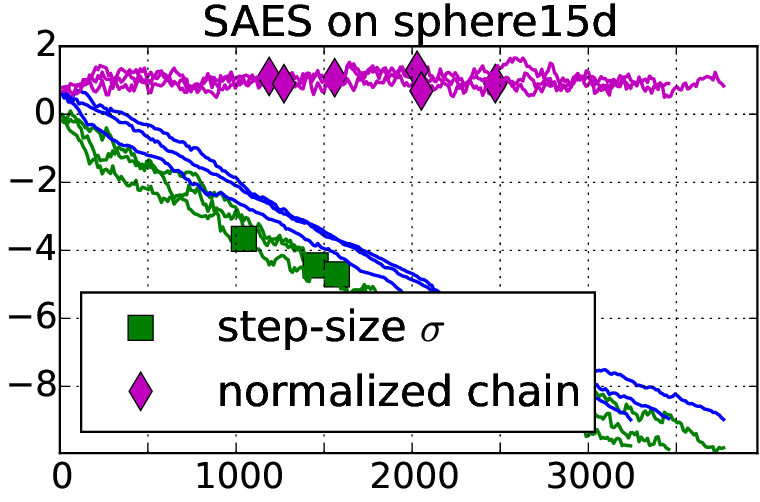}
\includegraphics[width=0.24\textwidth]{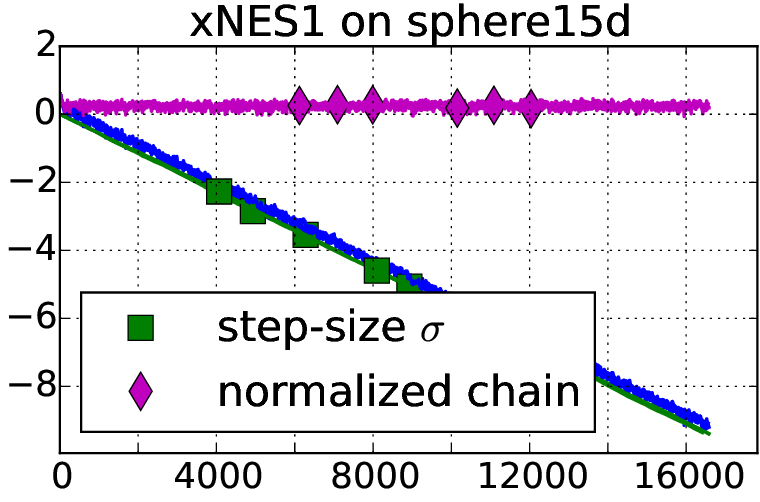}
\includegraphics[width=0.24\textwidth]{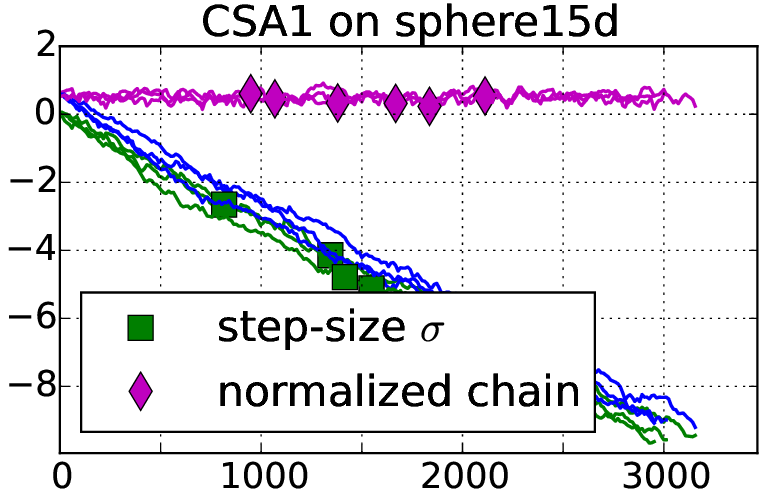}
\includegraphics[width=0.24\textwidth]{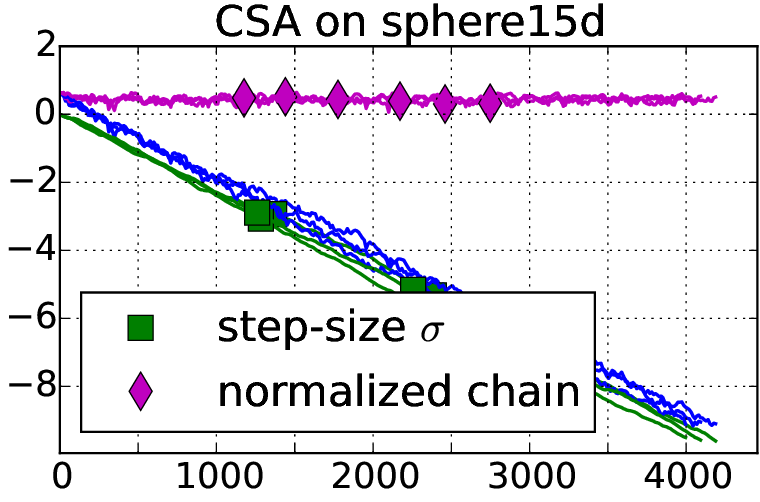}
\includegraphics[width=0.24\textwidth]{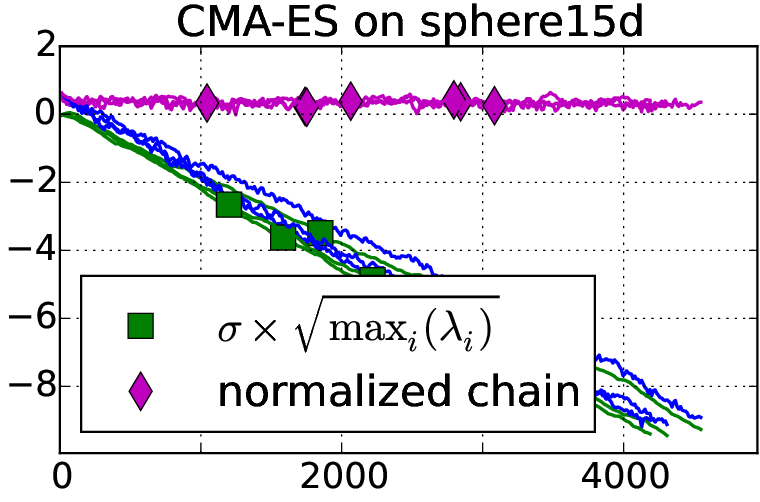}
\includegraphics[width=0.24\textwidth]{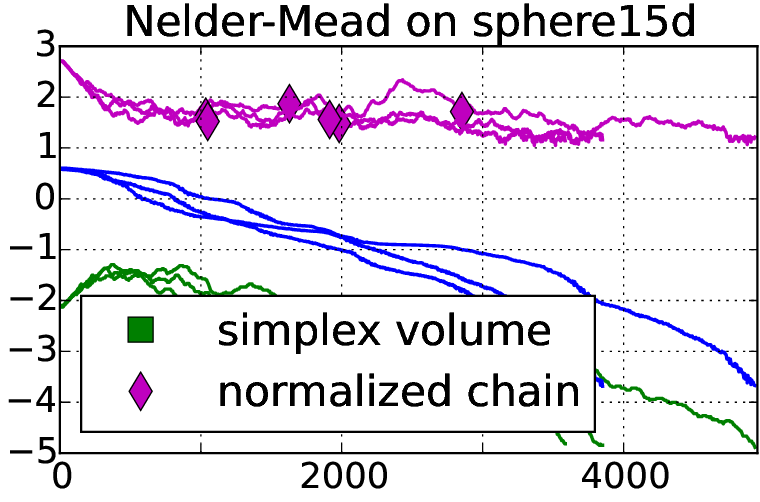}
\includegraphics[width=0.24\textwidth]{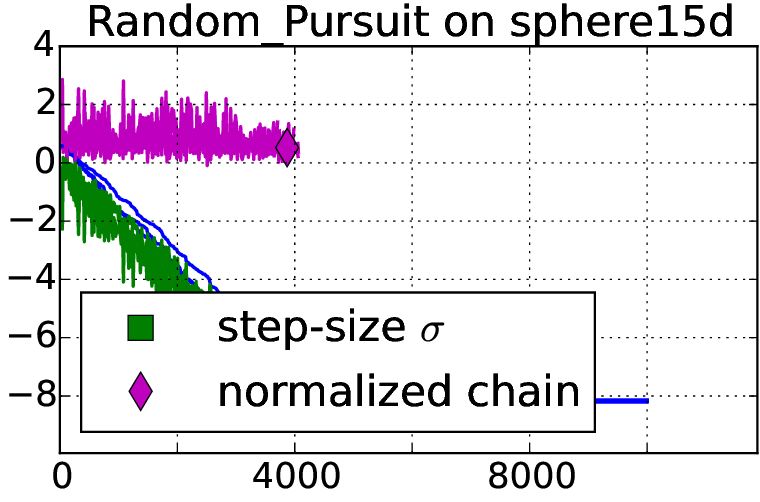}
\includegraphics[width=0.24\textwidth]{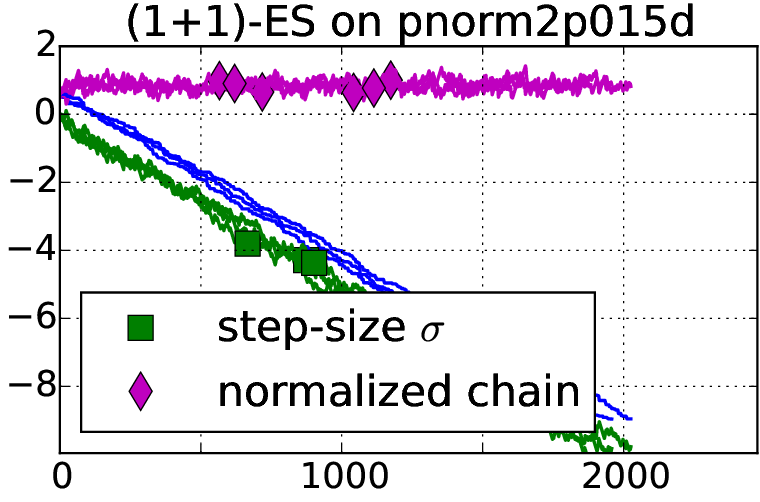}
\includegraphics[width=0.24\textwidth]{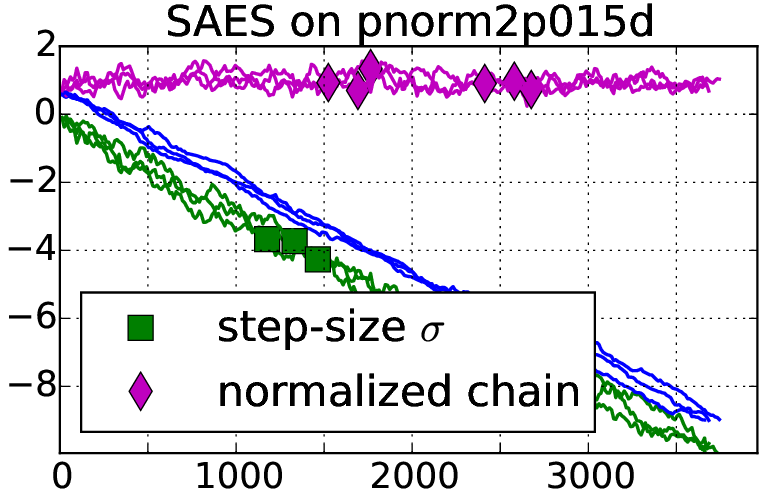}
\includegraphics[width=0.24\textwidth]{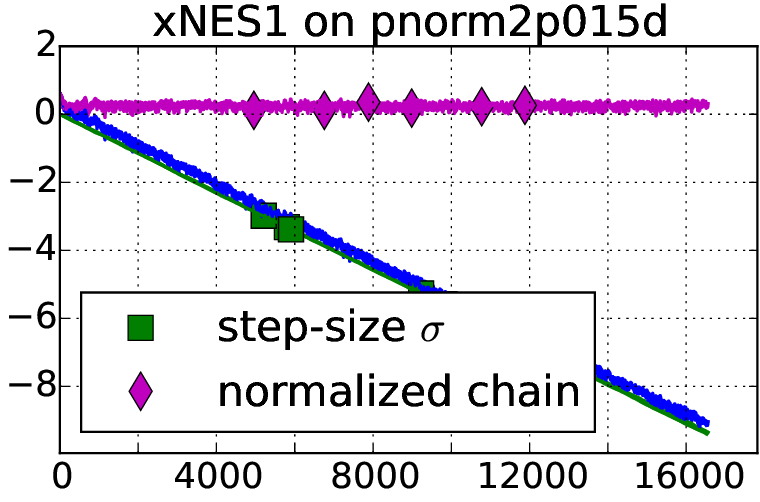}
\includegraphics[width=0.24\textwidth]{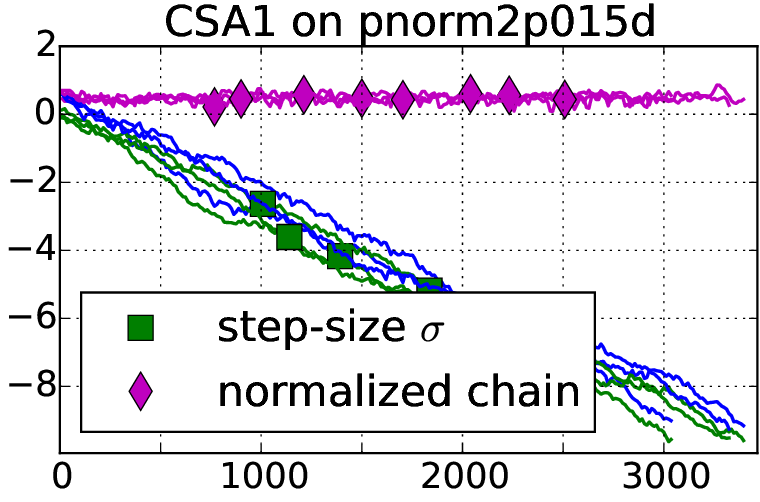}
\includegraphics[width=0.24\textwidth]{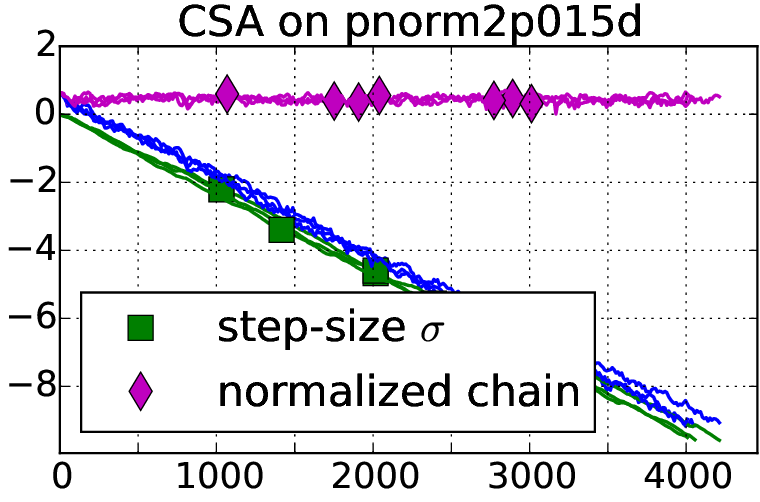}
\includegraphics[width=0.24\textwidth]{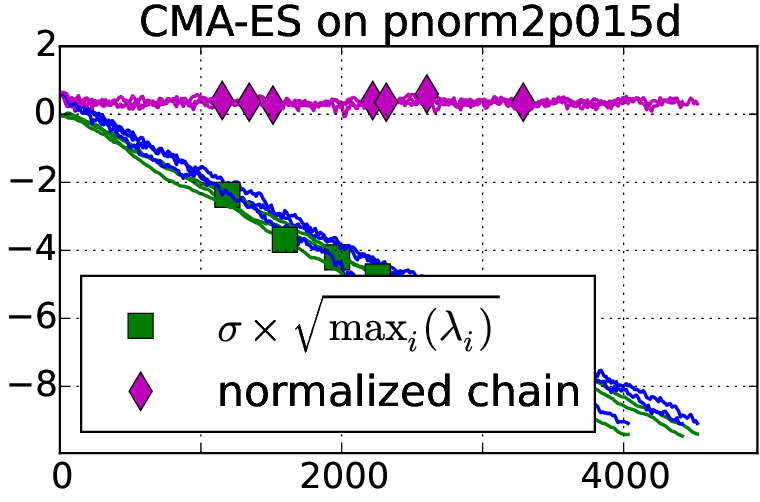}
\includegraphics[width=0.24\textwidth]{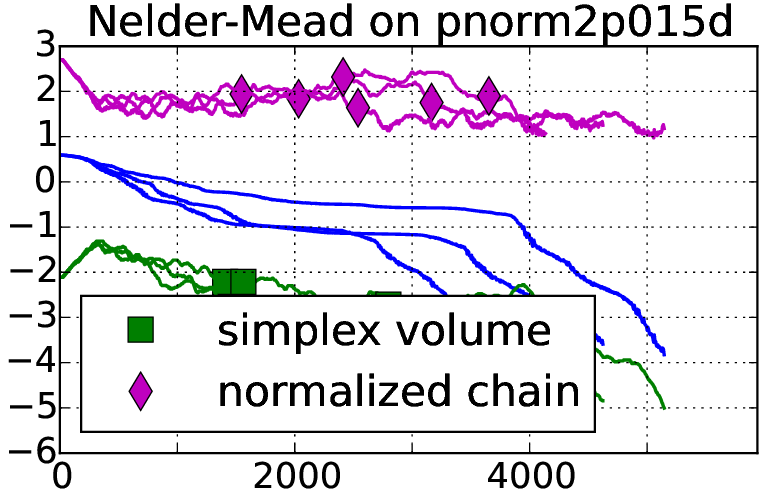}
\includegraphics[width=0.24\textwidth]{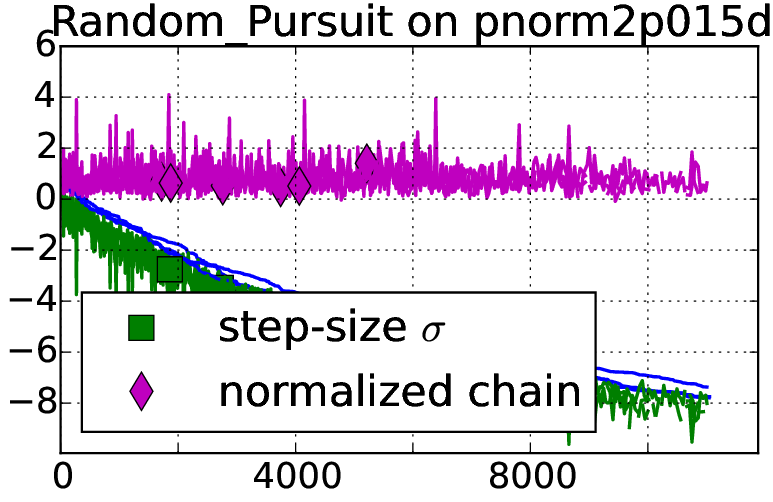}
\includegraphics[width=0.24\textwidth]{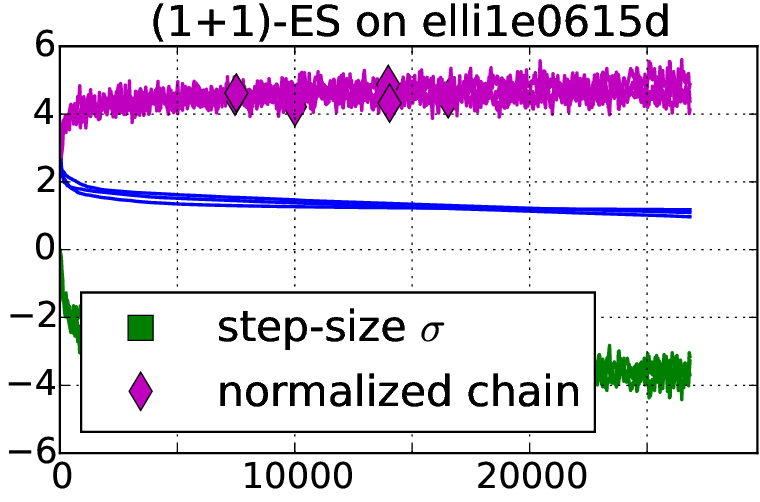}
\includegraphics[width=0.24\textwidth]{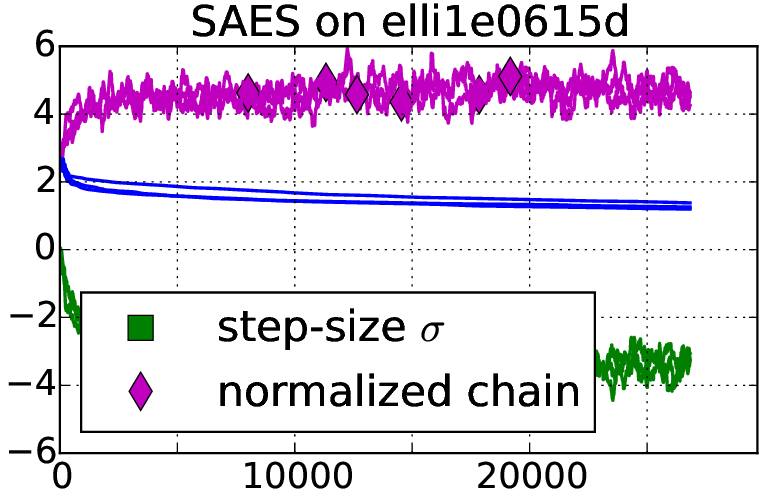}
\includegraphics[width=0.24\textwidth]{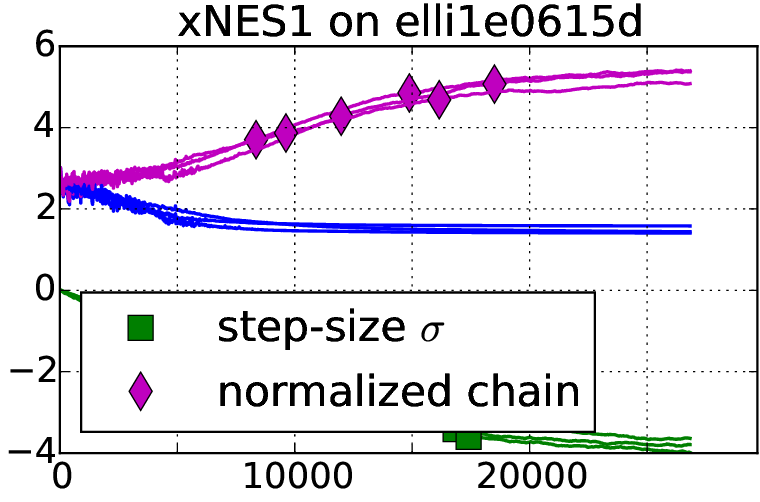}
\includegraphics[width=0.24\textwidth]{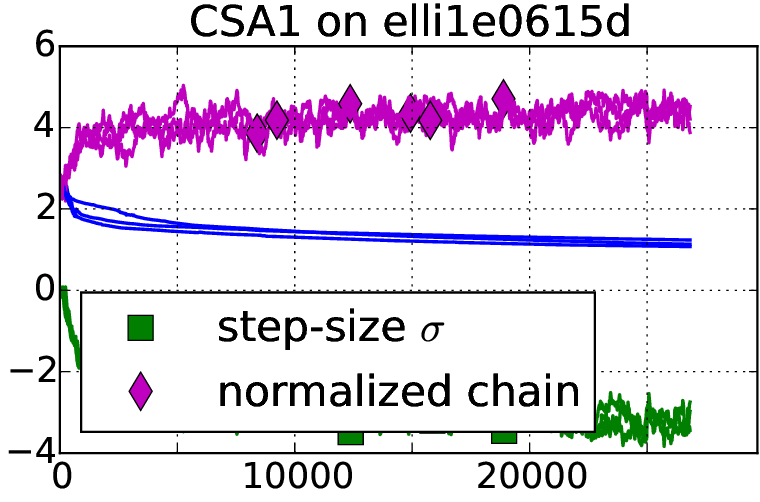}
\includegraphics[width=0.24\textwidth]{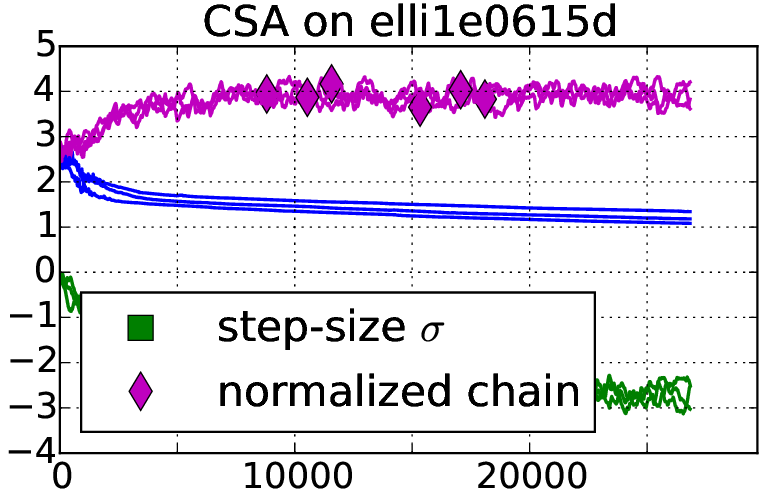}
\includegraphics[width=0.24\textwidth]{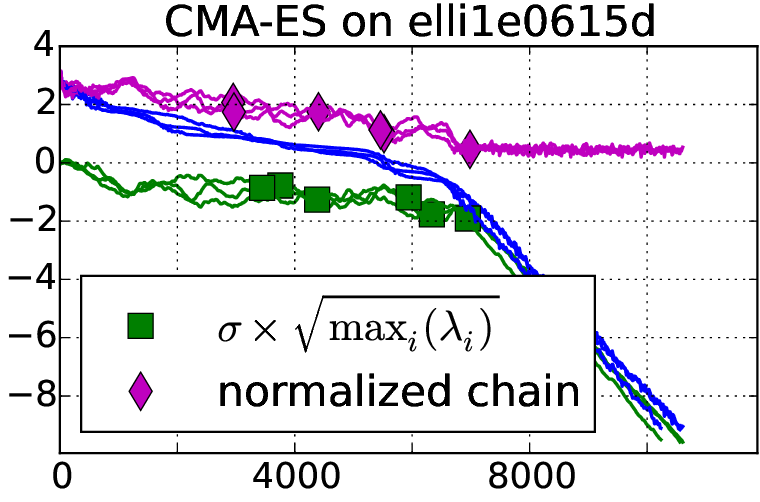}
\includegraphics[width=0.24\textwidth]{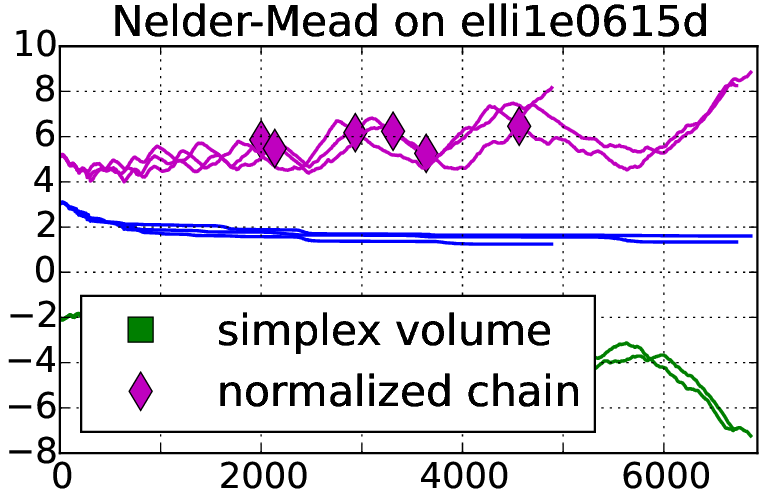}
\includegraphics[width=0.24\textwidth]{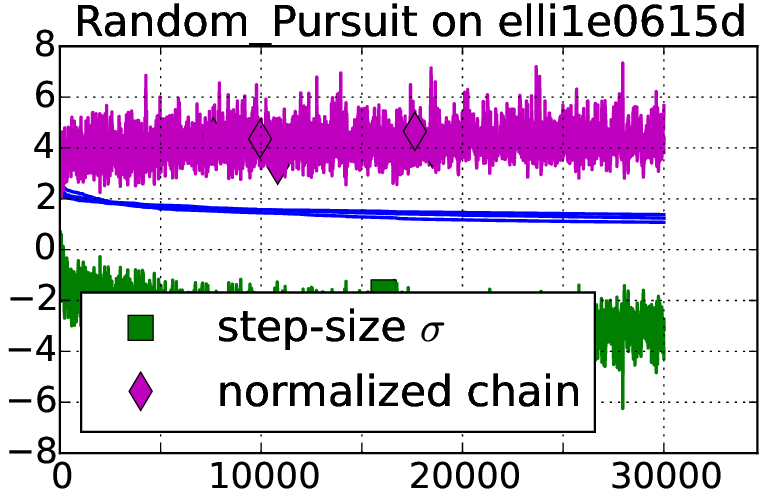}
\includegraphics[width=0.24\textwidth]{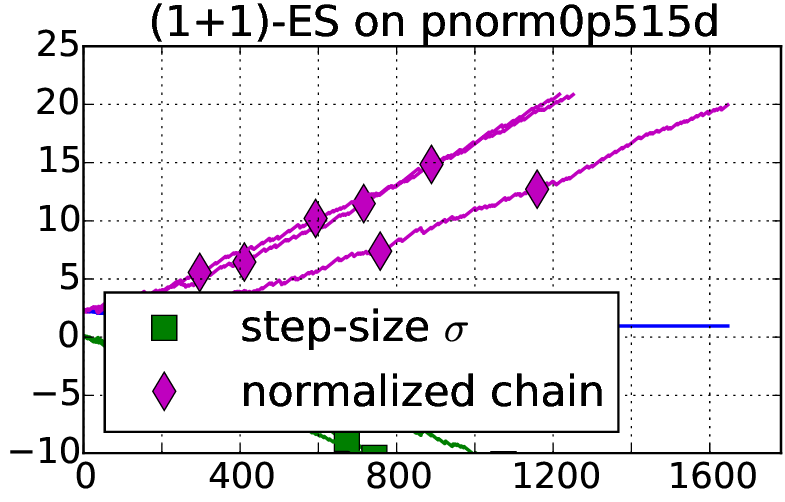}
\includegraphics[width=0.24\textwidth]{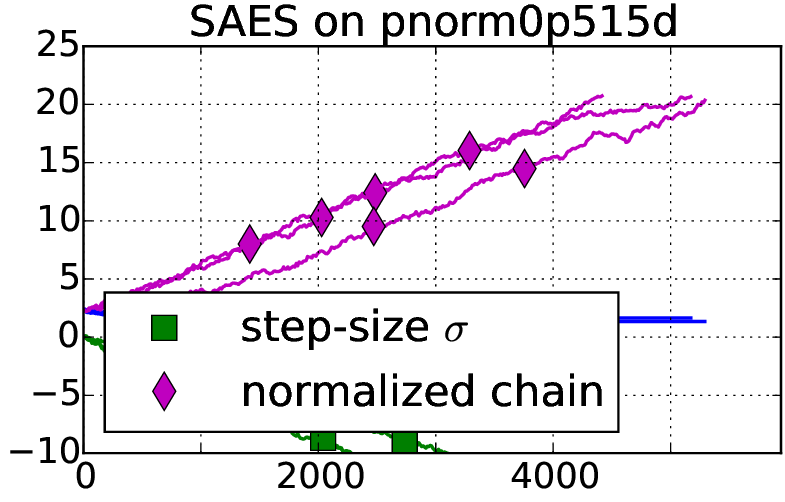}
\includegraphics[width=0.24\textwidth]{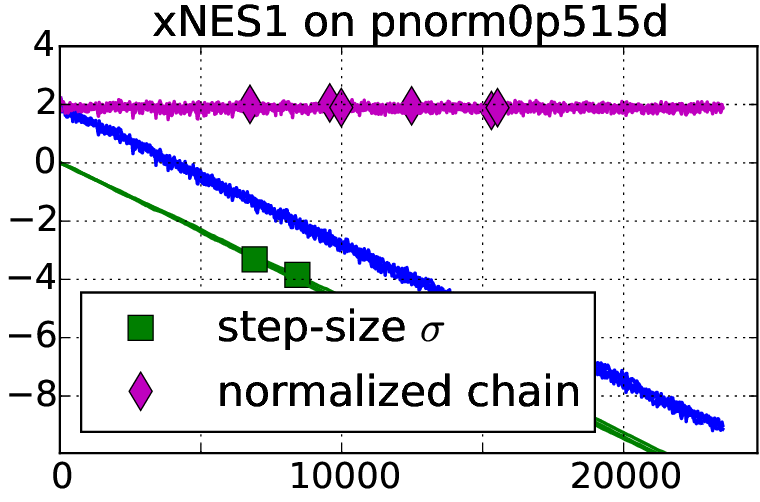}
\includegraphics[width=0.24\textwidth]{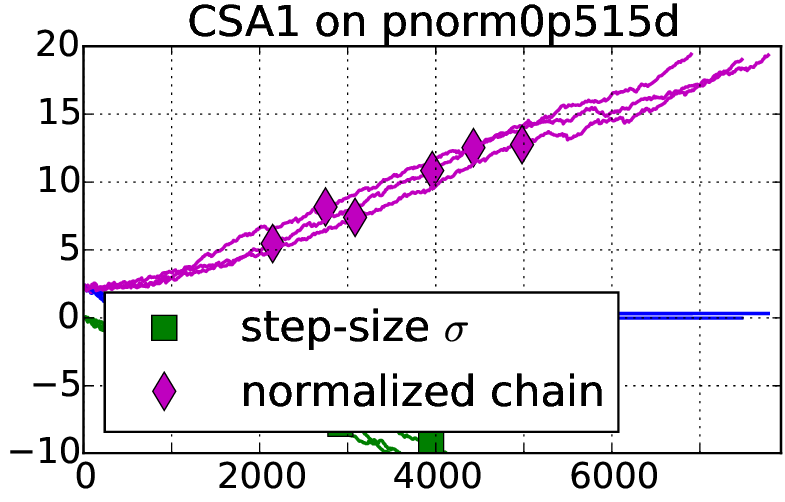}
\includegraphics[width=0.24\textwidth]{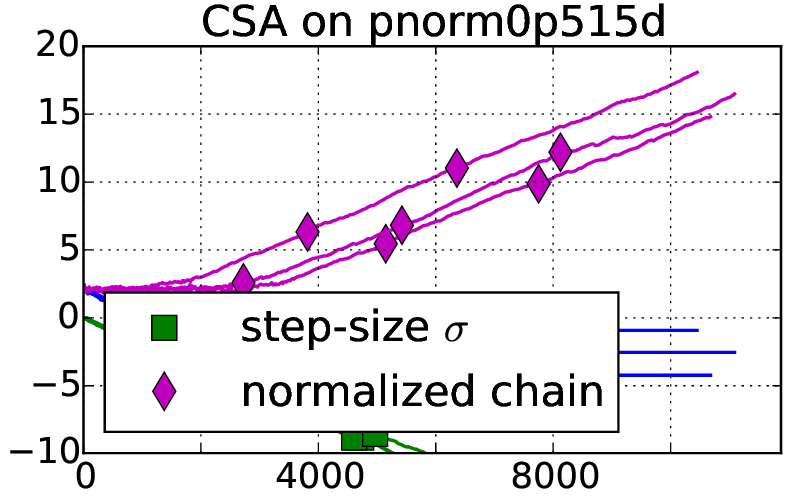}
\includegraphics[width=0.24\textwidth]{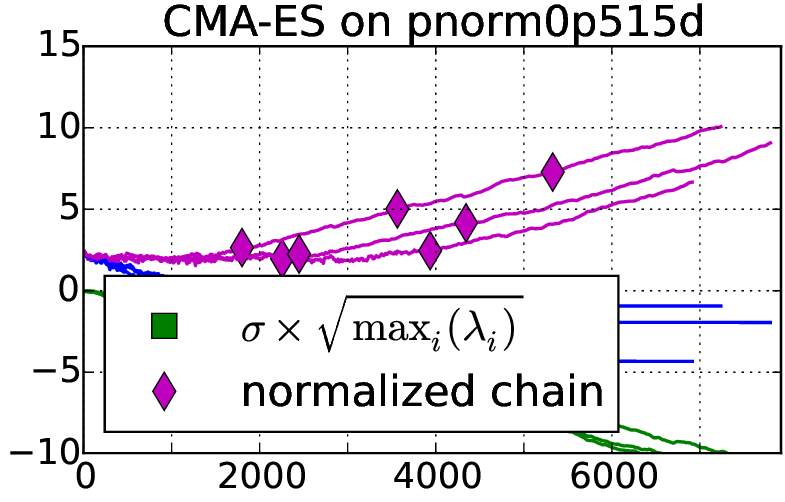}
\includegraphics[width=0.24\textwidth]{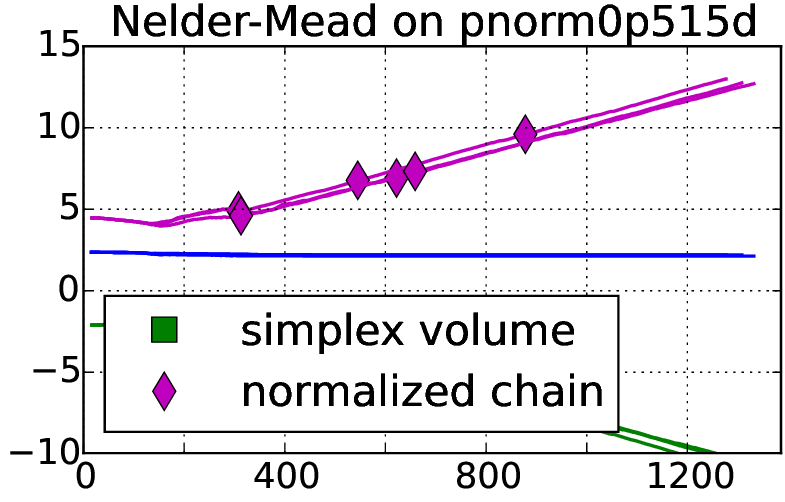}
\includegraphics[width=0.24\textwidth]{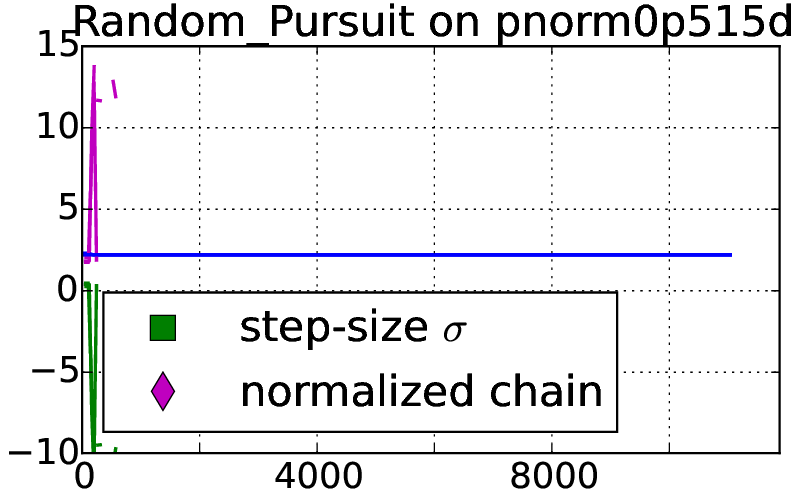}
\caption{\label{fig:single-runs} Single-runs in $15 D$. On the $x$-axis the number of function evaluations is displayed. A log scale is used for the $y$-axis whose numbers have to be read as $10^y$ where $y$ is the number \new{displayed}\del{to be read of the $y$-axis}. Curves with a square marker display the evolution of the step-size for all algorithms but CMA-ES where it is the evolution of the step-size time the square root of the maximal eigenvalue of the covariance matrix and the Nelder-Mead algorithm where the simplex volume is displayed. Curves without a marker display either the square root of the objective function for $f_{\rm sphere}$, $f_{\rm elli}$ or the objective function. Curves with a diamond display the square root of the objective function (resp.\ objective function) of the normalized chain for $f_{\rm sphere}$ and $f_{\rm elli}$ (resp.\ the p-norm functions).}
\end{figure}
We present in Figure~\ref{fig:single-runs} three independent runs of each algorithm on the functions presented in Table~\ref{def-function} for dimension $15$. The number of function evaluations is displayed on the x-axis. The lines without markers display, for $f_{\rm sphere}$ and $f_{\rm elli}$, the evolution of the square root of the objective function value of the incumbent $\Xt$ and, for the $f_{\rm pnorm}$ functions, the objective function value. The lines with a square marker display for $(1+1)_{\onefifth}$-ES, xNES1, CSA-ES, CSA-ES1, RP the step-size, for CMA-ES the step-size times the square root of the maximum eigenvalue of the covariance matrix and for Nelder-Mead the simplex volume. Lines with a diamond marker display the former line divided by the latter corresponding to $f(\Xt/\st)^{1/2}$ for  $f_{\rm sphere}$ and $f_{\rm elli}$ and $f(\Xt/\st)$ for the p-norm functions.
\niko{to self: xNES on the p1/2-norm is only "stable", because the initial step-size is large enough. }\del{
We observe }
\begin{itemize}
\item On the sphere function \new{we observe} fast linear convergence and a stable normalized Markov chain for all stochastic algorithms. 
\item \new{We observe the same behavior as on the sphere}\del{The same picture} for all stochastic algorithms except for RP on $f_{\rm pnorm}$ for $p=2$, because \fpnorm\ has the same level sets as \fsphere\ and the algorithms are ranked-based. RP is slower on the pnorm for $p=2$ than on the sphere; it needs roughly two times more function evaluations to reach a f-value of $10^{-8}$.
\item On $f_{\rm elli}$ \new{we observe} much slower convergence for all step-size adaptive algorithms and Nelder-Mead. The CMA-ES algorithm is the fastest, and we observe two stages: until $6500$ function evaluations, linear convergence but relatively slow compare to the second stage where the same convergence speed as on the sphere is observed. Nelder-Mead does not exhibit stable convergent behavior. Together with the invariance properties of the algorithm, this result suggests that Nelder-Mead is not\del{ be} a ``stable" algorithm even on the sphere function.\niko{given the definition of stable demands "for all initial states". }
\item On $f_{\rm pnorm}$ for $p=1/2$, all but xNES exhibit an unstable normalized chain (diamond line) and premature convergence to a non-optimal point. The apparent success of xNES hinges on two settings: a large enough initial $\sigma_0$ and a small enough learning rate $\kappa_\sigma$. Accordingly, if $\kappa_\sigma$ is chosen small enough, also the CSA and CMA-variants exhibit similar behavior (not shown). However, none of the algorithms is stable independently of the initial values for $\x$ and $\sigma$. 
\end{itemize}
\anne{Kind of conclusion for this series of plots:}
\begin{itemize}
\item \del{On the different graphs we see an equivalence between stability of the normalized Markov chain and the linear convergence.}\niko{I would not make any such statement about equivalence. The truth is much more intricate (see xNES or Nelder-Mead) and without a wide exploration of parameter settings \emph{and} a deep understanding of the involved mechanisms, such set of simulation cannot give any more than some hints. } The simulations illustrate the relationship between stability of the normalized MC and linear convergence. However, without a wide exploration of parameter settings \emph{and} a deep understanding of the involved mechanisms, such set of simulations only provide some hints as to whether stability is achieved. Simulations can also be difficult to interpret correctly (see Nelder-Mead on the sphere that could be interpreted as stable whereas it is most likely not if we consider the simulations on the ellipsoid function). They hence cannot replace a theoretical proof.
\item For step-size adaptive randomized search algorithms and CMA-ES, stability is observed on three out of four scaling invariant functions (i.e.\ not for $f_{\rm pnorm}$ and $p=1/2$).
\new{Unsurprisingly}\del{Interestingly},\niko{to me it would be rather unsurprisingly than interestingly} this observation also corresponds to the sufficient condition for stability proven for the $(1+1)_{\onefifth}$-ES. Indeed we have  been able to prove the stability (and hence linear convergence of the algorithm) on a specific class of scaling-invariant functions, namely positively homogeneous functions that additionally satisfy some regularity assumptions like $f_{\rm sphere}$, $f_{\rm elli}$ and $f_{\rm pnorm}$ for $p=2$ but not $f_{\rm pnorm}$ for $p=1/2$ \cite{companion-oneplusone}. \new{Our}\del{Those} simulations backup the intuition that additional assumptions on the regularity of the level sets are needed to be able to prove the stability of the Markov chains.
\item In order to achieve reasonable convergence rates on the ill-conditioned $f_{\rm elli}$ function, we see that it is crucial to adapt the covariance matrix together with the step-size. The CMA-ES algorithm adapts the underlying metric by adapting the different parameters of the covariance matrix  such that in the end the ill-conditioned function is transformed into the sphere function.
\end{itemize}

\begin{figure}
\centering
\includegraphics[width=0.49\textwidth]{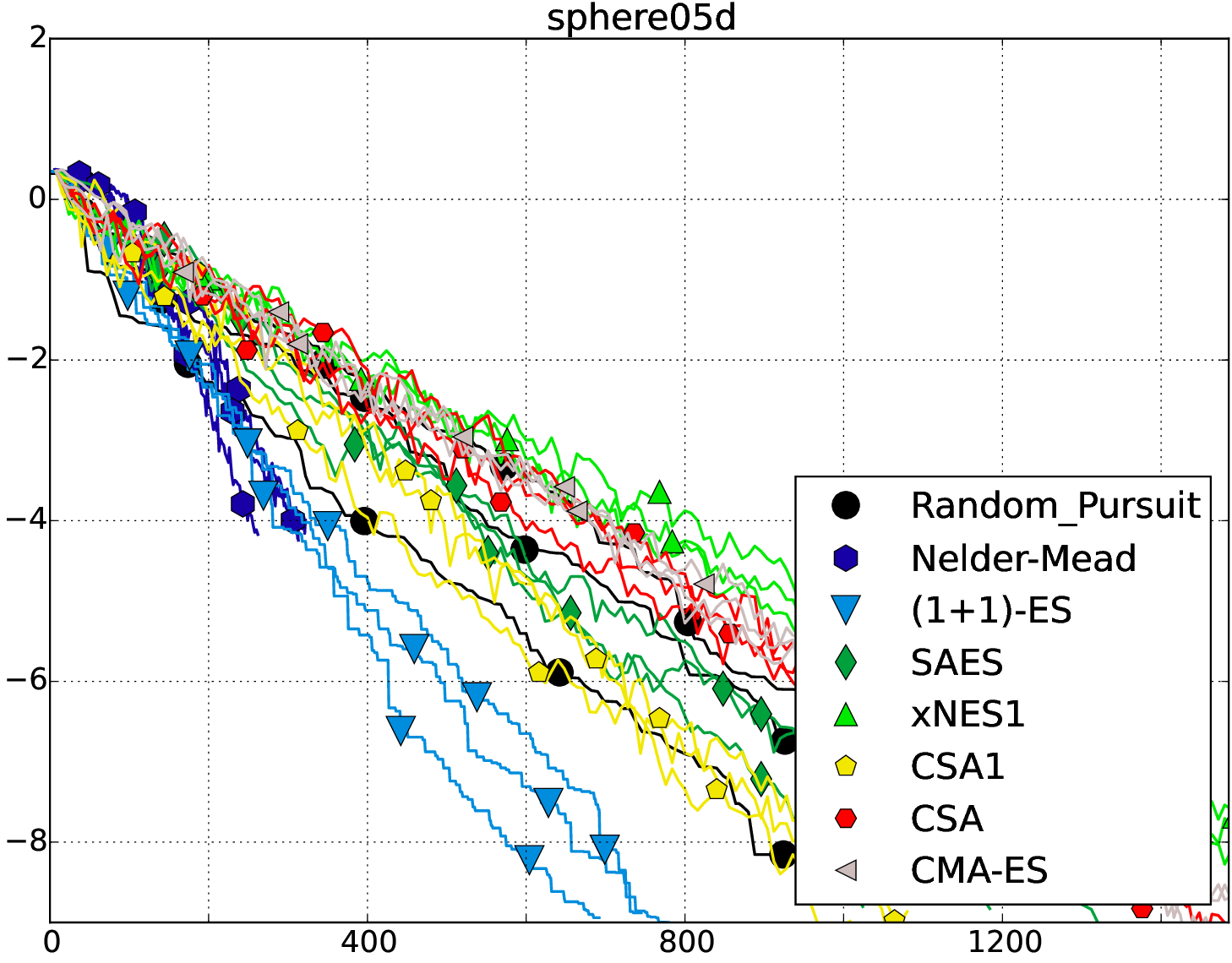}
\includegraphics[width=0.49\textwidth]{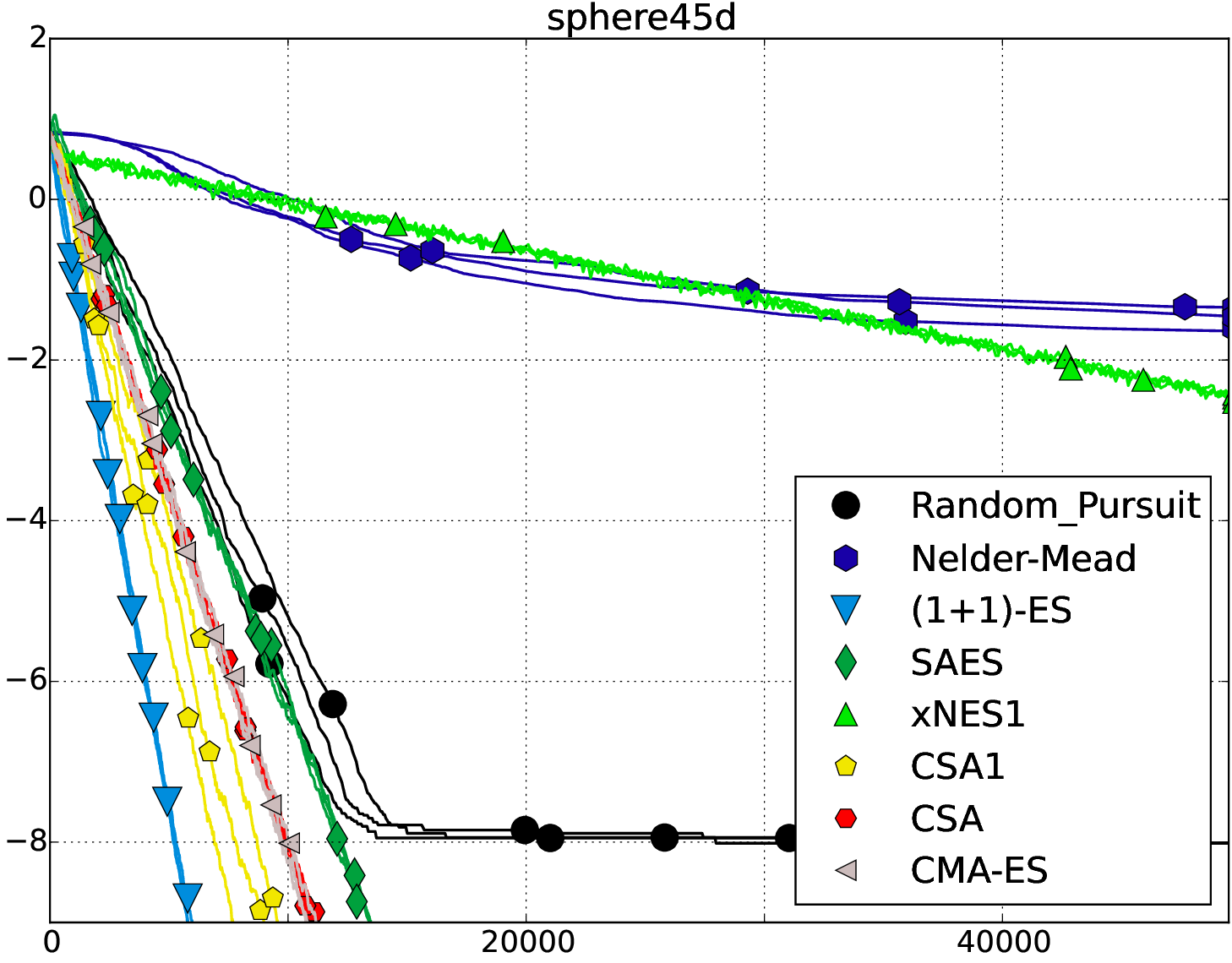}
\includegraphics[width=0.49\textwidth]{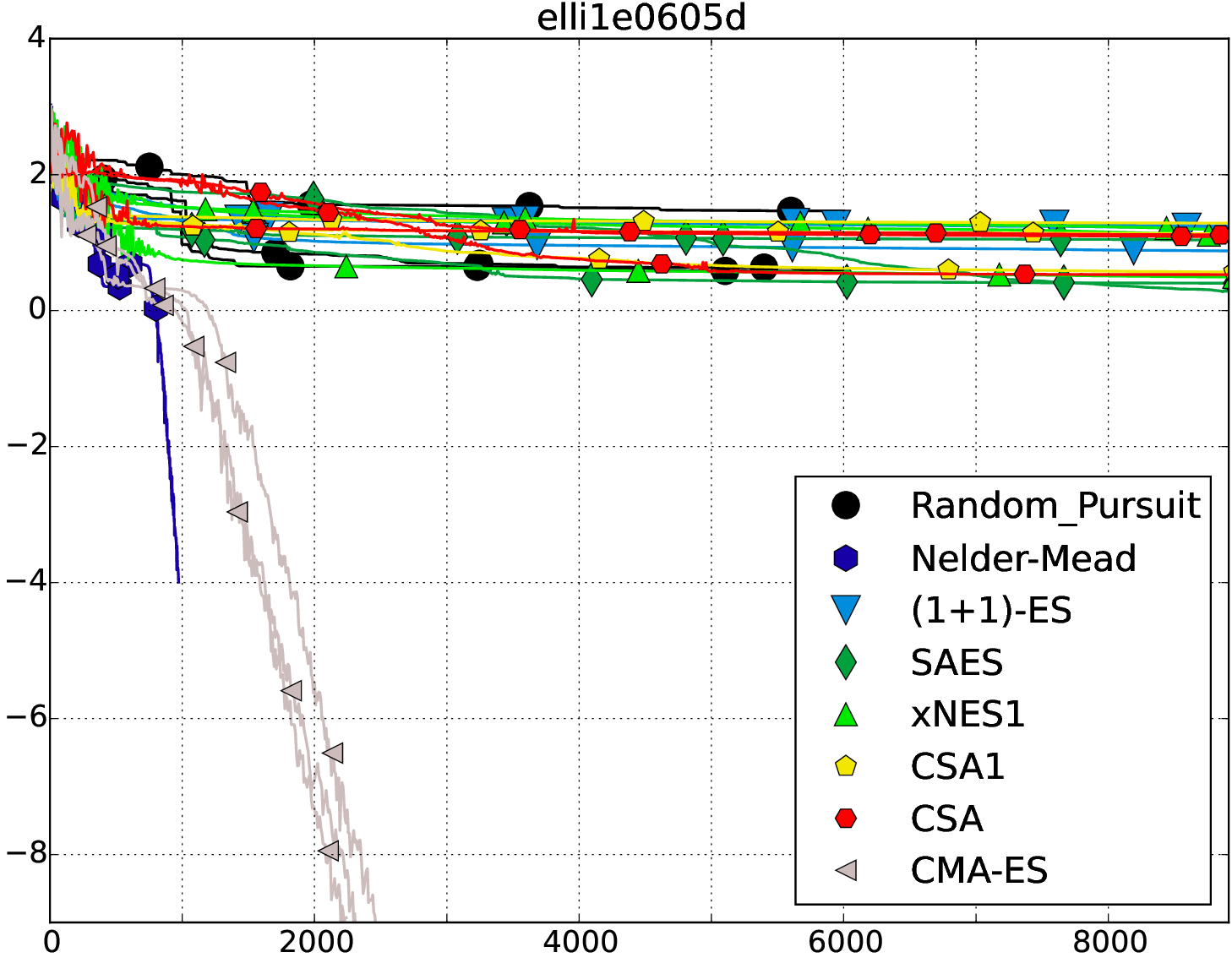}
\includegraphics[width=0.49\textwidth]{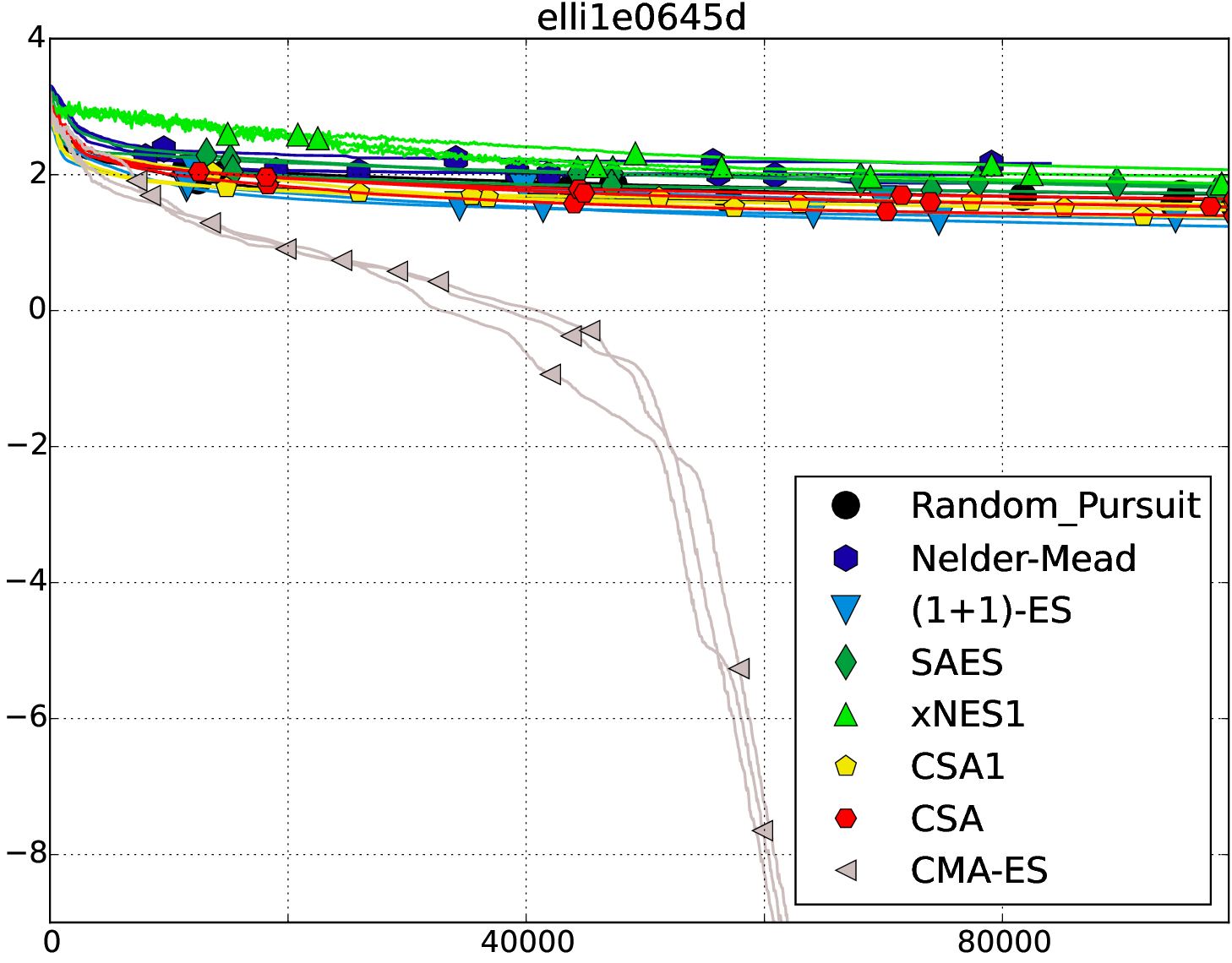}
\includegraphics[width=0.49\textwidth]{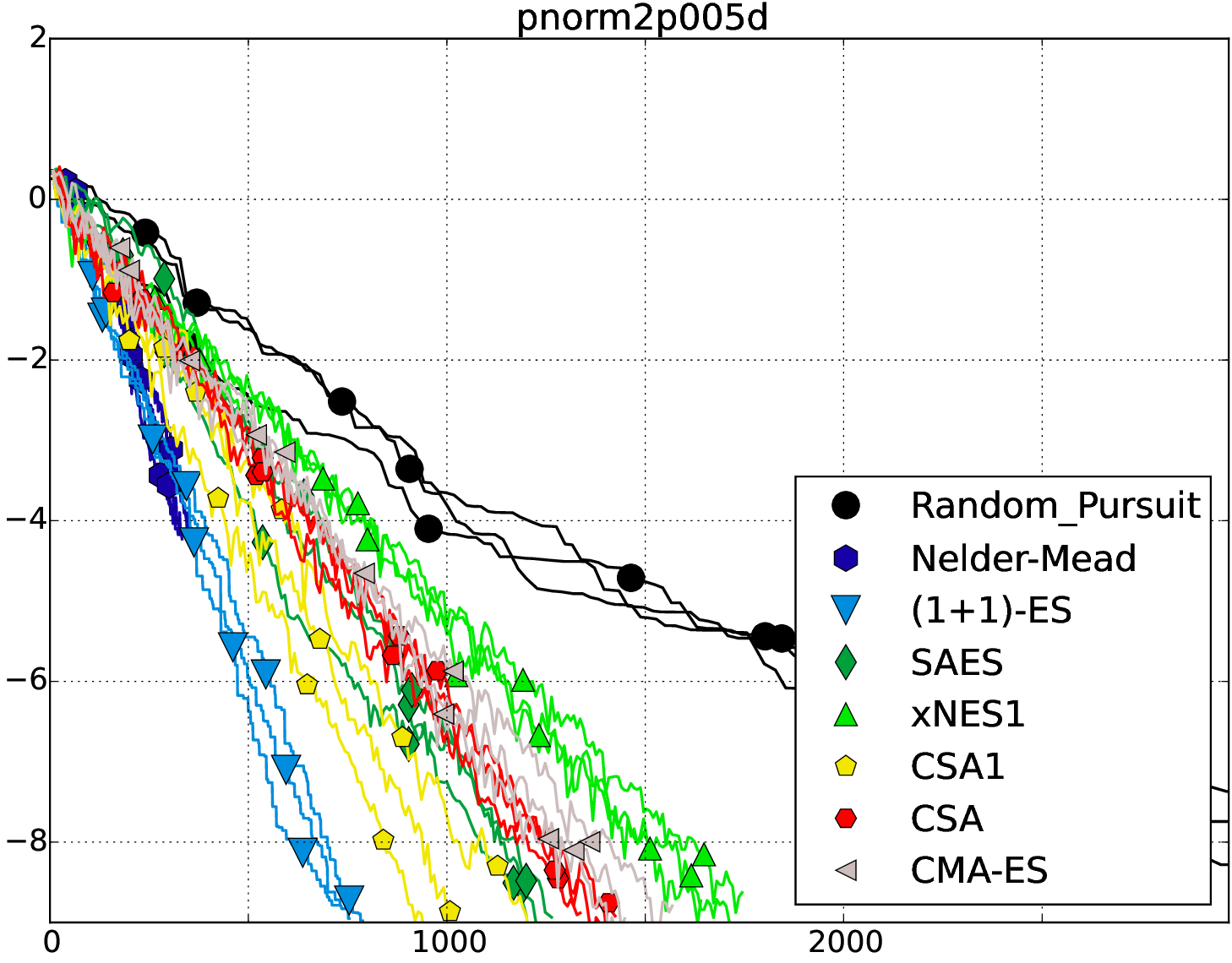}
\includegraphics[width=0.49\textwidth]{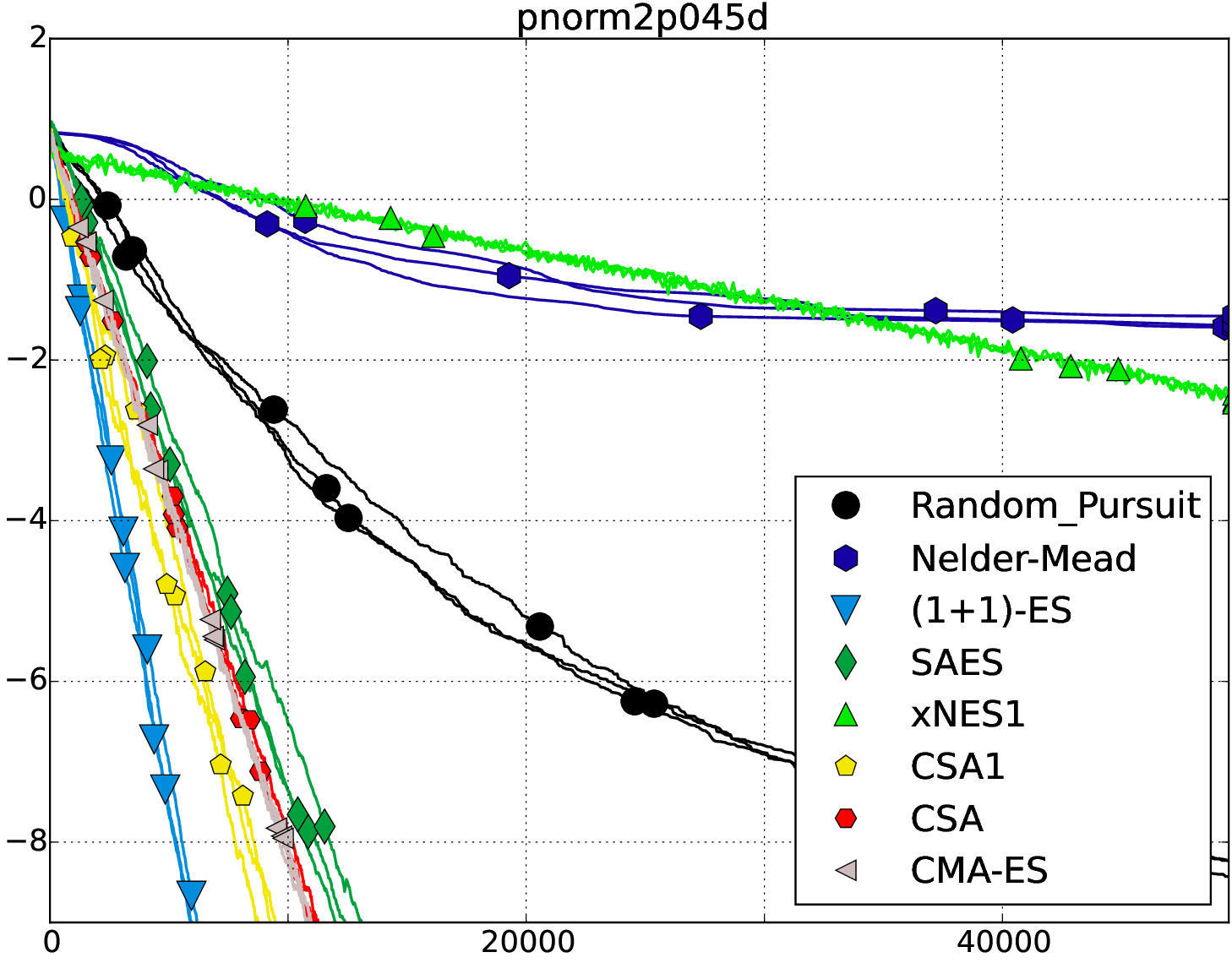}
\includegraphics[width=0.49\textwidth]{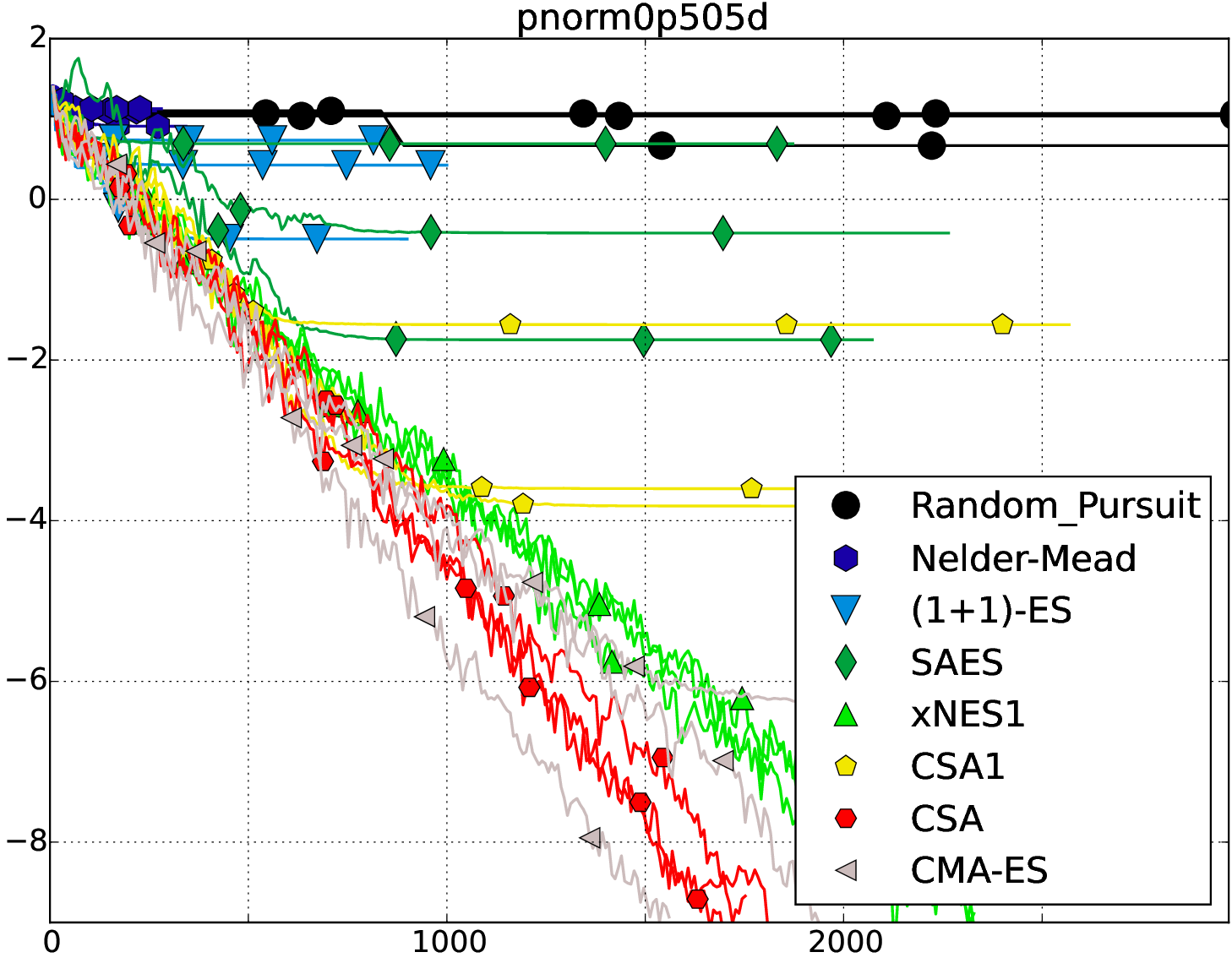}
\includegraphics[width=0.49\textwidth]{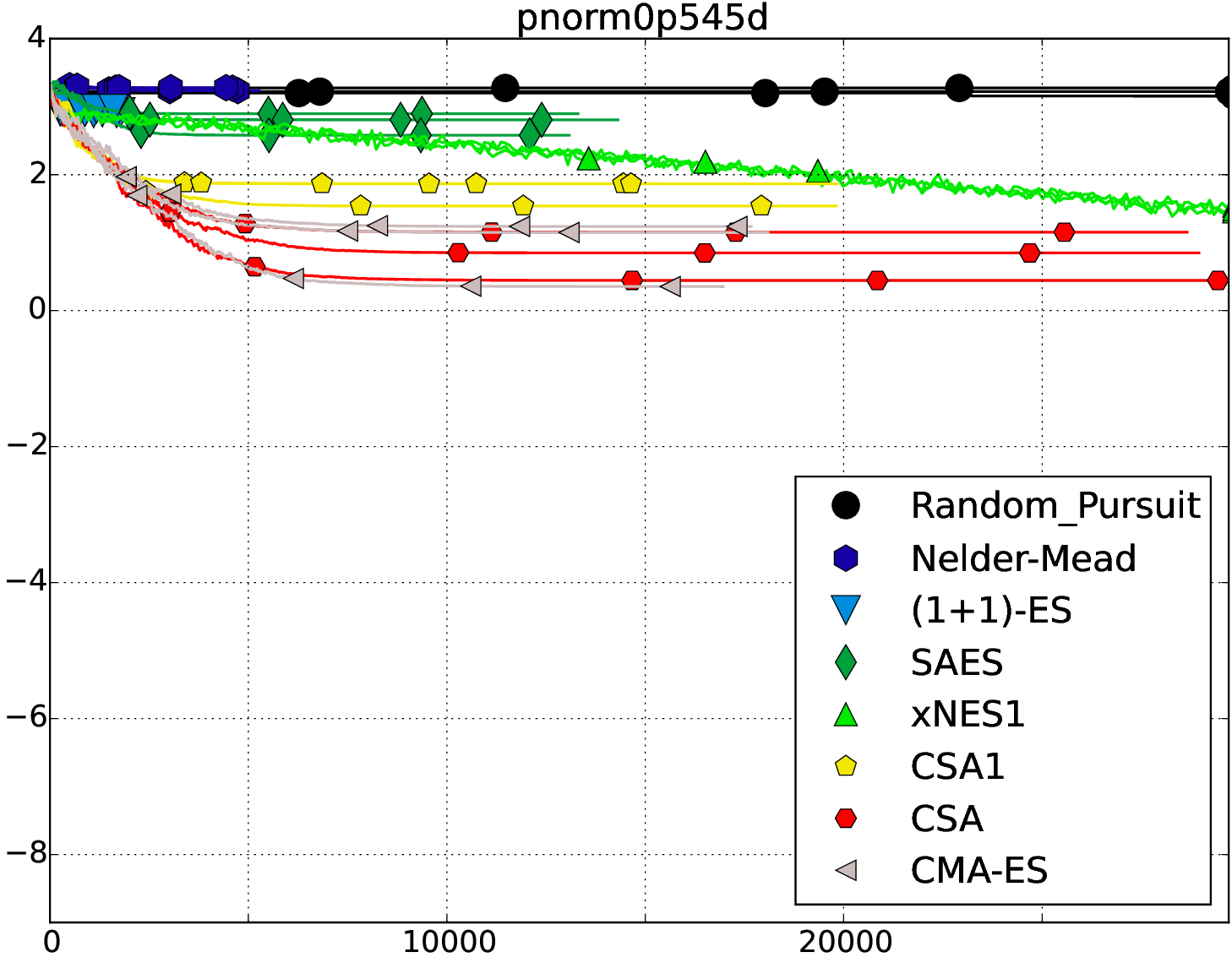}
\caption{\label{fig:all}
Single runs on $f_{\rm sphere}$, $f_{\rm elli}$, $f_{pnorm}$ for $p=2$ and $f_{\rm pnorm}$ for $p=1/2$ from top to bottom in dimension $5$ (left) and $45$ (right). The $x$-axis displays the number of function evaluations. A log scale is used for the $y$-axis \new{as in}\del{in a similar manner as for} Figure~\ref{fig:single-runs}. Subfigures in the upper two rows display the square root of the objective function, while subfigures in the lower rows
 display the objective function.
}
\end{figure}
In Figure~\ref{fig:all} we display in a single plot three runs of each algorithms on a given function in dimensions $5$ (left) and $45$ (right). The square root of the objective function is displayed for $f_{\rm sphere}$ and $f_{\rm elli}$ while the objective function is displayed for the $f_{\rm pnorm}$ functions.
\begin{itemize}
\item We observe a \new{strong}\del{large} impact of the dimension for the Nelder-Mead algorithm: while the algorithm is the fastest algorithm on the first three functions in dimension $5$, the algorithm does not work anymore in dimension $45$. This result is in agreement with previous observations that the Nelder-Mead algorithm does not work well for large dimension \cite{hansen2010comparing}.
\item The comparison of the graphs for $f_{\rm sphere}$ and $f_{\rm pnorm}$ \new{for $p=2$} illustrates again the invariance to monotonic transformation due to the rank-based property of all algorithms but RP.
\item On the $f_{\rm elli}$, we \new{see}\del{can envision} in direct comparison the large impact of having a covariance matrix adaptation mechanism compared to only step-size adaptation.
\item We observe that CMA-ES, CSA-ES and xNES1 are able to solve the $f_{\rm pnorm}$ function for $p=1/2$ in small dimension, however not necessarily in each single run. Here, the probability to succeed decisively depends on the chosen initial conditions. 
\end{itemize}

\end{appendix}

\bibliographystyle{plain}
\bibliography{optimbib}

\end{document}